\documentclass[reqno,
12pt]{amsart}

\usepackage{
amssymb,
amsmath,
amsthm,
eucal,
dsfont,
multicol,
mathrsfs,
tikz,
graphicx,
hyperref
}
\usepackage{lipsum}
\makeatletter\renewcommand*{\eqref}[1]{%
\hyperref[{#1}]{\textup{\tagform@{\!\!\ref*{#1}}}}%
}\makeatother 

\makeatletter
 
  \@addtoreset{equation}{section}
 \makeatother
\theoremstyle{plain}
\newtheorem{theorem}{Theorem}[section]
\newtheorem{lemma}[theorem]{Lemma}
\newtheorem{proposition}[theorem]{Proposition}

\theoremstyle{definition}
\newtheorem{definition}[theorem]{Definition}
\newtheorem{remark}[theorem]{Remark}

\newcommand{\norm}[1]{{\|#1\|}}

\def\supp{\mathop{\mathrm{supp}}\nolimits}

\def\Id{\mathop{\mathrm{Id}}\nolimits}

\def\ac{\mathop{\mathrm{ac}}\nolimits}

\def\loc{\mathop{\mathrm{loc}}\nolimits}

\def\sgn{\mathop{\mathrm{sgn}}\nolimits}
\def\BMO{{\mathop{\mathrm{BMO}}}}
\def\R{{\mathbb{R}}}
\def\Z{{\mathbb{Z}}}
\def\N{{\mathbb{N}}}
\def\C{{\mathbb{C}}}

\def\S{{\mathcal{S}}}

\def\H{{\mathcal{H}}}

\def\<{{\langle}}
\def\>{{\rangle}}

\def\ep{{\varepsilon}}
\def\ds{\displaystyle}
\DeclareMathOperator*{\slim}{s-lim}

\title[$L^p$-boundedness of wave operators]{$L^p$-boundedness of wave operators for bi-Schr\"odinger operators on the line}

\author{Haruya Mizutani}
\address[H. Mizutani]{Department of Mathematics, Graduate School of Science, Osaka University, Toyonaka, Osaka 560-0043, Japan}
\email{haruya@math.sci.osaka-u.ac.jp}
\author{Zijun Wan}
\address[Z. Wan]{Department of Mathematics, Central China Normal University, Wuhan, 430079, P.R. China}
\email{zijunwan@mails.ccnu.edu.cn}
\author{Xiaohua Yao\textsuperscript{$\dag$}}
\address[X. Yao]{Department of Mathematics and Key Laboratory of Nonlinear Analysis and Applications(Ministry of Education), Central China Normal University, Wuhan, 430079, P.R. China}
\email{yaoxiaohua@ccnu.edu.cn}
\begin{document}
\date{\today}
\thanks{ \textsuperscript{$\dag$}Corresponding author}

\begin{abstract}
This paper is devoted to establishing several types of $L^p$-boundedness of wave operators $W_\pm=W_\pm(H, \Delta^2)$ associated with the bi-Schr\"odinger operators  $H=\Delta^{2}+V(x)$ on the line $\mathbb{R}$.  Given suitable decay potentials $V$, we firstly prove that the wave and dual wave operators are bounded on $L^p(\mathbb{R})$ for all $1<p<\infty$:
$$
\|W_\pm f\|_{L^p(\mathbb{R})}+\|W_\pm^* f\|_{L^p(\mathbb{R})}\lesssim \|f\|_{L^p(\mathbb{R})},
$$
 which are further extended to the $L^p$-boundedness on the weighted spaces $L^p(\mathbb{R},w)$ with general even $A_p$-weights $w$ and to the boundedness on the Sobolev spaces $W^{s,p}(\mathbb{R})$. For the limiting case, we prove  that $W_\pm$ are bounded from $L^1(\R)$ to $L^{1,\infty}(\R)$ as well as bounded from the Hardy space $\H^1(\R)$ to $L^1(\R)$. These results especially hold whatever the zero energy is a regular point or a resonance of $H$. We also obtain that $W_\pm$ are bounded from $L^\infty(\R)$ to $\BMO(\R)$ if zero is a regular point or a first kind resonance of $H$.  Next, we show that $W_\pm$ are neither bounded on $L^1(\mathbb{R})$ nor on $L^\infty(\mathbb{R})$ even if zero is a regular point of $H$. Moreover, if zero is a second kind resonance of $H$, then $W_\pm$ are shown to be even not bounded from $L^\infty(\R)$ to $\BMO(\R)$ in general. In particular, we remark that our results give a complete picture of the validity of $L^p$-boundedness of the wave operators for all $1\le p\le \infty$ in the regular case. Finally, as applications, we deduce the $L^p$-$L^q$ decay estimates for the propagator $e^{-itH}P_{\mathrm{ac}}(H)$ with pairs $(1/p,1/q)$ belonging to a certain region of $\mathbb{R}^2$, as well as establish the H\"ormander-type $L^p$-boundedness theorem for the spectral multiplier $f(H)$.
\end{abstract}

\maketitle

\section{Introduction and main results}
\subsection{Introduction}
Let $\Delta^2=\frac{d^4}{dx^4}$ be the bi-Laplacian and $H=\Delta^2+V(x)$ be the (fourth-order) bi-Schr\"odinger operator on $\mathbb R$, where $V(x)$ is a real-valued potential satisfying
$$
|V(x)|\lesssim \<x\>^{-\mu}
$$
with some $\mu>0$ specified later and $\<x\>=\sqrt{1+x^2}$. By the Kato--Rellich theorem, $\Delta^2$ and $H$ are realized as self-adjoint operators on $L^2(\mathbb R)$ with domain $D(\Delta^2)=D(H)=H^4(\R)$, generating the associated unitary groups $e^{-it\Delta^2}$ and $e^{-itH}$ on $L^2(\R)$, respectively, where $H^4(\R)$ is the $L^2$-Sobolev space of order $4$.

For $\mu>1$, it is well-known  (see {\it e.g.} \cite{Agmon,Kuroda,ReSi}) that the {\it wave operators}
\begin{align}\label{def-wave}
W_\pm=W_\pm(H,\Delta^2) :=\slim_{t\to\pm\infty}e^{itH}e^{-it\Delta^2}
\end{align}
exist as partial isometries from $L^2(\mathbb{R})$ to $\mathcal H_{\ac}(H)$ and are asymptotically complete, {\it i.e.} $\mathop{\mathrm{Ran}} (W_\pm) =\mathcal H_{\ac}(H)$, where $\mathcal H_{\ac}(H)$ is the absolutely continuous spectral subspace of $H$. Moreover, the absolutely continuous spectrum $\sigma_{\ac}(H)$ coincides with $[0, \infty)$ and the singular continuous spectrum $\sigma_{\mathrm{sc}}(H)$ is absent. In particular, the {\it inverse (dual) wave operators}
$$
W_\pm(\Delta^2,H):=\slim_{t\to\pm\infty}e^{it\Delta^2}e^{-itH}P_{\mathrm{ac}}(H)
$$ also exist and satisfy $W_\pm(\Delta^2,H)=W_\pm(H,\Delta^2)^*$, where $P_{\ac}(H)$ is the projection onto $\mathcal H_{\ac}(H)$. The point spectrum $\sigma_{\mathrm{p}}(H)$ consists of finitely many negative eigenvalues and possible embedded eigenvalues in $[0,\infty)$. Throughout the paper, we always assume that $H$ has no embedded eigenvalue in $(0,\infty)$ (see Subsection \ref{subsection_remark} below for some sufficient conditions to ensure the absence of embedded eigenvalues of $H$).

$W_\pm$ and $W_\pm^*$ are clearly bounded on $L^2(\R)$. Then the main purpose of this paper is the following $L^p$-bounds of $W_\pm$ and $W_\pm^*$ for $p\neq2$:
\begin{align}
\label{Lp-bound}\|W_\pm \phi\|_{L^p(\mathbb{R})}\lesssim \|\phi\|_{L^p(\mathbb{R})},\quad \|W_\pm^* \phi\|_{L^p(\mathbb{R})}\lesssim \|\phi\|_{L^p(\mathbb{R})}.
\end{align}
To explain the importance of these bounds, we recall that $W_\pm$ satisfy the following identities
\begin{align*}
W_\pm W_\pm^* =P_{\ac}(H),\ \ W_\pm^*W_\pm=I,
\end{align*}
and {\it the intertwining property} $f(H)W_\pm=W_\pm f(\Delta^2)$,
where $f$ is any Borel measurable function on $\mathbb{R}$. These formulas especially imply
\begin{align}
\label{intertwining_1}
f(H)P_{\ac}(H)=W_\pm f(\Delta^2)W_\pm^*,
\end{align}
 which we also call the intertwining property.  By virtue of \eqref{intertwining_1}, the $L^p$-boundedness of $W_\pm, W_\pm^*$ can immediately be used to reduce the $L^p$-$L^q$ estimates for the perturbed operator $f(H)$ to the same estimates for the free operator $f(\Delta^2)$ as follows:
\begin{align}
\label{Lp-bound of f(H)}
\|f(H)P_{\ac}(H)\|_{L^p\to L^q}\le \|W_\pm\|_{L^q\to L^q}\ \|f(\Delta^2)\|_{L^p\to L^q}\ \|W_\pm^*\|_{L^{p}\to L^{p}}.
\end{align}
For many cases, under suitable conditions on $f$, it is accessible to establish the $L^p$-$L^q$ bounds of  $f(\Delta^2)$ by Fourier multiplier methods. Thus, in order to obtain the inequality \eqref{Lp-bound of f(H)},  it is a key problem to prove the $L^p$-bounds \eqref{Lp-bound} of $W_\pm$ and $W_\pm^*$. Note that this observation applies to not only the $L^p$-$L^q$ bounds, but also general $X$-$Y$ bounds, namely one has
\begin{align}
\label{Lp-bound of f(H)2}
\|f(H)P_{\ac}(H)\|_{X\to Y}\le \|W_\pm\|_{Y\to Y}\ \|f(\Delta^2)\|_{X\to Y}\ \|W_\pm^*\|_{X\to X}.
\end{align}

Because of such a useful feature, the $L^p$-boundedness of the wave operators has been extensively studied for the Schr\"odinger operator $-\Delta+V(x)$  on $\R^n$ and widely recognized as a fundamental tool for studying various nonlinear dispersive equations, such as the nonlinear Schr\"odinger and Klein--Gordon equations with potentials (see {\it e.g.} \cite{Cuccagna, Deng_Soffer_Yao, Schlag_1, Schlag_2, Soffer_Wu}). Therefore, it is natural and seems to be very important to try extending the $L^p$-boundedness of the wave operators to more general Hamiltonians, especially to the higher-order elliptic operator $P(D)+V(x)$ which has attracted increasing attention in the mathematical and mathematical physics literatures. Since the fourth-order Schr\"odinger operator $\Delta^2+V(x)$ can be considered as one of primal models of such higher-order operators, it thus is natural to ask whether the $L^p$-boundedness \eqref{Lp-bound}  for $W_\pm$ and $W_\pm^*$ holds or not.
For the higher-order Schr\"odinger operator $(-\Delta)^m+V(x)$ on $\R^n$ with $m\in \N$ and $m>1$,  there were significant progress  made   in recent years by Goldberg--Green \cite{GoGr21},  Erdo\u{g}an--Green \cite{Erdogan-Green21,Erdogan-Green23}, Erdo\u{g}an--Goldberg--Green \cite{EGG23} and Galtbayar--Yajima \cite{Galtbayar_Yajima_ArXiv23} (see also our recent works \cite{MWY_ArXiv23_2,MWY_ArXiv23_1}). Nevertheless, there are still many problems not addressed in the literature compared with  Schr\"odinger operator $-\Delta+V(x)$. In particular, there seems to be no results  in low dimensions $n=1,2$ for the higher-order case $m>1$. We refer to Subsection \ref{subsection_back_ground} below for more elaborations and existing literature.

In light of those observations, the main purpose of the paper is to show that the wave operators $W_\pm$ and $W_\pm^*$ for $H=\Delta^2+V(x)$ on $\R$ are bounded on $L^p(\R)$ for all $1<p<\infty$, whatever zero is a regular point or a resonance of $H$  (see Definition \ref{definition_resonance} below). Moreover, we also establish several related interesting results in both positive and negative directions, complementing to or improving upon this result, which specifically include:
\begin{itemize}
\item Several weak-boundedness in the limiting cases $p=1,\infty$;
\item Weighted $L^p$-boundedness for any even Muckenhoupt weights $w\in A_p$ and $1<p<\infty$ without assuming any additional condition on $V$;
\item $W^{s,p}$-boundedness, where $W^{s,p}$ is the $L^p$-Sobolev space of order $s$;
\item Counterexamples of the $L^1$- and $L^\infty$-boundedness.
\end{itemize}
These results particularly give a complete classification for the validity of $L^p$-boundedness of $W_\pm, W_\pm^*$ if $H$ has no non-negative eigenvalue nor zero resonance. Furthermore, we apply our main theorem to show the $L^p$-$L^q$ decay estimates for the propagator $e^{-itH}P_{\ac}(H)$ and the   H\"ormander-type theorem of the $L^p$-boundedness for the spectral multiplier $f(H)$.

\subsection{Main results}
\label{subsection_main_results}
To state our results, we need to recall the notion of the {\it zero resonances} for the operator $H=\Delta^2+V(x)$ on $\R$ due to Soffer--Wu--Yao \cite{SWY21}.  For $s\in \R$, we set $L^2_s(\R)=\{f\in L^2_{\mathrm{loc}}(\mathbb R)\ |\ \<x\>^sf\in L^2(\mathbb R)\}$, which is decreasing in $s$. Then we define
$$
W_\sigma(\mathbb R)=\bigcap_{s>\sigma}L^2_{-s}(\mathbb R),
$$
which is increasing in $\sigma$ and satisfies $L^2_{-\sigma}(\mathbb R)\subset W_\sigma(\mathbb R)$. Note that $(1+|x|)^\alpha\in W_\sigma(\R)$ if $\sigma\ge \alpha+1/2$. In particular, $f\in W_{1/2}(\R)$ and $\<x\>f\in W_{3/2}(\R)$ for any $f\in L^\infty(\R)$.

\begin{definition}
\label{definition_resonance} Let $H=\Delta^2+V(x)$ and $|V(x)|\lesssim \<x\>^{-\mu}$ for some $\mu>0$.
We say that
\begin{itemize}
\item zero is a {\it first kind resonance  of $H$} if there exists some nonzero $\phi\in W_{3/2}(\R)$ but no non-zero $\phi\in W_{1/2}(\R)$ such that $H\phi=0$ in the distributional sense;
\item zero is a {\it second kind resonance of $H$} if there exists some nonzero $\phi\in W_{1/2}(\R)$ but no non-zero $\phi\in L^2(\R)$ such that $H\phi=0$ in the distributional sense;
\item  zero is an {\it eigenvalue of $H$} if there exists  some nonzero $\phi\in L^2(\R)$ such that $H\phi=0$ in the distributional sense;
\item  zero is a {\it regular point of $H$} if $H$ has neither zero eigenvalue nor zero resonances.
\end{itemize}
\end{definition}
The case when zero is a regular point of $H$ is also called the {\it generic case} and the case when zero is a resonance or an eigenvalue of $H$ is called the {\it exceptional case} in the literature.

\begin{remark}
\label{remark_eigenvalue}
It was observed by Goldberg \cite{Goldberg} (see also \cite[Remark 1.2]{SWY21}) that if $|V(x)|\lesssim \<x\>^{-\mu}$ with some $\mu$ satisfying a weaker condition than \eqref{mu}, then $H$ has no zero eigenvalue. Hence in the following theorems of this paper, we do not need to consider the zero eigenvalue case (see also Subsection \ref{subsection_remark} below for more related comments).
\end{remark}

Let $\mathbb B(X,Y)$ be the space of bounded operators from $X$ to $Y$, namely $A\in \mathbb B(X,Y)$ if
$$
\|Af\|_{Y}\lesssim \|f\|_X,\quad f\in X.
$$
We also set $\mathbb B(X)=\mathbb B(X,X)$. We now state the main result of this paper as follows.

\begin{theorem}
\label{theorem_1}
Let $H=\Delta^2+V(x)$ and $V$ satisfy $|V(x)|\lesssim \<x\>^{-\mu}$ for some $\mu>0$ depending on the following types:
\begin{align}
\label{mu}
\mu>\begin{cases}
15&\text{if zero is a regular point of $H$},\\
21&\text{if zero is a first kind resonance of $H$},\\
29&\text{if zero is a  second kind resonance of $H$}.
\end{cases}
\end{align}
Assume also $H$ has no embedded eigenvalue in $(0,\infty)$. Let $W_\pm, W_\pm^*$ be the wave and inverse (or dual) wave operators defined by \eqref{def-wave}. Then the following statements hold:
\begin{itemize}
\item[(1)] $W_\pm,W_\pm^*\in \mathbb B(L^p(\R))$ for all $1<p<\infty$. Moreover, if $V$ is compactly supported, then $W_\pm,W_\pm^*\in \mathbb B(L^1(\R),L^{1,\infty}(\R))$.
  \vskip0.1cm
\item[(2)] $W_\pm \in \mathbb B(\mathcal H^1(\R),L^1(\R))$ and $W_\pm^*\in \mathbb B(L^\infty(\R),\BMO(\R))$. Moreover, if in addition zero is either a regular point or a first kind resonance of $H$, then $W_\pm\in \mathbb B(L^\infty(\R),\BMO(\R))$ and $W_\pm^*\in \mathbb B(\mathcal H^1(\R),L^1(\R))$.
  \vskip0.1cm
\item[(3)] Let $1<p<\infty$, $w_p\in A_p$ and set $\tau f(x)=f(-x)$. Then
\begin{align*}
\norm{W_\pm f}_{L^p(w_p)}+\norm{W_\pm^*f}_{L^p(w_p)}&\lesssim \norm{f}_{L^p(w_p)}+\norm{\tau f}_{L^p(w_p)}.
\end{align*}
In particular, $W_\pm, W_\pm^*\in \mathbb B(L^p(w_p))$ if $w_p$ is even. Moreover, if zero is a regular point of $H$ and the operator $Q_1A_1^0Q_1$ appeared in Lemma \ref{lemma_3_3} below is finite rank, then $W_\pm, W_\pm^*\in \mathbb B(L^1(w_1),L^{1,\infty}(w_1))$ for any even $w_1\in A_1$.
\end{itemize}
\end{theorem}

Here $A_p$ is the Muckenhoupt class (see Appendix \ref{appendix_CZ} below for more details and some examples), $L^{p}(w)$, $L^{1,\infty}(w)$, $\mathcal H^1(\R)$ and $\BMO(\R)$ are the weighted $L^p$, weighted weak $L^1$, Hardy and Bounded Mean Oscillation spaces on $\R$, respectively (see Subsection \ref{subsection_notation} below). 

\begin{remark}
\label{remark_theorem_1} We here make a few remarks  (see Subsection \ref{subsection_remark} for more remarks).

\begin{itemize}
\item[(1)] In Theorem \ref{theorem_1}, the presence of zero resonances has no effect on the $p$-range of the $L^p$-boundedness of wave operators $W_\pm, W_\pm^*$, and only require that the potentials $V$ satisfy stronger decay conditions than the regular case.
  \vskip0.1cm
\item[(2)] We in fact prove the following bounds with an explicit dependence on the weights:
\begin{align}
\label{theorem_1_1}
\norm{W_\pm f}_{L^p(w_p)}&\lesssim [w_p]_{A_p}^{\max\{1,1/(p-1)\}}(\norm{f}_{L^p(w_p)}+\norm{\tau f}_{L^p(w_p)}),\quad 1<p<\infty,\\
\label{theorem_1_2}
\norm{W_\pm f}_{L^{1,\infty}(w_1)}&\lesssim [w_1]_{A_1}(1+\log [w]_{A_1})(\norm{f}_{L^1(w_1)}+\norm{\tau f}_{L^1(w_1)}),
\end{align}
where $[w]_{A_p}$ is the $A_p$-characteristic constant of $w$ (see Appendix \ref{appendix_CZ}) and the implicit constants are independent of $w_p,w_1$. Moreover, the same bounds also hold for $W_\pm^*$. The estimates of type \eqref{theorem_1_1} (without $\norm{\tau f}_{L^p(w_p)}$) are known as the $A_p$-estimates in the theory of Calder\'on--Zygmund operators and known to be sharp (see \cite{Hytonen}). We also refer to \cite{LOP} for the estimates of type \eqref{theorem_1_2} for Calder\'on--Zygmund operators.
  \vskip0.1cm
\item[(3)] For the Schr\"odinger operator $-\Delta+V(x)$ on $\R^3$, Beceanu \cite{Be} proved a weighted $L^p$-boundedness of the wave operators with a specific weight $\<x\>^a$ for $|a|<1$ under a suitable assumption on $V$ depending on $a$. Compared with his result, the interesting point of Theorem \ref{theorem_1} (2) is that we can take general even  ({\it i.e.} radial) weight $w_p\in A_p$. Moreover, our assumption on $V$ is independent of the choice of weights.
\end{itemize}
 \end{remark}

In Theorem \ref{theorem_1}, we have obtained the desired $L^p$ (or even weighted $L^p$) boundedness of $W_\pm$ for non-endpoint cases $1<p<\infty$ and some weak-boundedness for the limiting cases $p=1,\infty$. Then it is natural to ask whether $W_\pm$ are bounded on $L^1(\R)$ and $L^\infty(\R)$ or not. The next theorem answers this question negatively in the regular case, which shows that Theorem \ref{theorem_1} is sharp (in general) in terms of the $p$-range of the $L^p$-boundedness.

\begin{theorem}
\label{theorem_2}
Suppose that $|V(x)|\lesssim \<x\>^{-\mu}$ with $\mu>15$, $V\not\equiv0$ and that $H$ has no embedded eigenvalue in $(0,\infty)$. Then we have the following statements:
\begin{itemize}
\item[(1)] Suppose that zero is a regular point of $H$. Then $W_\pm, W_\pm^*\notin \mathbb B(L^1(\R))\cup \mathbb B(L^\infty(\R))$.
  \vskip0.1cm
\item[(2)] Suppose that zero is a second kind resonance of $H$ and $V$ is compactly supported. If $D_*\neq0$, then  $W_\pm \notin \mathbb B(L^\infty(\R),\BMO(\R))$ and $W_\pm^*\notin \mathbb B(\H^1(\R),L^1(\R))$, where the constant $D_*$ is defined in Proposition \ref{proposition_example_3}.
\end{itemize}
\end{theorem}

\begin{remark}
One can also obtain some results on the unboundedness in $L^1$ and $L^\infty$ for the resonant cases. We refer to Remark \ref{remark_counterexample_1} in Section \ref{section_counterexample} for more details.
\end{remark}

Finally, we also obtain the $W^{s,p}$-boundedness of $W_\pm$, where $W^{s,p}=W^{s,p}(\R)$ is the $L^p$-Sobolev space of order $s$. For $N\in \N$, we set
\begin{align}
\label{B^N}
B^N(\R)=\{V\in C^N(\R)\ |\ V^{(k)}\in L^\infty(\R)\ \text{for all $k=0,1,...,N$}\}.
\end{align}

\begin{theorem}
\label{theorem_3}
Let $1<p<\infty$ and $H=\Delta^2+V(x)$ satisfy the same assumption in Theorem \ref{theorem_1}. Then $W_\pm, W_\pm^*\in \mathbb B(W^{s,p}(\R))$ for all $0\le s\le4$. Moreover, if in addition $V\in B^{4N}(\R)$ with some $N\in \N$, then $W_\pm, W_\pm^*\in \mathbb B(W^{s,p}(\R))$ for all $0\le s\le4(N+1)$.
\end{theorem}

Here we summarize the above results in the following Table 1, from which it is clear that, for the case when zero is a regular point of $H$, our results give a complete classification of the validity of the $L^p$-boundedness for all $1\le p\le \infty$ and weak-boundedness in the framework of $L^{1,\infty}$, $\H^1$ and $\BMO$ for the limiting cases $p=1,\infty$.

\begin{table}[htb]
\centering
  \begin{tabular}{|c||c||c|}  \hline
  Boundedness & $W_\pm(H,\Delta^2)$ & $W_\pm(H,\Delta^2)^*$ \\ \hline
  $L^p(\R)$, $L^p(w_p)$, $W^{s,p}(\R)$\ \  ($1<p<\infty$) & True & True  \\ \cline{1-1} \cline{1-3}
  $L^1(\R)\to L^{1,\infty}(\R)$ & True & True  \\ \cline{1-1} \cline{1-3}
  $L^1(\R)$, $L^\infty(\R)$ & False (R)& False (R)  \\ \cline{1-1} \cline{1-3}
  $\H^1(\R)\to L^1(\R)$ & True & True (R, 1st) \\ \hline
  $L^\infty(\R)\to \BMO(\R)$ & True (R, 1st)  & True \\ \hline
  \end{tabular}\\
  (R=regular case,\ 1st=first kind case)
  \vspace{0.5cm}
  \caption{Boundedness of $W_\pm(H,\Delta^2)$ and $W_\pm(H,\Delta^2)^*$}
\end{table}

\subsection{Further remarks on eigenvalues and potentials}
\label{subsection_remark}
Here we make further comments on the above theorems, especially the spectral assumptions and the decay condition on $V$.
\subsubsection{Zero resonance and zero eigenvalue} We first give two simple examples of $V$ such that $H$ has a zero resonance. On one hand, zero is a second kind resonance for the free case $H=\Delta^2,V\equiv0$ since any constant function $\phi_0\in W_{1/2}(\R)$ satisfies $\Delta^2\phi_0=0$. On the other hand, it is also easy to construct $V\not\equiv0$ such that $H$ has a zero resonance. Indeed, let $\phi_1\in C^\infty(\mathbb{R})$ be a positive function such that  $\phi_1(x)=c|x|+d$ for $|x|>1$ with some constants $c,d\ge0$ satisfying $(c,d)\neq(0,0)$.  Then $H\phi_1=0$ if taking $$ V(x)=-(\Delta^2\phi_1)/\phi_1, \ \ x\in \mathbb{R}.$$  Note that $V\in C_0^\infty(\mathbb{R})$ and $\phi_1\in W_{3/2}(\R)\setminus W_{1/2}(\R)$ if $c>0$ and $\phi_1\in W_{1/2}(\R)$ if $c=0$. These examples indicate that zero resonances may occur even for compactly supported potentials.

We next discuss on the zero eigenvalue of $H$. It is again easy to construct an example of $H$ having zero eigenvalue if $V$ decays sufficiently slowly.
In fact, let $\phi=(1+|x |^2)^{-s/2}$ and $V(x)=-(\Delta^2\phi)/\phi$. Then $\phi\in H^4(\mathbb{R})$ for any $s>1$ and $(\Delta^2+V)\phi=0$, which means $|V(x)|\lesssim \<x\>^{-4}$ and zero is an eigenvalue of $H $. However, as already mentioned in Remark \ref{remark_eigenvalue}, if $|V(x)|\lesssim \<x\>^{-\mu}$ with some $\mu$ satisfying \eqref{mu}, then zero cannot be an eigenvalue of $H$ in dimension one. We believe such a decay condition on $V$ may not be sharp, expecting that the decay rate $\mu>4$ is optimal to ensure the absence of zero eigenvalue for $\Delta^2+V$ on $\R$.

Based on these remarks, and in view of the the fast decay conditions of potential $V$ in our theorems, we remark that zero eigenvalue can be actually excluded, while zero resonances must be taken into account. However, we again emphasize that the presence of zero resonances has no effect on the validity of $L^p$-boundedness of $W_\pm, W_\pm^*$ at least for $1<p<\infty$.
\subsubsection{Embedded positive eigenvalue} In contrast with the zero energy case, the absence of positive eigenvalues of $H$ are more subtle than that of zero resonance or zero eigenvalue.

It is well-known as Kato's theorem \cite{K} that if $V$ is bounded and $V=o(|x|^{-1})$ as $|x|\rightarrow\infty$ then the Schr\"odinger operator $-\Delta+V$ has no positive eigenvalues  (also see \cite{FHHH, IJ, KoTa} for more related results and references). By contraries, such a criterion cannot hold for the fourth-order Schr\"odinger operator  $H=\Delta^2+V$, so the assumption on the absence of positive eigenvalues seems to be indispensable. Indeed, it is easy to construct a Schwartz function $V(x)$ so that $H$ on $\R$ has an eigenvalue $1$.\footnote{In fact, $V(x)=20/\cosh^2(x)-24/\cosh^4(x)\in \S(\R)$ satisfies $\frac{d^4\psi_0}{dx^4}+ V(x)\psi_0=\psi_0$ where $\psi_0=1/\cosh(x)=2/(e^x+e^{-x})\in L^2(\mathbb{R})$.} Moreover, in any dimensions $n\ge1$, one can also easily construct $V\in C_0^\infty(\R^n)$ so that $H$ has positive eigenvalues (see Feng et al.\cite[Section 7.1]{FSWY}). These results clearly indicate that the absence of positive eigenvalues for the fourth-order Schr\"odinger operator would be more  subtle and unstable than the second order cases under the potential  perturbation $V$. 

We however stress that if $V\in C^1(\R)\cap L^\infty(\R)$ is repulsive, {\it i.e.}, $xV'(x)\le0$, then $H$ has no eigenvalues (see \cite[Theorem 1.11]{FSWY}). Note that such a criterion also works for the general higher-order elliptic operator $P(D)+V$ in any dimensions $n\ge1$.
Besides, we also notice that for a general selfadjoint operator $\mathcal{H}$ on $L^{2}(\mathbb R^n)$, even if $\mathcal{H}$ has a simple embedded eigenvalue, Costin--Soffer in \cite{CoSo} have proved that $\mathcal{H}+\ep W$ can kick off the eigenvalue located in a small interval under certain small perturbation of the potential $\ep W$.

\subsubsection{Decay condition on the potential}
The rather fast decay condition \eqref{mu} on the potential $V$ in our theorems is due to the use of low energy expansions of the resolvent $(H-\lambda^4-i0)^{-1}$ obtained by Soffer--Wu--Yao \cite{SWY21} (see Lemma \ref{lemma_3_3} below). In fact,  in the regular case for instance, the proof of Theorem \ref{theorem_1} works well if $|V(x)|\lesssim \<x\>^{-\mu}$ for $\mu>9$ under the assumption that the expansion \eqref{lemma_3_3_1} holds. Although it is an interesting problem to improve the assumption of Lemma \ref{lemma_3_3}, we do not pursue it for the sake of simplicity.

Note that in the case of the Schr\"odinger operator $-\Delta+V(x)$, the Jost functions are known to be very useful tools for studying asymptotic behaviors of the resolvent (see \cite{DeTr}) and have been widely used in the proof of $L^p$-boundedness of wave operators (see \cite{ArYa, DaFa, Weder}). However, it is not clear whether  the same method can be also applied to the fourth-order case. Indeed, in view of the explicit formula of the free resolvents $(\Delta^2-\lambda^4\mp i0)^{-1}$ (see \eqref{free_resolvent} below), we must construct four Jost functions $f_\pm(\lambda,x),g_\pm(\lambda,x)$ such that$$
f_\pm(\lambda,x)\sim e^{\pm i\lambda x},\quad g_\pm(\lambda,x)\sim e^{\mp\lambda x},\quad x\to \pm\infty,
$$
Hence the situation is very different from the second-order case since $g_\pm(\lambda,x)$ can grow exponentially fast if $\lambda<0$, while the Jost functions are uniformly bounded in the second-order case. Note that one needs several global estimates of Jost functions or their Fourier transforms with respect to $\lambda,x\in \R$ in the proof of $L^p$-bounds for the wave operators (see {\it e.g.} \cite[Section 2]{Weder}). For readers interested in the construction of $f_\pm,g_\pm$, we refer to \cite{Hill} where the potential $V$ has been assumed to be compactly supported.

\subsection{Two types of applications}
\label{subsection_application}
By virtue of \eqref{Lp-bound of f(H)}, or more generally \eqref{Lp-bound of f(H)2}, our main estimates may have a lot of potential applications. We however do not pursue to list them as many as possible, but focus on a few primal applications which will be important for further applications to nonlinear equations. More precisely, we prove the following two types of results (see Section \ref{section_application} for the precise statements):
\begin{itemize}
\item $L^p$-$L^q$ decay estimates for the propagator $e^{-itH}P_{\ac}(H)$:
$$
\norm{e^{-itH}P_{\ac}(H)\phi}_{L^q(\R)}\lesssim |t|^{-\frac14(\frac1p-\frac1q)}\norm{\phi}_{L^p(\R)},\quad t\neq0,
$$
for $(1/p,1/q)$ belonging to a region of $\mathbb{R}^2$ (see Figure 1 in Section \ref{section_application}).
\item $L^p$-boundedness of the spectral multiplier $f(H)$:
$$
\norm{f(H)\phi}_{L^p(\R)}\lesssim \norm{\phi}_{L^p(\R)},\quad 1<p<\infty,
$$
where $f\in L^\infty(\R)$ satisfies the standard {\it H\"ormander condition}  (see \eqref{Hormander_conditions}).
\end{itemize}
These $L^p$-$L^q$ decay estimates for $e^{-itH}P_{\ac}(H)$ generalize the $L^1$-$L^\infty$ decay estimate obtained recently by Soffer--Wu--Yao \cite{SWY21}. On the other hand, the new interesting point of this spectral multiplier theorem for $f(H)$ is that our operator $H=\Delta^2+V(x)$ may have negative eigenvalues as well as zero resonances, so $e^{-tH}$ possibly has no sharp (generalized) Gaussian kernel bounds. Hence a standard criterion based on Gaussian kernel bounds (see {\it e.g.} Sikora--Yan--Yao \cite{SYY}) cannot be applied in the present case.

Furthermore, we notice that in the case of the Schr\"odinger operator $-\Delta+V(x)$, the $L^p$-boundedness of wave operators, as well as the $L^p$-$L^{p'}$ decay estimates for $e^{it(\Delta-V)}$ and the spectral multiplier theorem for $f(-\Delta+V)$ are very important tools for studying associated dispersive equations such as the nonlinear Schr\"odinger equations with potentials  (see {\it e.g.} \cite{Cuccagna, Deng_Soffer_Yao, Schlag_1, Schlag_2, Soffer_Wu} and reference therein). Hence, we believe that Theorems \ref{theorem_1} and \ref{theorem_3}, as well as these two results on $e^{-itH}$ and $f(H)$,  will be fundamental tools for studying several nonlinear dispersive equations associated with $H$, especially for the following fourth-order nonlinear Schr\"odinger equation with a potential:
$$
i\partial_t u-\partial_x^4u-V(x)u=\mu |u|^{p-1}u,\quad t,x\in \R.
$$

\subsection{More related  backgrounds}
\label{subsection_back_ground}
In this subsection, we record some known results on the $L^p$-boundedness of the wave operators, comparing them with our theorems. We also discuss some related results, as well as some backgrounds on the higher-order elliptic operators.

For the Schr\"odinger operator $-\Delta+V(x)$ on $\R^n$ in any dimensions $n\ge1$, there exists a great number of works are devoted to establish the $L^p$-boundedness of the wave operators in last almost thirty years. For instance, Yajima in the seminar work \cite{Yajima-JMSJ-95} proved the $L^p$-boundedness of wave operators for $n\ge 3$ in the regular case. Subsequently, the case $n=1$ were studied by Weder \cite{Weder} and Artbazar--Yajima \cite{ArYa} independently and the case  $n=2$ by Yajima \cite{Yajima-CMP-99}. Since then later, many further progresses and applications related to the $L^p$-boundedness of wave operators have been made for all the regular, zero resonance and zero eigenvalue cases under various conditions on the potential $V$ (see \cite{Be, BeSc_JST, BeSc, CMY1, CMY2, Cuccagna2, DaFa, DMSY, DMW, EGG, Finco_Yajima_II, Goldberg-Green-Advance, Goldberg-Green-Poincare, Jensen, Jensen_Yajima_2D, Jensen_Yajima_4D, Soffer_Wu, Weder_arxiv, Yajima_2006,Yajima_2016,Yajima_2018, Yajima_2021AHP, Yajima_2021arxiv} and references therein). Certainly, these  works would play indispensable roles in the studies of higher-order elliptic operators.

The weighted boundedness considerably less is known  compared with the unweighted one. As already mentioned in Remark \ref{remark_theorem_1} (3), Beceanu \cite{Be} obtained some weighted $L^p$-boundedness with polynomial weights $\<x\>^a$. Note that Beceanu--Schlag \cite{BeSc_JST,BeSc} proved (again for the Schr\"odinger operator) $W_\pm \in \mathbb B(X)$ if $X$ is any Banach space of measurable functions on $\R^3$ such that the norm $\norm{\cdot}_X$ is invariant under reflections and translations and that
$
\norm{\chi_H f}_X\le A \norm{f}_{X}
$ for any half space $H\subset \R^3$ with some uniform constant $A$. This result clearly implies the $L^p$-boundedness (even the $W^{s,p}$-boundedness) of $W_\pm$, but not the weighted $L^p$-boundedness since weighted $L^p$-norms are not invariant under translations.

Next we explain known results for the Schr\"odinger operator on $\R$ more precisely. Weder \cite{Weder} proved $W_\pm\in \mathbb B(L^p(\R))$ for $1<p<\infty$ if $\<x\>^{\gamma}V\in L^1(\R)$ with some $\gamma>3/2$ in the regular case and $\gamma>5/2$ in the zero resonant case. Artbazar--Yajima \cite{ArYa} also proved independently a similar result under a slightly stronger decay condition on $V$. Later, the assumption on $V$ has been weakened to $\<x\>V\in L^1(\R)$ in the regular case and $\<x\>^2V\in L^1(\R)$ in the zero resonant case by D'Ancona--Fanelli \cite{DaFa}, and finally to $\<x\>V\in L^1(\R)$ in the zero resonant case by Weder \cite{Weder_arxiv}. It was also shown by \cite{Weder} that $W_\pm\in \mathbb B(W^{k,p}(\R))\cap \mathbb B(L^1(\R),L^{1,\infty}(\R))\cap \mathbb B(\H^1(\R),L^{1}(\R))$ for general cases and that $W_\pm\in \mathbb B(L^1(\R))\cap \mathbb B(L^\infty(\R))$ if zero is a resonance and the scattering matrix at $\lambda=0$ is the identity matrix. It was also mentioned in \cite{Weder} that $W_\pm$ are neither bounded on $L^1(\R)$ nor on $L^\infty(\R)$ in general. The case with a delta potential $V=a\delta$ was studied by Duch\^ene--Marzuola--Weinstein \cite{DMW} and Weder \cite{Weder_arxiv}. Weder \cite{Weder_arxiv} also studied  the case with matrix Schr\"odinger operators on the line or the half line. Note that, in all these papers \cite{ArYa, DaFa, DMW, Weder, Weder_arxiv}, the proofs heavily rely on the Jost functions and their properties studied by Deift--Trubowitz \cite{DeTr}.

Now we shall consider the higher-order Schr\"odinger operator $(-\Delta)^m+V(x)$ on $\R^n$ with $m\in \N$ and $m>1$ and sufficiently fast decaying potential $V(x)$ for which great progresses have been made in recent years.  The first result in this direction is due to Goldberg--Green \cite{GoGr21} for the case $(m,n)=(2,3)$, where the $L^p$-boundedness of wave operators was proved for  $1<p<\infty$ if the zero energy is a regular point. For $n>2m\ge4$, Erdo\u{g}an--Green \cite{Erdogan-Green21,Erdogan-Green23} proved the $L^p$-boundedness for all $1\le p\le \infty$ if the zero energy is a regular point and the potential $V(x)$ is sufficiently smooth.  Furthermore, for the case $n>4m-1$, Erdo\u{g}an--Goldberg--Green \cite{EGG23} provides examples of compactly supported non-smooth potential $V(x)$ for which the wave operators are not bounded on $L^p$ if $2n/(n-4m+1)<p\le \infty$. More recently, the case $n=2m=4$ was considered by Galtbayar--Yajima \cite{Galtbayar_Yajima_ArXiv23} where the $L^p$-boundedness was proved for $1<p<p_0$ with suitable $p_0$ depending on the type of the singularity at the zero energy. It can be observed from these works that the behavior of wave operators are roughly classified into three cases: $n<2m$, $n=2m$ and $n>2m$. When $n<2m$, as observed by \cite{GoGr21}, the resolvent has a singularity at the zero energy even in the free case and singular integrals similar to  Hilbert transform are appeared in the stationary representation of the low energy part of wave operators even in the regular case. It thus can be expect that the wave operators are generically not bounded on $L^p$ for $p=1,\infty$ in this case. On the other hand, when $n>2m$, the singularity at the zero energy of the resolvent is relatively mild, but the high energy part becomes much more complicated than the case $n<2m$ since the resolvent does not decay (or even can grow in higher space dimensions) in the high energy limit. The case $n=2m$ is critical in the sense that it has these difficulties in the low and high energy parts of the wave operators together.

Compared with these existing works, the interest of our results in this paper is that we  provide not only the $L^p$-boundedness for all $1<p<\infty$, but also counterexamples of the $L^p$-boundedness at the endpoint $p=1,\infty$, as well as some weak-boundedness in the framework of $L^{1,\infty}$, $\H^1$ and $\BMO$. Moreover, we study all cases of the types of the singularity at the zero energy which has not been carried out at least in the case $n<2m$. Finally, the weighted $L^p$-boundedness with general even $A_p$-weight, as well as the explicit bounds \eqref{theorem_1_1} and \eqref{theorem_1_2}, seems to be totally new (see also Remark \ref{remark_theorem_1} (3)).

Finally  we should  mention that there is a huge  literature on the study of higher-order elliptic operators $P(D)+V(x)$ in many topics in mathematics and mathematical physics. In addition to the aforementioned works \cite{GoGr21,Erdogan-Green21,EGG23,Erdogan-Green23}, we refer the readers to {\it e.g.} \cite{Agmon, H2, Kuroda} for the spectral and scattering theory, \cite{DDY_JFA, SYY} for Harmonic analysis and \cite{DY, Erdogan-Green-Toprak, FSWY, FSY18, FWY, Green-Toprak,LSY, MiYa, SWY21, CLSY_ArXiv, EGG_ArXiv, CHHZ_ArXiv} for various dispersive properties such as time decay, local energy decay, Strichartz estimates for $e^{-itH}$, and the asymptotic expansion and uniform resolvent estimates for $(H-z)^{-1}$.


\subsection{The outline of the proof}
Here we briefly explain the ideas of the proof of the above theorems. For simplicity, we consider the case when zero is a regular point of $H$ only.

The starting point is the following stationary formula:
$$
W_-=\Id-\frac{2}{\pi i}\int_0^\infty \lambda^3 R_V^+(\lambda^4)V\left(R_0^+(\lambda^4)-R_0^-(\lambda^4)\right)d\lambda,
$$
where $R_0^\pm(\lambda^4)=(\Delta^2-\lambda^4\mp i0)^{-1}$ and $R_V^\pm(\lambda^4)=(H-\lambda^4\mp i0)^{-1}$ are the boundary values of the free  and perturbed resolvents. The integral kernels of $R_0^\pm(\lambda^4)$ are explicitly given by
\begin{align}
\label{idea_1}
R_0^\pm(\lambda^4,x,y)
=\frac{F_\pm(\lambda |x-y|)}{4\lambda^3}=\frac{F_\pm(\lambda |x|)}{4\lambda^3}-\frac{y}{4\lambda^2}\int_0^1\sgn(x-\theta y)F_\pm'(\lambda|x-\theta y|)d\theta,
\end{align}
where $F_\pm(s):=\pm ie^{\pm is}-e^{-s}$ and we have used the Taylor expansion near $y=0$ in the second line. In particular, $R_0^\pm(\lambda^4)=O(\lambda^{-3})$ at the level of the order of $\lambda$.

Decompose $W_--\Id$ into the low energy $\{0\le \lambda\ll1\}$ and the high energy $\{\lambda\gtrsim1\}$ parts. The high energy part is easier to treat than the low energy part since the free resolvent does not have singularity for $\lambda\ge1$, so we here consider the following low energy part only:
\begin{align}
\label{idea_2}
W_-^L:=\int_0^\infty \lambda^3 \chi(\lambda) R_V^+(\lambda^4)V\left(R_0^+(\lambda^4)-R_0^-(\lambda^4)\right)d\lambda,
\end{align}
where $\chi\in C_0^\infty(\R)$ such that $\chi\equiv1$ near $\lambda=0$. Setting $v(x)=\sqrt{|V(x)|}$, $U(x)=\sgn V(x)$ and $M(\lambda)=U+vR_0^+(\lambda^4)v$, one has the standard symmetric second resolvent equation:
$$
R_V^+(\lambda^4)V=R_0^+(\lambda^4)v M^{-1}(\lambda)v.
$$
Then one of key tools in our argument is the asymptotic expansion of $M^{-1}(\lambda)$ as $\lambda\to +0$ obtained recently  by \cite{SWY21} which, in the regular case, is of the form
\begin{align}
\nonumber
M^{-1}(\lambda)&=Q_{2}A_{0}^0Q_{2}+\lambda Q_1A_{1}^0Q_1+\lambda^2\left(Q_1A_{21}^0Q_1+Q_{2}A_{22}^0+A_{23}^0Q_{2}\right)\\
\label{idea_3}
&\quad +\lambda^3\left(Q_1A_{31}^0+A_{32}^0Q_1\right)+\lambda^3\widetilde P+\Gamma^0_4(\lambda),
\end{align}
where $A_{k}^0,A_{kj}^0,\widetilde P,Q_\alpha\in \mathbb B(L^2)$, $\norm{\Gamma^0_4(\lambda)}_{L^2\to L^2}=O(\lambda^4)$ and $Q_\alpha$ satisfies
\begin{align}
\label{idea_4}
Q_\alpha(x^kv)\equiv0,\quad \<x^kv,Q_\alpha f\>=0,
\end{align}
for any $f\in L^2$ and any integer $0\le k\le \alpha-1$.
The interest of these properties \eqref{idea_4} is that, combined with the Taylor expansion formula \eqref{idea_1}, one has
\begin{align}
\label{idea_5}
Q_\alpha vR_0^\pm(\lambda^4)=O(\lambda^{-3+\alpha}),\quad R_0^\pm(\lambda^4)vQ_\alpha=O(\lambda^{-3+\alpha}),
\end{align}
which are less singular in $\lambda\in (0,1]$ compared with the free resolvent $R_0^\pm(\lambda^4)=O(\lambda^{-3})$.
\vskip0.2cm
$\bullet$\ {\it\underline{On the $L^p(\R)$-boundedness.}} Substituting \eqref{idea_3} into \eqref{idea_2} one can find that $W_-^L$ is a sum of nine integral operators with integral kernels of the form
\begin{align}
\label{idea_6}
\int_0^\infty \lambda^{\ell-\alpha-\beta}\chi(\lambda)\left(R_0^+(\lambda^4)vQ_\alpha BQ_\beta v[R_0^+-R_0^-](\lambda^4)\right)(x,y)d\lambda,
\end{align}
where $B\in \mathbb B(L^2)$ varies from line to line, $\ell=6$ or $7$ and we set $Q_0=\Id$. Note that the integrand is of order $\lambda^{\ell-6}$ by \eqref{idea_5}. Then such nine integral operators
are classified into the following two classes with respect to the order of $\lambda$ of the integrands of their integral kernels:
\begin{itemize}
\item[(I)] $O(\lambda)$: $Q_\alpha BQ_\beta=Q_{2}A_{0}^0Q_{2},Q_1A_{21}^0Q_1,Q_{2}A_{22}^0,A_{23}^0Q_{2},Q_1A_{31}^0,A_{32}^0Q_1$ and $\lambda^{-4}\Gamma_4^0(\lambda)$;
     \vskip0.2cm
\item[(II)] $O(1)$: $Q_\alpha BQ_\beta=Q_1A_{1}^0Q_1$ and $\widetilde P$.
\end{itemize}

The operators in the class (I) can be shown to be bounded on $L^p(\R)$ for any $1\le p\le \infty$. We shall explain this for $Q_\alpha BQ_\beta=Q_1A_{21}^0Q_1$ as a model case. In such a case, by using \eqref{idea_1}, \eqref{idea_5} and the identity
$$
F_+'(\lambda|x|)[F_+'-F_-'](\lambda|y|)=e^{i\lambda(|x|+|y|)}-e^{i\lambda(|x|-|y|)}+e^{-\lambda(|x|+i|y|)}-e^{-\lambda(|x|-i|y|)},
$$
we can rewrite \eqref{idea_6} as a linear combination of following four functions:
\begin{align*}
K^\pm_{a_{1}}(x,y)&=\int_0^\infty  e^{i\lambda(|x|\pm|y|)}\lambda \chi(\lambda)a_{1}^\pm(\lambda,x,y)d\lambda,\\
K^\pm_{a_{2}}(x,y)&=\int_0^\infty  e^{-\lambda(|x|\pm i|y|)}\lambda \chi(\lambda)a_{2}^\pm(\lambda,x,y)d\lambda,
\end{align*}
where $a^\pm_j$ satisfy
$$
|\partial_\lambda^\ell a_{1}^\pm(\lambda,x,y)|+e^{-\lambda|x|}|\partial_\lambda^\ell a_{2}^\pm(\lambda,x,y)|\lesssim \norm{\<x\>^{4+2\ell}V}_{L^1},\quad x,y\in \R,\ \lambda\ge0,\ \ell=0,1,2.
$$
Then we apply integration by parts twice to $K_{a_j}^\pm$, obtaining
$$
|K^\pm_{a_{j}}(x,y)|\lesssim \<|x|\pm|y|\>^{-2},\quad x,y\in \R,
$$
where note that $K^\pm_{a_{j}}\in L^\infty(\R^2)$ since $\chi\in C_0^\infty(\R)$ and the term $O(\<|x|\pm|y|\>^{-1})$ does not appear thanks to the fact $\lambda\chi(\lambda)|_{\lambda=0}=0$. Now the $L^p(\R)$-boundedness for any $1\le p\le \infty$ follows from standard Schur's lemma since $\<|x|\pm|y|\>^{-2}\in L^\infty_yL^1_x\cap L^\infty_xL^1_y$.

In the above argument, the crucial point is that we have an additional $\lambda$ in the integrands. For the operators in the class (II) which do not have such a factor $\lambda$, we need more precise estimates for the integral kernels to employ the theory of Calder\'on--Zygmund operators. As a model case, we shall consider the integral \eqref{idea_6} with $Q_\alpha BQ_\beta=\widetilde P$ (in which case $\ell=6$). In such a case, a similar argument as above yields that \eqref{idea_6} can be rewritten in the form
$$
-\sum_{\pm} \int_0^\infty  \left(e^{i\lambda(|x|\pm|y|)}\chi(\lambda)c_{1}^\pm(\lambda,x,y)+i e^{-\lambda(|x|\pm i|y|)}\chi(\lambda)c_{2}^\pm(\lambda,x,y)\right) d\lambda
$$
with some $c_j^\pm $ satisfying the same estimates as for $a_j^\pm$. Applying integration by parts twice, we find that this integral is a sum of the leading term $\frac{-1+i}{8}g_1^+(x,y)$, where
$$
g_1^+(x,y)=\frac{\psi_+}{|x|+|y|}+\frac{\psi_-}{|x|-|y|}+\frac{\psi_-}{|x|+i|y|}+\frac{\psi_-}{|x|-i|y|},
$$
and the error term $O(\<|x|-|y|\>^{-2})$ which can be dealt as above, where $\psi_\pm=\psi(||x|\pm|y||^2)$ are smooth cut-off functions supported in $\{(x,y)\ |\ ||x|\pm|y||\ge1\}$. Although the integral operator $T_{g_1^+}$ with the kernel $g_1^+$ itself is not a Calder\'on--Zygmund operator, using the identity
$$
g_1^+(x,y)=\left(\chi_{+}(x)+\chi_{-}(x)\right)g_1^+(x,y)\left(\chi_{+}(y)+\chi_{-}(y)\right)
$$
with $\chi_\pm$ being the indicator function of $\R_\pm$, one can write
\begin{align}
\label{idea_7}
T_{g_1^+}
=\left((\chi_{+}-\chi_{-})T_{\widetilde k_1}+
\chi_+T_{\widetilde k_2^+}\chi_+- \chi_-T_{\widetilde k_2^+}\chi_-+\chi_+T_{\widetilde k_2^-}\chi_+- \chi_-T_{\widetilde k_2^-}\chi_-\right)(1+\tau),
\end{align}
where $\tau:f(x)\mapsto f(-x)$, $\widetilde k_1(x,y)=\psi(|x-y|^2)(x-y)^{-1}$ and $\widetilde k_2(x,y)=\psi(|x-y|^2)(x\pm iy)^{-1}$ so that $T_{\widetilde k_1}$ and $T_{\widetilde k_2^\pm}$ can be shown to be Calder\'on--Zygmund operators. The abstract theorem for Calder\'on--Zygmund operators then shows $T_{g_1^+}\in \mathbb B(L^p(\R))\cap \mathbb B(L^1(\R),L^{1,\infty}(\R))$.
\vskip0.2cm
$\bullet$\ {\it\underline{On the weighted $L^p$-boundedness.}} We shall consider $T_{g^+_1}$ as a model case. The theory of Calder\'on--Zygmund operators shows $T_{\widetilde k_1},T_{\widetilde k_2^\pm}\in \mathbb B(L^p(w_p))$ for any $w_p\in A_p$. Moreover, recent deep results by \cite{Hytonen} for $1<p<\infty$ imply
\begin{align*}
\norm{T_{\widetilde k_1}}_{L^p(w_p)\to L^p(w_p)}+\norm{T_{\widetilde k_2^\pm}}_{L^p(w_p)\to L^p(w_p)}\lesssim [w_p]_{A_p}^{\max\{1,1/(p-1)\}},\quad 1<p<\infty.
\end{align*}
If $w_p$ and $w_1$ are even, then these bounds on $T_{\widetilde k_1},T_{\widetilde k_2^\pm}$ and \eqref{idea_7} yield desired weighted boundedness of $T_{g^+_1}$ with explicit operator norm bounds in terms of $[w_p]_{A_p}$.
\vskip0.2cm
$\bullet$\ {\it\underline{On the $\H^1$-$L^1$ and $L^\infty$-$\BMO$ boundedness.}}
Let us consider again the operator $T_{g^+_1}$. Since $\H^1$ is not invariant under the map $f\mapsto \chi_\pm f$ (recall that any $f\in \H^1$ satisfies $\int fdx=0$), the formula \eqref{idea_7} is not enough to prove $T_{g^+_1}\in \mathbb B(\H^1(\R),L^1(\R))\cap \mathbb B(L^\infty(\R),\BMO(\R))$, although $T_{\widetilde k_1},T_{\widetilde k_2^\pm}\in \mathbb B(\H^1(\R),L^1(\R))\cap \mathbb B(L^\infty(\R),\BMO(\R))$ by the abstract theory for Calder\'on--Zygmund operators. Instead, we prove $T_{g^+_1},T_{g^+_1}^*\in \mathbb B(\H^1(\R),L^1(\R))$ directly by following the classical proof of the $\H^1$-$L^1$ boundedness for Calder\'on--Zygmund operators based on the atomic decomposition of $\H^1$. By the duality, $(\H^1)^*=\BMO$, one also has $T_{g^+_1}\in \mathbb B(L^\infty(\R),\BMO(\R))$.
\vskip0.2cm
$\bullet$\ {\it\underline{Counterexamples of $L^1$ and $L^\infty$ boundedness.}} As seen above, all the operators in the class (I) are bounded on $L^p(\R)$ for all $1\le p\le \infty$. Let $T_{K^0_1}$ (resp. $T_{K^0_{33}}$) be the integral operator in the class (II) associated with $Q_1A_{1}^0Q_1$ (resp. $\widetilde P$). Both of them in fact can be shown to be not bounded on $L^1(\R)$ nor $L^\infty(\R)$. Although this is not sufficient to disprove such boundedness properties of $W_-$, if we take a test function $f_R=\chi_{[-R,R]}$ then one can prove that $\sup_{R>0}\norm{T_{K^0_1}f_R}_{L^\infty}<\infty$ and $T_{K^0_1}f_1\in L^1(\R)$, but $|(T_{K^0_{33}}f_R)(R+2)|\to \infty$ as $R\to \infty$ and $T_{K^0_{33}}f_1\notin L^1(\R)$. This shows $W_-\notin \mathbb B(L^1(\R))\cup \mathbb B(L^\infty(\R))$. 

\vskip0.2cm$\bullet$\ {\it\underline{On the $W^{s,p}$-boundedness.}} Once the $L^p(\R)$-boundedness of $W_\pm$ is obtained, its $W^{s,p}$-boundedness easily follows from the intertwining identity $(H+E)^{s/4}W_\pm=W_\pm(\Delta^2+E)^{s/4}$ and the inequality
$
\norm{(\Delta^2+E)^{s/4}f}_{L^p(\R)}\lesssim \norm{(H+E)^{s/4}f}_{L^p}
$ for sufficiently large $E>0$, which can be shown by a standard method (see {\it e.g.} \cite{Blunck})  based on the generalized Gaussian bound \eqref{lemma_6_2_proof_2} for the semi-group $e^{-t(H+E)}$ proved by \cite{DDY_JFA}.

\subsection{Organizations of the paper} The rest of the paper is devoted to the proof of Theorems \ref{theorem_1}, \ref{theorem_2} and \ref{theorem_3} and their applications, and is organized as follows.

In Section \ref{section_preliminaries}, we prepare several preliminary materials, which include the stationary formula of wave operators (Subsection \ref{subsection_stationary}), the asymptotic expansion at low energy of the resolvent and several useful formulas for the free resolvent (Subsection \ref{subsection_resolvent}).

In Section \ref{subsection_integral_operators} we prepare a few criterions to obtain several boundedness properties of integral operators appeared in the stationary formula of the wave operator $W_-$.

The proof of Theorem \ref{theorem_1} for the low energy part of $W_-$ is given by Section \ref{section_low}, while
the proof for high energy part is given by Section \ref{section_high}.

Section \ref{section_counterexample} is devoted to the proof of Theorem \ref{theorem_2}. Theorem \ref{theorem_3} is proved in Section \ref{section_Sobolev}. Section \ref{section_application} is concerned with the applications explained in Subsection \ref{subsection_application}.

Finally, we give a short review of Calder\'on--Zygmund operators in Appendix \ref{appendix_CZ}.

\subsection{Notations}
\label{subsection_notation}
Throughout the paper we use the following notations:
\begin{itemize}
\item $A\lesssim B$ (resp. $A\gtrsim B$) means $A\le CB$ (resp. $A\ge CB$) with some constant $C>0$.
\item $L^p=L^p(\R),L^{1,\infty}=L^{1,\infty}(\R)$ denote the Lebesgue and weak $L^1$ spaces, respectively.
\item $\<f,g\>=\int f\overline g$ denotes the inner product in $L^2$.
\item For $w\in L^1_{\loc}(\R)$ positive almost everywhere and $1\le p<\infty$, $L^p(w)=L^p(\R,wdx)$ denotes the weighted $L^p$-space with the norm
$$
\norm{f}_{L^p(w)}=\left(\int |f(x)|^pw(x)dx\right)^{1/p}.
$$
$L^{1,\infty}(w)$ denotes the weighted weak $L^1$ space with the quasi-norm
$$
\norm{f}_{L^{1,\infty}(w)}=\sup_{\lambda>0} \lambda w(\{x\ |\ |f(x)|>\lambda\}).
$$
\item $\mathop{\mathrm{BMO}}=\mathop{\mathrm{BMO}}(\R)$ is the Bounded Mean Oscillation space: $f\in \mathop{\mathrm{BMO}}$  if $f\in L^1_{\loc}(\R)$ and
$$
\norm{f}_{\BMO}:=\sup_{I}\frac{1}{|I|}\int_I|f-f_I|dx<\infty,
$$
where the supremum takes over all bounded intervals and $f_I=\frac{1}{|I|}\int_Ifdx$. Note that $L^\infty \subset \BMO$. For instance, $\log (a|x|+b)\in \BMO\setminus L^\infty$ for any positive $a,b$.
\item $\H^1=\H^1(\R)$ is the Hardy space: $f\in \H^1$ if $f$ is a tempered distribution and
$$
\norm{f}_{\H^1}=\int_\R\sup_{t>0}|(f*\varphi_t)(x)|dx<\infty
$$
with some Schwartz function $\varphi$ satisfying $\int_\R\varphi(x)dx=1$ and $\varphi_t(x)=t^{-1}\varphi(x/t)$. It is known that $(\H^1)^*=\BMO$ and
$
|\<f,g\>|\lesssim \norm{f}_{\H^1}\norm{g}_{\BMO}
$ (see \cite{Fefferman}).
\item $T_K$ denotes the integral operator with the kernel $K(x,y)$, namely
$$
T_Kf(x)=\int_\R K(x,y)f(y)dy.
$$
\item Let $\{\varphi_N\}_{N\in \Z}$ be a homogeneous dyadic partition of unity on $(0,\infty)$: $\varphi_0\in C_0^\infty(\R_+)$, $0\le \varphi\le1$, $\supp \varphi\subset[\frac14,1]$, $\varphi_N(\lambda)=\varphi_0(2^{-N}\lambda)$, $\supp \varphi_N\subset [2^{N-2},2^N]$  and
$$
\sum_{N\in \Z}\varphi_N(\lambda)=1,\quad\lambda>0.
$$
\end{itemize}

\section{Preliminaries}
\label{section_preliminaries}

\subsection{Stationary representation of wave operators}
\label{subsection_stationary}
First of all, we observe that it suffices to deal with $W_-$ only since \eqref{def-wave} implies $W_+f=\overline{W_-\overline f}$.

The starting point is the well-known stationary formula \eqref{stationary} of $W_-$. To state the formula, we need to introduce some notations. Let
\begin{align*}
R_0(z)=(\Delta^2-z)^{-1},\quad
R_V(z)=(H-z)^{-1},\quad z\in \C\setminus[0,\infty),
\end{align*}
be the resolvents of $\Delta^2$ and $H=\Delta^2+V(x)$, respectively. We denote by $R^\pm_0(\lambda),R^\pm_V(\lambda)$ their boundary values (limiting resolvents) on $(0,\infty)$, namely
$$
R^\pm_0(\lambda)=\lim_{\ep \searrow 0}R_0(\lambda\pm i\ep),\quad R_V^\pm(\lambda)=\lim_{\ep \searrow 0}R_V(\lambda\pm i\ep),\quad \lambda>0.
$$
The existence of $R^\pm_0(\lambda)$ as bounded operators from $L^2_s$ to $L^2_{-s}$ with $s>1/2$ follows from the limiting absorption principle for the resolvent $(-\partial_x^2-z)^{-1}$ of the free Schr\"odinger operator $-\partial_x^2$ (see e.g. \cite{Agmon}) and the formula
$$
R_0(z)=\frac{1}{2\sqrt z}\left[(-\partial_x^2-\sqrt z)^{-1}-(-\partial_x^2+\sqrt z)^{-1}\right],\quad z\in \C\setminus[0,\infty),
$$
which is obtained by the identity $\partial_x^4-z=(-\partial_x^2-\sqrt z)(-\partial_x^2+\sqrt z)$ and the first resolvent equation.
This formula also gives the explicit formula of the kernel of $R_0^\pm(\lambda^4)$:
\begin{align}
\label{free_resolvent}
R_0^\pm(\lambda^4,x,y)=\frac{1}{4\lambda^3}\Big(\pm ie^{\pm i\lambda|x-y|}-e^{-\lambda|x-y|}\Big)=\frac{F_\pm(\lambda|x-y|)}{4\lambda^3},
\end{align}
where
$
F_\pm(s)=\pm ie^{\pm is}-e^{-s}$.
The existence of $R^\pm_V(\lambda)$ for $\lambda>0$ under our assumption of Theorem \ref{theorem_1} has been also already shown (see \cite{Agmon,Kuroda}).

Then $W_-$ has the following stationary representation (see {\it e.g}. \cite{Kuroda,ReSi}):
\begin{align}
\label{stationary}
W_-=\Id-\frac{2}{\pi i}\int_0^\infty \lambda^3 R_V^+(\lambda^4)V\left(R_0^+(\lambda^4)-R_0^-(\lambda^4)\right)d\lambda.
\end{align}
We decompose the second term of $W_-$ into the low and high energy parts as follows:
taking $\lambda_0>0$ small enough, we let $\chi_0\in C_0^\infty(\mathbb R)$ be such that $\chi_0\equiv1$ on $(-\lambda_0,\lambda_0)$ and $\supp \chi_0\subset [-2\lambda_0,2\lambda_0]$ and set $\chi(\lambda)=\chi_0(\lambda^2)$. We then define
\begin{align}
\label{W^{L}_-}
W^{L}_-&=\int_0^\infty \lambda^3 \chi(\lambda) R_V^+(\lambda^4)V\left(R_0^+(\lambda^4)-R_0^-(\lambda^4)\right)d\lambda,\\
\label{W^H_-}
W^H_-&=\int_0^\infty \lambda^3 \Big(1-\chi(\lambda)\Big) R_V^+(\lambda^4)V\left(R_0^+(\lambda^4)-R_0^-(\lambda^4)\right)d\lambda
\end{align}
such that
\begin{align}
\label{wave_operator}
W_-=\Id-\frac{2}{\pi i}\left(W_-^L+W_-^H\right).
\end{align}
We will deal with $W^{L}_-,W^H_-$ in Sections \ref{section_low} and \ref{section_high}, separately.

\subsection{Resolvent expansion}
\label{subsection_resolvent}

This subsection is mainly devoted to the study of asymptotic behaviors of the resolvent $R_V^+(\lambda^4)$ at low energy $\lambda\to +0$. We also prepare some elementary (but useful) lemmas used in the proof of our main theorems.

We begin with the well known symmetric second resolvent formula for $R_V^+(\lambda^4)$. Let $v(x)=|V(x)|^{1/2}$ and $U(x)=\sgn V(x)$, namely $U(x)=1$ if $V(x)\ge0$ and $U(x)=-1$ if $V(x)<0$. Let $M(\lambda)=U+vR_0^+(\lambda^4)v$ and $M^{-1}(\lambda):=[M(\lambda)]^{-1}$ as long as it exists.
\begin{lemma}
\label{lemma_3_2}
For $\lambda>0$, $M(\lambda)$ is invertible on $L^2$. Moreover, $R_V^+(\lambda^4)V$ has the form
\begin{align}
\label{lemma_3_2_1}
R_V^+(\lambda^4)V=R_0^+(\lambda^4)v M^{-1}(\lambda)v.
\end{align}
\end{lemma}

\begin{proof}
Thanks to the absence of embedded eigenvalues and the Birman-Schwinger principle, $M(\lambda)$ is invertible. Using the decompositions $V=vUv$ and $1=U^2$, we compute
\begin{align*}
R_V^+(\lambda^4)v
&=R_0^+(\lambda^4)v-R_V^+(\lambda^4) v UvR_0^+(\lambda^4)v
=R_0^+(\lambda^4)v\Big(1+UvR_0^+(\lambda^4)v\Big)^{-1}\\
&=R_0^+(\lambda^4)v\Big(U+vR_0^+(\lambda^4)v\Big)^{-1}U^{-1}.
\end{align*}
Multiplying $Uv$ from the right, we obtain the desired formula for $R_V^+(\lambda^4)V$.
\end{proof}

By virtue of the formula \eqref{lemma_3_2_1}, $W_-^L$ defined by \eqref{W^{L}_-} is rewritten in the form
\begin{align}
\label{W^{L}_-_2}
W^{L}_-=\int_0^\infty \lambda^3 \chi(\lambda) R_0^+(\lambda^4)v M^{-1}(\lambda)v\left(R_0^+(\lambda^4)-R_0^-(\lambda^4)\right)d\lambda.
\end{align}
We now recall the asymptotic expansion of $M^{-1}(\lambda)$ proved by \cite{SWY21}, which plays a crucial role in the paper. To this end, we introduce some notations. We say that an integral operator $T_K\in \mathbb B(L^2(\R))$ with kernel $K$ is {\it absolutely bounded} if $T_{|K|}\in \mathbb B(L^2(\R))$. Let
\[
P:=\frac{\<\cdot,v\>v}{\norm{V}_{L^1}},\quad \widetilde P=-\frac{2(1+i)}{\norm{V}_{L^1}}P=-\frac{2(1+i)}{\norm{V}_{L^1}^2}\<\cdot,v\>v,\quad Q_1:=\Id-P.
\]
Note that $P$ is the orthogonal projection onto the span of $v$ in $L^2(\R)$, {\it i.e.} $PL^2=\mathrm{\mathop{span}}\{v\}$. Let $G_0:=(\Delta^2)^{-1}$ and $T_0:=U+vG_0v$ and define the subspaces $Q_{2}L^2,Q_{2}^{0}L^2,Q_{3}L^2$ of $L^2$ by
\begin{align*}
f\in Q_{2}L^2\ &\Longleftrightarrow\ f\in \mathrm{\mathop{span}}\{v,xv\}^\perp;\\
f\in Q_{2}^{0}L^2\ &\Longleftrightarrow\ f\in \mathrm{\mathop{span}}\{v,xv\}^\perp\ \text{and}\ T_0f\in \mathrm{\mathop{span}}\{v,xv\};\\
f\in Q_{3}L^2\ &\Longleftrightarrow\ f\in \mathrm{\mathop{span}}\{v,xv,x^2v\}^\perp\ \text{and}\ T_0f\in \mathrm{\mathop{span}}\{v\}.
\end{align*}
Note that $Q_{3}L^2\subset Q_{2}^{0}L^2\subset Q_{2}L^2$. Let $Q_{\alpha}$ and $Q_{2}^0$ be the orthogonal projection onto $Q_{\alpha}L^2$ and $Q_{2}^0L^2$, respectively. Since $v$ is real-valued, by definition, $Q_1,Q_2,Q_2^0,Q_{3}$ satisfy
\begin{align}
\label{Q}
Q_{\alpha}(x^k v)=0,\quad \<x^kv,Q_{\alpha}f\>=0,\quad Q_{2}^0(x^k v)=0,\quad \<x^k v,Q_{2}^0f\>=0.
\end{align}
for $k=0$ for $Q_1$, $k=0,1$ for $Q_{2},Q_2^0$ and $k=0,1,2$ for $Q_{3}$ . Recall that $|V(x)|\lesssim \<x\>^{-\mu}$.

\begin{lemma}[{\cite[Theorem 1.8 and Remark 1.9]{SWY21}}]
\label{lemma_3_3}
There exists $\lambda_0>0$ such that $M^{-1}(\lambda)$ satisfies the following asymptotic expansions on $L^2(\R)$ for $0<\lambda\le \lambda_0$:
\begin{itemize}
\item[(i)] If zero is a regular point of $H$ and $\mu>15$, then
\begin{align}
\label{lemma_3_3_1}
\nonumber
M^{-1}(\lambda)&=Q_{2}A_{0}^0Q_{2}+\lambda Q_1A_{1}^0Q_1+\lambda^2\left(Q_1A_{21}^0Q_1+Q_{2}A_{22}^0+A_{23}^0Q_{2}\right)\\
&\quad +\lambda^3\left(Q_1A_{31}^0+A_{32}^0Q_1\right)+\lambda^3\widetilde P+\Gamma_4^0(\lambda).
\end{align}
\item[(ii)] If zero is a first kind resonance of $H$ and $\mu>21$, then
\begin{align}
\label{lemma_3_3_2}
M^{-1}(\lambda)
\nonumber
&=\lambda^{-1}Q_{2}^{0}A_{-1}^1Q_{2}^{0}+Q_{2}A_{01}^1Q_1+Q_1A_{02}^1Q_{2}+\lambda\left(Q_1A_{11}^1Q_1+Q_{2}A_{12}^1+A_{13}^1Q_{2}\right)\\
&\quad +\lambda^2\left(Q_1A_{21}^1+A_{22}^1Q_1\right)+\lambda^3\left(Q_1A_{31}^1+A_{32}^1Q_1\right)+\lambda^3\widetilde P+\Gamma_4^1(\lambda).
\end{align}
\item[(iii)] If zero is a second kind resonance of $H$ and $\mu>29$, then
\begin{align}
\label{lemma_3_3_3}
\nonumber
M^{-1}(\lambda)
&=\lambda^{-3}Q_{3}A_{-3}^2Q_{3}
+\lambda^{-2}\left(Q_{3}A_{-21}^2Q_{2}+Q_{2}A_{-22}^2Q_{3}\right)\\
\nonumber
&\quad +\lambda^{-1}\left(Q_{2}A_{-11}^2Q_{2}+Q_{3}A_{-12}^2Q_1+Q_1A_{-13}^2Q_{3}\right)\\
\nonumber
&\quad+Q_{2}A^2_{01}Q_1+Q_1A^2_{02}Q_{2}+Q_{3}A^2_{03}+A^2_{04}Q_{3}+\lambda\left(Q_1A_{11}^2Q_1+Q_{2}A^2_{12}+A^2_{13}Q_{2}\right)\\
&\quad+\lambda^2\left(Q_1A^2_{21}+A^2_{22}Q_1\right)+\lambda^3\left(Q_1A_{31}^2+A_{32}^2Q_1\right)+\lambda^3\widetilde P+\Gamma_4^2(\lambda).
\end{align}
\end{itemize}
Here $A^j_k$ and $A^j_{k\ell}$ are $\lambda$-independent bounded operators on $L^2$ and $\Gamma_4^j(\lambda)$ are $\lambda$-dependent bounded operators on $L^2$ such that all the operators appeared in the right hand sides of \eqref{lemma_3_3_1}, \eqref{lemma_3_3_2} and \eqref{lemma_3_3_3} are absolutely bounded. Moreover, $\Gamma_4^j(\lambda)$ satisfy, for $\ell=0,1,2$,
\begin{align}
\label{lemma_3_3_4}
\norm{\partial_\lambda^\ell \Gamma^j_4(\lambda)}_{L^2\to L^2}\le C_\ell \lambda^{4-\ell},\quad \lambda>0.
\end{align}
\end{lemma}

\begin{remark}$ $
\begin{itemize}
\item[(1)] We have used different notations $Q_1,Q_{2},Q_2^0,Q_{3}$ in Lemma \ref{lemma_3_3} from ones in \cite{SWY21}, which is convenient for our purpose. The relation between our notations and original ones are as follows: $(Q_1,Q_{2},Q_2^0,Q_{3})$ correspond to $(Q,S_0,S_1,S_2)$ in \cite[Theorem 1.9]{SWY21}.
    \vskip0.2cm
\item[(2)] In \cite[Remark 1.9]{SWY21}, it was only stated that \eqref{lemma_3_3_4} holds for $\ell=0,1$ under a slightly weaker condition on $V$ than \eqref{mu}. However, it can be seen from the proof of \cite[Theorem 1.9]{SWY21} that \eqref{lemma_3_3_4} in fact holds for $\ell=0,1,2,3,4$ under the condition \eqref{mu}.
\end{itemize}
\end{remark}

We also prepare three elementary (but useful) lemmas.

\begin{lemma}[{\cite[Lemma 2.5]{SWY21}}]
\label{lemma_3_4}
Let $\lambda>0$ and $x,y\in \R$.
\begin{itemize}
\item[(i)] If $F\in C^1(\R_+)$ then
\[
F(\lambda|x-y|)=F(\lambda |x|)-\lambda y\int_0^1\sgn(x-\theta y)F'(\lambda|x-\theta y|)d\theta.
\]
\item[(ii)] If $F\in C^2(\R_+)$ and $F'(0)=0$, then
\[
F(\lambda|x-y|)=F(\lambda |x|)-\lambda y \sgn(x)F'(\lambda|x|)+\lambda^2 y^2\int_0^1(1-\theta)F''(\lambda|x-\theta y|)d\theta.
\]
\item[(iii)] If $F\in C^3(\R_+)$ and $F'(0)=F''(0)=0$, then
\begin{align*}
F(\lambda|x-y|)
&=F(\lambda |x|)-\lambda y \sgn(x)F'(\lambda|x|)+\frac{\lambda^2 y^2}{2}F''(\lambda|x|)\\
&\quad-\frac{\lambda^3 y^3}{2}\int_0^1(1-\theta)^2\sgn(x-\theta y)F'''(\lambda|x-\theta y|)d\theta.
\end{align*}
\end{itemize}
\end{lemma}

We will mainly use this lemma for
$
F_\pm(s)=\pm ie^{\pm i s}-e^{-s}
$.
Combined with \eqref{free_resolvent} and \eqref{Q}, Lemma \ref{lemma_3_4} implies the following formulas, which will be one of key tools in the paper.

\begin{lemma}
\label{lemma_projection}
Let $Q_1,Q_{2},Q_{2}^{0},Q_{3}$ be as above, $\alpha=0,1,2,3$ and $\lambda>0$. Then:
\begin{align*}
&[Q_{\alpha} vR_0^\pm(\lambda^4)f](x)\\
&=\frac{(-1)^\alpha \lambda^{-3+\alpha}}{4\cdot (\alpha-1)!} Q_{\alpha}\left(x^\alpha v\int \int_0^1(1-\theta)^{\alpha-1}\left(\sgn(y-\theta x)\right)^\alpha  F_\pm^{(\alpha)}(\lambda|y-\theta x|)f(y)d\theta dy\right),\\
&[R_0^\pm(\lambda^4)vQ_{\alpha}f](x)\\
&=\frac{(-1)^\alpha \lambda^{-3+\alpha}}{4\cdot (\alpha-1)!}\int \int_0^1(1-\theta)^{\alpha-1}\!\left(\sgn(x-\theta y)\right)^\alpha  F_\pm^{(\alpha)}(\lambda|x-\theta y|)y^\alpha v(y)(Q_{\alpha}f)(y)d\theta dy,
\end{align*}
where for simplicity we have used the convention that $(\sgn x)^2\equiv1$ for all $x\in \R$ . Moreover, these estimates for $\alpha=2$ also hold with $Q_2$ replaced by $Q_2^0$.
\end{lemma}
More precisely speaking, the above formula for $Q_{\alpha}vR_0^\pm(\lambda^4)f$ means
$$
Q_{\alpha}vR_0^\pm(\lambda^4)f=\frac{(-1)^\alpha \lambda^{-3+\alpha}}{4\cdot (\alpha-1)!} Q_{\alpha}\widetilde f_{\pm,\alpha}
$$
with
$$
\widetilde f_{\pm,\alpha}(\lambda,x)=x^\alpha v\int\int_0^1(1-\theta)^{\alpha -1}\left(\sgn(y-\theta x)\right)^\alpha F_\pm^{(\alpha )}(\lambda|y-\theta x|)f(y)d\theta dy.
$$

Note that the subscript $\alpha $ of $Q_{\alpha}$ coincides with the order of differentiation for $F_\pm$. This is the main reason why we use the notations $Q_1,Q_{2},Q_{2}^{0},Q_{3}$ instead of the original ones.

\begin{remark}
\label{remark_lemma_projection}
At the level of the order with respect to $\lambda$,  this lemma shows
$$
Q_{\alpha}vR^\pm_0(\lambda^4),R^\pm_0(\lambda^4)vQ_{\alpha}=O(\lambda^{-3+\alpha }),\quad \lambda\to +0.
$$
This gain of positive powers of $\lambda$, compared with that of the free resolvent $R^\pm_0(\lambda^4)=O(\lambda^{-3})$, is useful to cancel out singularities near $\lambda=0$ appeared in the expansion for $M^{-1}(\lambda)$ (see Lemma \ref{lemma_3_3}). This cancellation properties will be crucial in our argument.
\end{remark}

\begin{proof}[Proof of Lemma \ref{lemma_projection}]
Since $F_\pm'(0)=0$, we can apply Lemma \ref{lemma_3_4} to $F_\pm$ obtaining
\begin{align*}
R_0^\pm(\lambda^4,x,y)
&=\frac{F_\pm(\lambda |x|)}{4\lambda^3}-\frac{y}{4\lambda^2}\int_0^1\sgn(x-\theta y)F_\pm'(\lambda|x-\theta y|)d\theta\\
&=\frac{F_\pm(\lambda |x|)}{4\lambda^3}-\frac{y \sgn(x)F_\pm'(\lambda|x|)}{4\lambda^2}+\frac{y^2}{4\lambda}\int_0^1(1-\theta)F_\pm''(\lambda|x-\theta y|)d\theta.
\end{align*}
The cases $\alpha =1,2$ follow from this formula and \eqref{Q}. Indeed, we have
\begin{align*}
Q_2 vR_0^\pm(\lambda^4)f
&=\frac{1}{4\lambda^3}Q_2 (v)\int F_\pm(\lambda|y|)f(y)dy-\frac{1}{4\lambda^2} Q_2 (xv)\int \sgn(y)F_\pm'(\lambda|y|)f(y)dy\\
\nonumber
&\quad+\frac{1}{4\lambda} Q_2 \left(x^2v\int \int_0^1(1-\theta)F_\pm''(\lambda|y-\theta x|) f(y)d\theta dy\right)\\
&=\frac{1}{4\lambda}Q_2 \left(x^2v\int \int_0^1(1-\theta)F_\pm''(\lambda|y-\theta x|) f(y)d\theta dy\right).
\end{align*}
The proofs for the other cases with $Q_1,Q_2^0$ are similar. For the case $\alpha =3$, we write
$$4\lambda^3R^\pm(\lambda^4,x,y)=F_\pm(\lambda|x-y|)=\widetilde F_\pm(\lambda|x-y|)-\frac{1\pm i}{2}\lambda^2|x-y|^2,$$ where $
\widetilde F_\pm(s)=F_\pm(s)+\frac{1\pm i}{2}s^2
$. Then we can write
\begin{align*}
Q_{3}vR_0^\pm(\lambda^4)f=Q_{3}\left\{v\int\left(\widetilde F_\pm(\lambda|x-y|)-\frac{1\pm i}{2}\lambda^2|x-y|^2\right)f(y)dy\right\}.
\end{align*}
For the first term of the right hand side, since $\widetilde F_\pm'(0)=\widetilde F_\pm''(0)=0$ and $F_\pm'''\equiv\widetilde F_\pm'''$, we can apply Lemma \ref{lemma_3_4} (iii) and \eqref{Q} to compute
\begin{align*}
&Q_{3}\left(v \int \widetilde F_\pm(\lambda|x-y|)f(y)dy\right)\\
&=-\frac{1}{8}Q_{3}\left(x^3v\int \int_0^1(1-\theta)^2\sgn(y-\theta x)F_\pm'''(\lambda|y-\theta x|) f(y)d\theta dy\right),
\end{align*}
while the second part related with $|x-y|^2$ vanishes identically by virtue of \eqref{Q}. The proof for $R_0^+(\lambda^4)vQ_{3}f$ is analogous.
\end{proof}

We will also use often the following simple formula:

\begin{lemma}
\label{lemma_FF}
Let $F_\pm(s)=\pm ie^{\pm is}-e^{-s}$ and $\alpha,\beta\in \N\cup\{0\}$. Then
\begin{align*}
&f_{\alpha\beta}(\lambda,x,y):=F_+^{(\alpha)}(\lambda|x|)[F_+^{(\beta)}-F_-^{(\beta)}](\lambda|y|)\\
&=-i^{\alpha+\beta}\Big(e^{i\lambda(|x|+|y|)}+(-1)^\beta e^{i\lambda(|x|-|y|)}\Big)+(-1)^{\alpha+1}i^{\beta+1}\Big((-1)^\beta e^{-\lambda(|x|+i|y|)}+e^{-\lambda(|x|-i|y|)}\Big).
\end{align*}
\end{lemma}

\begin{proof}
A direct calculation yields
\begin{align*}
&F_+^{(\alpha)}(s)[F_+^{(\beta)}-F_-^{(\beta)}](t)
=\Big(i^{\alpha+1}e^{is}+(-1)^{\alpha+1}e^{-s}\Big)\Big(i^{\beta+1}e^{it}-(-i)^{\beta+1}e^{-it}\Big)\\
&=-i^{\alpha+\beta}\Big(e^{i(s+t)}+(-1)^\beta e^{i(s-t)}\Big)+(-1)^{\alpha+1}i^{\beta+1}\Big((-1)^\beta e^{-(s+it)}+e^{-(s-it)}\Big)
\end{align*}
The lemma then follows by letting $s=\lambda|x|$ and $t=\lambda|y|$.
\end{proof}

\section{Boundedness of some integrals related with wave operators}
\label{subsection_integral_operators}
Recall that $T_K$ denotes the integral operator with the kernel $K(x,y)$:
$$
T_Kf(x)=\int K(x,y)f(y)dy.
$$
This section is devoted to preparing some basic criterion to obtain several boundedness of $T_K$ related with the wave operator $W_-$.

\subsection{Classical Schur kernels}
We first recall the classical Schur lemma:
\begin{lemma}
\label{lemma_2_1}
$T_K\in \mathbb B(L^p(\R))$ for all $1\le p\le \infty$ if $K$ satisfies $$
\sup_{x\in \R} \int|K(x,y)|dy+\sup_{y\in \R} \int|K(x,y)|dx<\infty.
$$
\end{lemma}

We often use this lemma for the kernel satisfying  $|K(x,y)|\lesssim \<|x|-|y|\>^{-\rho}$ with some $\rho>1$. In fact, one can also obtain several weighted boundedness for such operators:

\begin{lemma}
\label{lemma_2_2}
Let $K$ satisfy $|K(x,y)|\lesssim \<|x|-|y|\>^{-\rho}$ on $\R^2$ with some $\rho>1$ and $\tau f(x)=f(-x)$. Let $1<p<\infty$, $w_p\in A_p$ and $w_1\in A_1$. Then $T_{K}$ satisfies the following bounds:
\begin{align}
\label{lemma_2_2_1}
\norm{T_{K}f}_{L^p(w_p)}+\norm{T_{K}^*f}_{L^p(w_p)}
&\lesssim [w_p]_{A_p}^{\max\{1,1/(p-1)\}}(\norm{f}_{L^p(w_p)}+\norm{\tau f}_{L^p(w_p)}),\\
\label{lemma_2_2_2}
\norm{T_{K}f}_{L^{1,\infty}(w_1)}+\norm{T_{K}^*f}_{L^{1,\infty}(w_1)}
&\lesssim [w_1]_{A_1}(1+\log [w]_{A_1})(\norm{f}_{L^1(w_1)}+\norm{\tau f}_{L^1(w_1)}).
\end{align}
\end{lemma}

\begin{proof}
Let $\chi_\pm=\chi_{\R_\pm}$ be the characteristic function of $\R_\pm$. We decompose $K$ as
\begin{align*}
K(x,y)
&=\left(\chi_+(x)+\chi_-(x)\right)K(x,y)\left(\chi_+(y)+\chi_-(y)\right)\\
&=\sum_\pm \left(\chi_\pm(x)K(x,y)\chi_\pm(y)+\chi_\pm(x)K(x,y)\chi_\mp(y)\right)\\
&=:\sum_\pm\left(K_{\pm,\pm}(x,y)+K_{\pm,\mp}(x,y)\right),
\end{align*}
By assumption, $K_{\pm,\pm}$ and $K_{\pm,\mp}$ satisfy
$$
|K_{\pm,\pm}(x,y)|\lesssim \<x-y\>^{-\rho},\quad |K_{\pm,\mp}(x,y)|\lesssim \<x+y\>^{-\rho}.
$$
Hence if we set $\widetilde K_{\pm,\mp}(x,y)=K_{\pm,\mp}(x,-y)$ and $\tau f(x)=f(-x)$, then
\begin{align}
\label{lemma_2_2_proof_1}
|K_{\pm,\pm}(x,y)|+|\widetilde K_{\pm,\mp}(x,y)|\lesssim \<x-y\>^{-\rho}
\end{align}
and $T_{K}=T_{K_{+,+}}+T_{K_{-,-}}+(T_{K_{+,-}}+T_{K_{-,+}})\tau$. It follows from \eqref{lemma_2_2_proof_1} and Lemma \ref{lemma_2_1} that the integral operator $T$ with the kernel $\<x-y\>^{-\rho}$ is a Calder\'on--Zygmund operator (see Appendix \ref{appendix_CZ}). Theorem \ref{theorem_CZ_1} in Appendix \ref{appendix_CZ} thus implies, for $1<p<\infty$,
\begin{align*}
\norm{T_{K}f}_{L^p(w_p)}
&\lesssim \norm{Tf}_{L^p(w_p)}+\norm{T\tau f}_{L^p(w_p)}\\
&\lesssim [w_p]_{A_p}^{\max\{1,1/(p-1)\}}(\norm{f}_{L^p(w_p)}+\norm{\tau f}_{L^p(w_p)}).
\end{align*}
Similarly, we obtain for $p=1$,
\begin{align*}
\norm{T_{K}f}_{L^{1,\infty}(w_1)}&\lesssim [w_1]_{A_1}(1+\log[w_1]_{A_1})(\norm{f}_{L^1(w_1)}+\norm{\tau f}_{L^1(w_1)}).
\end{align*}
By virtue of \eqref{lemma_2_2_proof_1}, the same argument also implies the desired bounds for $T_{K}^*$.
\end{proof}

\subsection{Non-classical kernels related with wave operators}

As observed by \cite{GoGr21} for the case $\Delta^2+V(x)$ on $\R^3$, the wave operator for $(-\Delta)^m+V(x)$ on $\R^n$ has some singular integrals in its stationary representation if $n<2m$.
Precisely, in the present case,  the low energy part $W^{L}_-$ of the wave operator $W_-$ also has several terms with kernels satisfying $|K(x,y)|\lesssim \<|x|-|y|\>^{-1}$ only. To deal with such terms, we further prepare two lemmas based on the theory of Calde\'ron--Zygmund operators (see Appendix \ref{appendix_CZ} for Calde\'ron--Zygmund operators). The following lemma is concerned with the boundedness on weighted $L^p$-spaces:

\begin{lemma}
\label{lemma_2_3}
Let $1<p<\infty$ and $\psi\in C^\infty(\R;\R)$ be such that $\psi(s)=0$ for $0\le s\le 1$ and $\psi(s)=1$ for $s\ge2$. Let $K(x,y)$ be a linear combination of the following four functions
\begin{align*}
k_1^\pm(x,y)=\frac{\psi(||x|\pm |y||^2)}{|x|\pm |y|},\quad k_2^\pm(x,y)=\frac{\psi(||x|-|y||^2)}{|x|\pm i|y|}.
\end{align*}
Then  $T_{K}$ and $T_{K}^*$ satisfy the same bounds as \eqref{lemma_2_2_1} and \eqref{lemma_2_2_2}.
\end{lemma}

\begin{remark}
Some singular integrals similar to $T_{k_j^\pm}$ have been already appeared in \cite[Lemma 3.3]{GoGr21}. Precisely, the singular integral with the kernel $|x|(|x|^4-|y|^4)^{-1}$ in $\R^3$
has been studied by using the spherical average and $L^p$-boundedness of the maximal (truncated) Hilbert transform and the Hardy-Littlewood Maximal function. Here we make use of a specific feature in one space dimension to observe that our operators  $T_{k_j^\pm}$ also fall within the scope of the theory of Calde\'ron--Zygmund operators.
\end{remark}

\begin{proof}
With some constants $a,b,c,d\in \C$, we can write
$$
K=a\frac{\psi(\big||x|+ |y|\big|^2)}{|x|+ |y|}+b\frac{\psi(\big||x|- |y|\big|^2)}{|x|- |y|}+c\frac{\psi(||x|-|y||^2)}{|x|+ i|y|}+d \frac{\psi(||x|-|y||^2)}{|x|- i|y|}.
$$
We set $\chi_\pm=\chi_{\R_\pm}$, $\tau f(x)=f(-x)$ and
$$
\widetilde k_1(x,y)=\frac{\psi(|x- y|^2)}{x-y},\quad \widetilde k_2^\pm(x,y)=\frac{\psi(|x- y|^2)}{x\pm iy}.
$$
Then $T_K$  is written in the form
\begin{align}
\label{lemma_2_3_proof_1}
\nonumber
T_{K}
&=\left\{a\left(\chi_{+}T_{\widetilde k_1}\chi_{-}-\chi_{-}T_{\widetilde k_1}\chi_{+}\right)+b\left(\chi_{+}T_{\widetilde k_1}\chi_{+}-\chi_{-}T_{\widetilde k_1}\chi_{-}\right)\right\}(1+\tau)\\
&\quad +\left\{c\left(\chi_+T_{\widetilde k_2^+}\chi_+- \chi_-T_{\widetilde k_2^+}\chi_-\right)+d\left(\chi_+T_{\widetilde k_2^-}\chi_+- \chi_-T_{\widetilde k_2^-}\chi_-\right)\right\}(1+\tau).
\end{align}
Indeed, since
$k_1^\pm(x,y)=\widetilde k_1(|x|,\mp |y|)$ and $\widetilde k_1(-x,-y)=-\widetilde k_1(x,y)$,
we have
\begin{align*}
k_1^\pm(x,y)&=\left(\chi_{+}(x)+\chi_{-}(x)\right)k_1^\pm(x,y)\left(\chi_{+}(y)+\chi_{-}(y)\right)\\
&=\chi_{+}(x)\widetilde k_1 (x,\mp y)\chi_{+}(y)
+\chi_{+}(x)\widetilde k_1(x,\pm y)\chi_{-}(y)
\\&\quad
-\chi_{-}(x)\widetilde k_1(x,\pm y)\chi_{+}(y)
-\chi_{-}(x)\widetilde k_1(x,\mp y)\chi_{-}(y).
\end{align*}
By the change of variable $y\mapsto -y$ in the first and fourth terms for the ``$+$" case and the second and third terms for the ``$-"$ case, respectively, we obtain
\begin{align*}
T_{k_1^\pm}f(x)
=[(\chi_{+}T_{\widetilde k_1}\chi_{\mp}-\chi_{-}T_{\widetilde k_1}\chi_{\pm})(1+\tau)]f(x).
\end{align*}
A similar calculation also implies
\begin{align*}
k_2^\pm(x,y)&=\chi_+(x)\widetilde k_2^\pm(x,y)\chi_+(y)+\chi_+(x)\widetilde k_2^\pm(x,-y)\chi_-(y)\\
&\quad -\chi_-(x)\widetilde k_2^\pm(x,-y)\chi_+(y)-\chi_-(x)\widetilde k_2^\pm(x,y)\chi_-(y).
\end{align*}
Hence, by the change of variable $y\mapsto -y$ in the second and third terms, we have
\begin{align*}
T_{k_2^\pm}f(x)=[(\chi_+T_{\widetilde k_2^\pm}\chi_+- \chi_-T_{\widetilde k_2^\pm}\chi_-)(1+\tau)]f(x).
\end{align*}
These two formulas imply \eqref{lemma_2_3_proof_1}. Since both the multiplication operator by $\chi_{\pm}(x)$ belongs to $\mathbb B(L^p(w_p))\cap\mathbb B(L^{1,\infty}(w_1))$ with operator norms $1$ for all $1\le p<\infty$, we obtain
\begin{align*}
\norm{T_{K}f}_{\mathcal Y}\lesssim \left(\norm{T_{\widetilde k_1}}_{\mathcal X\to \mathcal Y}+\norm{T_{\widetilde k_2^+}}_{\mathcal X\to \mathcal Y}+\norm{T_{\widetilde k_2^-}}_{\mathcal X\to \mathcal Y}\right)\left(\norm{f}_{\mathcal X}+\norm{\tau f}_{\mathcal X}\right)
\end{align*}
if $(\mathcal X,\mathcal Y)\in \{(L^p(w_p),L^p(w_p)),(L^1(w_1),L^{1,\infty}(w_1))\}$. Moreover, since
\begin{align*}
\overline{\widetilde k_1(y,x)}=-\widetilde k_1(x,y),\quad \overline{\widetilde k_2^\pm(y,x)}=\pm i\widetilde k_2^\pm(x,y)
\end{align*}
we have $(T_{\widetilde k_1})^*=-T_{\widetilde k_1}$ and $(T_{\widetilde k_2^\pm})^*=\pm iT_{\widetilde k_2^\pm}$ and
$$
\norm{T_{K}^*f}_{\mathcal Y}\lesssim  \left(\norm{T_{\widetilde k_1}}_{\mathcal X\to \mathcal Y}+\norm{T_{\widetilde k_2^+}}_{\mathcal X\to \mathcal Y}+\norm{T_{\widetilde k_2^-}}_{\mathcal X\to \mathcal Y}\right)\left(\norm{f}_{\mathcal X}+\norm{\tau f}_{\mathcal X}\right).
$$
By virtue of Theorem \ref{theorem_CZ_1}, it thus is enough to show that $T_{\widetilde k_1},T_{\widetilde k_2^\pm}$ are Calder\'on--Zygmund operators, namely $\widetilde k_1,\widetilde k_2^\pm$ are standard kernels and $T_{\widetilde k_1},T_{\widetilde k_2^\pm}\in \mathbb B(L^2(\R))$ (see Appendix \ref{appendix_CZ}).

Since $\widetilde k_1,\widetilde k_2^\pm$ are supported away from a neighborhood of the diagonal line, they are smooth on $\R^2$. Moreover, since $1\le |x-y|^2\le 2$ on $\supp \psi'(|x-y|^2)$, we have
\begin{align*}
|\partial_x^\alpha \partial_y^\beta \widetilde k_1(x,y)|+|\partial_x^\alpha \partial_y^\beta \widetilde k_2^\pm(x,y)|&\lesssim \<x-y\>^{-1-\alpha-\beta},\quad \alpha,\beta=0,1,2,....
\end{align*}
Hence $\widetilde k_1$ and $\widetilde k_2^\pm$ are standard kernels. To show $T_{\widetilde k_1}\in \mathbb B(L^2(\R))$, we observe that, modulo a rapidly decaying error term, $T_{\widetilde k_1}$ is essentially a truncated Hilbert transform $\mathcal H^{(2)}$ defined by
$$
\mathcal H^{(\ep)}f(x)=\int_{|x-y|>\ep}\frac{f(y)}{x-y}dy,\quad \ep>0.
$$
 Indeed, since $\psi(s)=1$ for $s\ge2$, we have, for any $N\ge0$,
\begin{align*}
\widetilde k_1(x,y)
=(\chi_{\{|x-y|>2\}}+\chi_{\{|x-y|\le2\}})\widetilde k_1(x,y)
=\frac{\chi_{\{|x-y|>2\}}}{x-y}+O(\<x-y\>^{-N}).
\end{align*}
Since $\mathcal H^{(2)}\in \mathbb B(L^2(\R))$ (see \cite[Theorems 5.1.12]{Grafakos_Classical_III}) and the error term is also bounded on $L^2(\R)$ by Lemma \ref{lemma_2_1}, so is $T_{\widetilde k_1}$. For the operators $T_{\widetilde k_2^\pm}$, we compute
$$
\widetilde k_2^\pm(x,y)=\frac{x}{x^2+y^2}\psi(|x- y|^2)\mp i\frac{y}{x^2+y^2}\psi(|x- y|^2)=:\widetilde k_{21}(x,y)\mp i \widetilde k_{22}(x,y).
$$
Then, $T_{\widetilde k_{21}}\in \mathbb B(L^\infty(\R))\cap \mathbb B(L^1(\R),L^{1,\infty}(\R))$ since
$$
\sup_{x\in \R}\int|\widetilde k_{21}(x,y)|dy\lesssim \sup_{x\in \R}\int \frac{|x|}{x^2+y^2+1}dy\lesssim1,\quad |\widetilde k_{21}(x,y)|\lesssim\<x\>^{-1}\in L^{1,\infty}(\R).
$$
The Marcinkiewicz interpolation theorem then yields $T_{\widetilde k_{21}}\in \mathbb B(L^2(\R))$. Since $T_{\widetilde k_{22}}=(T_{\widetilde k_{21}})^*$,  $T_{\widetilde k_{22}}\in \mathbb B(L^2(\R))$ by duality.  Hence $T_{\widetilde k_{2}^\pm}\in \mathbb B(L^2)$. Summarizing these arguments, we conclude that $T_{\widetilde k_1}$ and $T_{\widetilde k_2^\pm}$ are Calder\'on--Zygmund operators. This completes the proof.
\end{proof}

\begin{remark}
Although the proof is reduced to the theory of Calder\'on--Zygmund operators, the operator $T_K$ in Lemma \ref{lemma_2_3} itself is not a Calder\'on--Zygmund operator in general. Indeed, for instance, $\psi(||x|- |y||^2)(|x|-|y|)^{-1}$ is not a standard kernel.
\end{remark}

The following lemma will be used to prove the $\H^1$-$L^1$ and $L^\infty$-$\BMO$ boundedness.
\begin{lemma}
\label{lemma_2_4}
Let $k_1^\pm,k_2^\pm$ be as in Lemma \ref{lemma_2_3} and $a,b\in \C$. Define $g_{a,b}^\pm=g_{a,b}^\pm(x,y)$ by
$$
g_{a,b}^\pm=a(k_1^+\pm k_1^-)+b(k_2^+\pm k_2^-)
$$
and consider the following eight integral kernels
\begin{align*}
g_{1,a,b}^\pm(x,y)&=g_{a,b}^\pm(x,y),\\
g_{2,a,-a}^+(x,y)&=g_{a,-a}^+(x,y)\sgn y,\\
g_{2,a,b}^-(x,y)&=g_{a,b}^-(x,y)\sgn y,\\
g_{3,a,b}^+(x,y)&=g_{a,b}^+(x,y)\sgn x,\\
g_{3,a,-ia}^-(x,y)&=g_{a,-ia}^-(x,y)\sgn x,\\
g_{4,a,-a}^+(x,y)&=g_{a,-a}^+(x,y)\sgn x\sgn y,\\
g_{4,a,b}^-(x,y)&=g_{a,b}^-(x,y)\sgn x\sgn y.
\end{align*}
Then $T_{g_1^\pm},T_{g_2}^\pm,T_{g_3^\pm}\in \mathbb B(\mathcal H^1(\R),L^1(\R))\cap \mathbb B(L^\infty(\R),\BMO(\R))$ and $T_{g_4^\pm}\in \mathbb B(\mathcal H^1(\R),L^1(\R))$.
\end{lemma}

\begin{remark}
For simplicity, we often omit the subscript $a,b$ and write simply $g_{1}^\pm(x,y)=g_{1,a,b}^\pm(x,y)$ and so on if there is no confusion. Note that, in contrast with $g_1^\pm$, $g_2^-$, $g_3^+$ and $g_4^-$,  there are restrictions on the choice of $b$ for the kernels  $g_{2,a,-a}^+$, $g_{3,a,-ia}^-$ and $g_{4,a,-a}^+$.
\end{remark}

\begin{proof}
We first observe that Lemma \ref{lemma_2_3} applies to $T_{g_j^\pm}$ for all $1\le j\le4$ since the multiplication by $\sgn x$ is bounded on $L^p(w_p)$ for all $1\le p<\infty$ and on $L^{1,\infty}(w_1)$. We also observe some duality among $T_{g_{j,a,b}^\pm}$. Namely, since a direct calculation yields
\begin{align*}
\overline{g_{a,b}^\pm(y,x)}&=g_{\overline a,i\overline b}^\mp(x,y),
\end{align*}
we have
$$
(T_{g_{1,a,b}^\pm})^*=T_{g_{1,\overline a,i\overline b}^\mp},\quad (T_{g_{2,a,b}^\pm})^*=T_{g_{3,\overline a,i\overline b}^\mp},\quad (T_{g_{3,a,b}^\pm})^*=T_{g_{2,\overline a,i\overline b}^\mp},\quad (T_{g_{4,a,b}^\pm})^*=T_{g_{4,\overline a,i\overline b}^\mp}.
$$
Since $\BMO(\R)=\H^1(\R)^*$ (see \cite{Fefferman}),  it is thus enough to show $T_{g_j^\pm}\in \mathbb B(\mathcal H^1(\R),L^1(\R))$ for $1\le j\le 4$ with the above restrictions on $b$ for $g_2^+,g_3^-$ and $g_4^+$. Moreover, since the multiplication by $\sgn x$ is bounded on $L^1(\R)$, it is enough to consider $T_{g_1^\pm}$ and $T_{g_2^\pm}$ only.

The proof of $T_{g_1^\pm},T_{g_2^\pm}\in \mathbb B(\mathcal H^1(\R),L^1(\R))$ follows a classical argument in the proof of the $\mathcal H^1$-$L^1$ boundedness of Calder\'on--Zygmund operators. We let $f\in \mathcal H^1$ and apply the atomic decomposition (see \cite[Section 2.3.5]{Grafakos_Modern_III}) to obtain
$$
f=\sum_{j=1}^\infty \lambda_ja_j(x),\quad \sum_{j=1}^\infty |\lambda_j|\lesssim \norm{f}_{\mathcal H^1},
$$
where $\lambda_j\in \C$ and $a_j$ are $L^\infty$-atoms for $\mathcal H^1$ satisfying, with some $x_j\in \R$ and $r_j\ge2$,
$$
\supp a_j\subset (x_j-r_j,x_j+r_j),\quad \norm{a_j}_{L^\infty}\lesssim r_j^{-1},\quad \int a_j(x)dx=0.
$$
Hence, for a given integral operator $T$, once we obtain
\begin{align}
\label{lemma_2_4_proof_1}
\norm{T a_j}_{L^1(\R)}\lesssim 1
\end{align}
uniformly in $j$, $T$ is bounded from $\H^1$ to $L^1$ since
$$
\norm{Tf}_{L^1(\R)}\le \sum_{j=1}^\infty |\lambda_j|\norm{Ta_j}_{L^1(\R)}\lesssim \norm{f}_{\H^1(\R)}.
$$
It is thus enough to prove \eqref{lemma_2_4_proof_1} for $T= T_{g_1^\pm},T_{g_2^\pm}$. Let $I=(x_0-r,x_0+r)$ with some fixed $x_0\in \R,r\ge2$ and take an $L^\infty$-atom $a$ satisfying
$$
\supp a\subset I,\quad \norm{a}_{L^\infty}\lesssim r^{-1},\quad \int_Ia(x)dx=0.
$$
We also let $I_*=(x_0-3r,x_0+3r)\cup (-x_0-3r,-x_0+3r)$ and decompose
$$
\norm{Ta}_{L^1(\R)}=\norm{Ta}_{L^1(I_*)}+\norm{Ta}_{L^1(I_*^c)}.
$$
For the first term, Lemma \ref{lemma_2_3} and H\"older's inequality imply
\begin{align}
\label{lemma_2_4_proof_2}
\norm{Ta}_{L^1(I_*)}\le |I_*|^{1/2}\norm{Ta}_{L^2(I_*)}\lesssim r^{1/2}\norm{a}_{L^2(I_*)}\lesssim r^{1/2}\norm{a}_{L^\infty}r^{1/2}\lesssim1
\end{align}
uniformly in $x_0$ and $r$ for all $T= T_{g_1^\pm},T_{g_2^\pm}$ and $a,b\in \C$.

To deal with the second term, we first deal with $g_1^\pm$ and $g_2^-$. Let $x\in I_*^c$ and $y\in \supp a$. Then, since $\supp a\subset I$, we have
\begin{align*}
||x|\pm |y||\ge \min_\pm |x\pm y|\ge r\ge2
\end{align*}
and hence $\psi(\big||x|\pm|y|\big|^2)=1$ since $\psi\equiv1$ on $[2,\infty)$. It also follows that
\begin{align*}
|x\pm y|\ge |x\pm x_0|-|x_0-y|\ge |x\pm x_0|-r\ge |x\pm x_0|/2.
\end{align*}
With these remarks at hand, we obtain
\begin{align*}
|g_1^\pm(x,y)-g_1^\pm(x,x_0)|
&\lesssim \sum_\pm \left(\left|\frac{1}{|x|\pm |y|}-\frac{1}{|x|\pm |x_0|}\right|+\left|\frac{1}{|x|\pm i|y|}-\frac{1}{|x|\pm i|x_0|}\right|\right)\\
&\lesssim \sum_\pm \left(\left|\frac{|x_0|-|y|}{(|x|\pm |y|)(|x|\pm |x_0|)}\right|+\left|\frac{|x_0|-|y|}{(|x|\pm i|y|)(|x|\pm i|x_0|)}\right|\right)\\
&\lesssim \frac{|x_0-y|}{\ds \min_\pm |x\pm x_0|^2}.
\end{align*}
Using the relations
\begin{align*}
\frac{\sgn y}{|x|+|y|} -\frac{\sgn y}{|x|-|y|}&=\frac{-2y}{x^2-y^2}=\frac{1}{x+y}-\frac{1}{x-y},\\
\frac{\sgn y}{|x|+i|y|}-\frac{\sgn y}{|x|-i|y|}&=\frac{-2iy}{x^2+y^2}=\frac{1}{x+iy}-\frac{1}{x-iy},
\end{align*}
we also have, for $x\in I_*^c$ and $y\in I$,
\begin{align*}
|g_2^-(x,y)-g_2^-(x,x_0)|
&\lesssim \sum_\pm \left(\left|\frac{1}{x\pm y}-\frac{1}{x\pm x_0}\right|+\left|\frac{1}{x\pm iy}-\frac{1}{x\pm ix_0}\right|\right)\\
&\lesssim \frac{|x_0-y|}{|x+x_0|^2}+\frac{|x_0-y|}{|x-x_0|^2}\lesssim \frac{|x_0-y|}{\ds \min_\pm |x\pm x_0|^2}.
\end{align*}
Hence, for the three cases $K=g_1^\pm,g_2^-$, $T_K$ satisfies
\begin{align*}
\nonumber
\norm{T_Ka}_{L^1(I_*^c)}
&=\int_{I_*^c}\left|\int_I\Big(K(x,y)-K(x,x_0)\Big)a(y)dy\right|dx\\
\nonumber
&\lesssim \norm{a}_{L^\infty} \int_I |x_0-y|dy \int_{I_*^c}\frac{1}{\ds \min_\pm |x\pm x_0|^2} dx\\
&\lesssim r\int_{I_*^c}\frac{1}{\ds \min_\pm |x\pm x_0|^2} dx.
\end{align*}
Setting $U_{x_0}=\{x\ |\ |x+x_0|\le |x-x_0|\}$, we further obtain
\begin{align*}
\int_{I_*^c}\frac{1}{\ds \min_\pm |x\pm x_0|^2} dx
&=\int_{I_*^c\cap U_{x_0}}\frac{1}{|x+ x_0|^2} dx+\int_{I_*^c\cap U_{x_0}^c}\frac{1}{|x-x_0|^2} dx\\
&\le \int_{|x+x_0|\ge 3r}\frac{1}{|x+ x_0|^2} dx+\int_{|x-x_0|\ge 3r}\frac{1}{|x-x_0|^2} dx\lesssim r^{-1}.
\end{align*}
It thus follows that
$\norm{T_Ka}_{L^1(I_*^c)}\lesssim 1
$ uniformly in $x_0,r$. This bound, together with \eqref{lemma_2_4_proof_2}, implies \eqref{lemma_2_4_proof_1} for $T= T_{g_1^\pm},T_{g_2^-}$.

It remains to show \eqref{lemma_2_4_proof_1} for $T=T_{g_2^+}$. To this end, as above it is enough to check
\begin{align}
\label{lemma_2_4_proof_3}
|g_2^+(x,y)-g_2^+(x,x_0)|\lesssim \frac{|x_0-y|}{\ds \min_\pm |x\pm x_0|^2}
\end{align}
for $x\in I_*^c$ and $y\in I$. Let us compute
\begin{align*}
|g_2^+(x,y)-g_2^+(x,x_0)|
&=|g_{a,-a}^+(x,y)\sgn y-g_{a,-a}^+(x,x_0)\sgn x_0|\\
&\le |g_{a,-a}^+(x,y)-g_{a,-a}^+(x,x_0)|+|g_{a,-a}^+(x,x_0)(\sgn x_0-\sgn y)|,
\end{align*}
where the first term is dominated by $C(\min_\pm |x\pm x_0|)^{-2}|x_0-y|$ as above. For the second term, we further calculate
\begin{align*}
&g_{a,-a}^+(x,x_0)(\sgn x_0-\sgn y)\\
&=a(\sgn x_0-\sgn y)\left\{\frac{1}{|x|+|x_0|}+\frac{1}{|x|-|x_0|}-\frac{1}{|x|+i|x_0|}-\frac{1}{|x|-i|x_0|}\right\}\\
&=(i-1)a(\sgn x_0-\sgn y)\left\{\frac{|x_0|}{(|x|+|x_0|)(|x|+i|x_0|)}-\frac{|x_0|}{(|x|-|x_0|)(|x|-i|x_0|)}\right\}\\
&=(i-1)a(x_0-|x_0|\sgn y)\left\{\frac{1}{(|x|+|x_0|)(|x|+i|x_0|)}-\frac{1}{(|x|-|x_0|)(|x|-i|x_0|)}\right\},
\end{align*}
where we have
$$
|x_0-|x_0|\sgn y|=|x_0-y+(|y|-|x_0|)\sgn y|\le 2|x_0-y|
$$
and, for $x\in I_*^c$,
$$
\min\{||x|\pm |x_0||,||x|\pm i|x_0||\}\ge \min_\pm |x\pm x_0|.
$$
Hence we have \eqref{lemma_2_4_proof_3}, so \eqref{lemma_2_4_proof_1} for $T=T_{g_2^+}$. This completes the proof.
\end{proof}

\section{Low energy estimate}
\label{section_low}
In this section we consider the low energy part of Theorem \ref{theorem_1}. Namely, we prove

\begin{theorem}
\label{theorem_low_1}
Under the assumption in Theorem \ref{theorem_1}, the low energy part $W_-^L$ defined by \eqref{W^{L}_-} satisfies the same statement as that in Theorem \ref{theorem_1}.
\end{theorem}

\subsection{Regular case}
\label{subsection_regular}
We first consider the regular case. Throughout this subsection, we thus always assume that $|V(x)|\lesssim \<x\>^{-\mu}$ with $\mu>15$ and zero is a regular point of $H$.

Substituting the expansion \eqref{lemma_3_3_1} into \eqref{lemma_3_2_1}, we obtain
\begin{align*}
R^+(\lambda^4)V
&=R_0^+(\lambda^4)v\Big\{Q_{2}A_{0}^0Q_{2}+\lambda Q_1A_{1}^0Q_1+\lambda^2\left(Q_1A_{21}^0Q_1+Q_{2}A_{22}^0+A_{23}^0Q_{2}\right)\\
&\quad +\lambda^3\left(Q_1A_{31}^0+A_{32}^0Q_1\right)+\lambda^3\widetilde P+\Gamma_4^0(\lambda)\Big\}v.
\end{align*}
Then $W^{L}_-$ can be written as follows:
\begin{align}
\label{wave_operator_regular}
W^{L}_-=T_{K^0_{0}}+T_{K^0_{1}}+T_{K^0_{21}}+T_{K^0_{22}}+T_{K^0_{23}}+T_{K^0_{31}}+T_{K^0_{32}}+T_{K^0_{33}}+T_{K^0_{4}}
\end{align}
with the integral kernels
\begin{align*}
K^0_{0}(x,y)&=\int_0^\infty \lambda^3\chi(\lambda)\Big(R_0^+(\lambda^4)vQ_{2}A_{0}^0Q_{2}v[R_0^+-R_0^-](\lambda^4)\Big)(x,y)d\lambda,\\
K^0_{1}(x,y)&=\int_0^\infty \lambda^4\chi(\lambda)\Big(R_0^+(\lambda^4)vQ_1A_{1}^0Q_1v[R_0^+-R_0^-](\lambda^4)\Big)(x,y)d\lambda,\\
K^0_{kj}(x,y)&=\int_0^\infty \lambda^{3+k} \chi(\lambda)\Big(R_0^+(\lambda^4)vB^0_{kj}v[R_0^+-R_0^-](\lambda^4)\Big)(x,y)d\lambda,\\
K^0_{4}(x,y)&=\int_0^\infty \lambda^3 \chi(\lambda)\Big(R_0^+(\lambda^4)v\Gamma_4^0(\lambda)v[R_0^+-R_0^-](\lambda^4)\Big)(x,y)d\lambda,
\end{align*}
where $j=1,2,3$, $k=2,3$ and $B^0_{kj}$ are given by
\begin{itemize}
\item $B^0_{21}=Q_1A_{21}^0Q_1$, $B^0_{22}=Q_{2}A_{22}^0$ and $B^0_{23}=A_{23}^0Q_{2}$;
 \vskip0.2cm
\item $B^0_{31}=Q_1A_{31}^0$, $B^0_{32}=A_{32}^0Q_1$ and $B^0_{33}=\widetilde P$.
\end{itemize}
By virtue of this formula for $W^{L}_-$, Theorem \ref{theorem_low_1} for the regular case follows from the corresponding boundedness of these nine integral operators.  Note that, since $|v(x)|\lesssim \<x\>^{-\mu/2}$ with $\mu>15$ by the assumption on $V$, we have
$$
\norm{\<x\>^{k}vBv\<x\>^k f}_{L^1}\le \norm{\<x\>^kv}_{L^2}^2\norm{B}_{L^2\to L^2}\norm{f}_{L^\infty}\lesssim \norm{\<x\>^{2k}V}_{L^1}\norm{f}_{L^\infty}
$$
for all $B=Q_{2}A_{0}^0Q_{2},Q_1A_{1}^0Q_1,B^0_{kj},\Gamma^0_4(\lambda)$ and $k<(\mu-1)/2$.
Hence, in all cases, $\<x\>^kvBv\<x\>^k$ is an absolutely bounded integral operator for any $k\le 7$ at least, satisfying
\begin{align}
\label{vBv}
\int_{\R^2}\<x\>^k|(vBv)(x,y)|\<y\>^kdxdy\lesssim \norm{\<x\>^{2k}V}_{L^1}<\infty,
\end{align}
where, denoting the integral kernel of $B$ by $B(x,y)$, we use the notation $$(vBv)(x,y)=v(x)B(x,y)v(y).$$
By virtue of Remark \ref{remark_lemma_projection} (2), these nine operators $T_{K^0_j},T_{K^0_{kj}}$ are classified into the following two cases with respect to the order of $\lambda$ of the integrands of their kernels :
\begin{itemize}
\item[(I)] $O(\lambda)$: $T_{K^0_{0}},T_{K^0_{21}},T_{K^0_{22}},T_{K^0_{23}},T_{K^0_{31}},T_{K^0_{32}}$ and $T_{K^0_4}$.
 \vskip0.2cm
\item[(II)] $O(1)$: $T_{K^0_{1}}$ and $T_{K^0_{33}}$.
\end{itemize}
The class (I) is further decomposed into $T_{K^0_4}$ and otherwise.

\begin{remark}
Note that the two projections $Q_{2},Q_{2}^{0}$ will play completely the same role in the following arguments. Hence, in what follows, we do not  distinguish them and use the same notation $Q_{2}$ to denote these operators $Q_{2},Q_{2}^{0}$.
\end{remark}

We start dealing with  the operators in the class (I), except for the last one $T_{K^0_{4}}$, namely the operators $T_{K^0_{0}},T_{K^0_{21}},T_{K^0_{22}},T_{K^0_{23}},T_{K^0_{31}}$ and $T_{K^0_{32}}$.

\begin{proposition}
\label{proposition_regular_1}
Let $K\in \{K^0_{0},K^0_{21},K^0_{22},K^0_{23},K^0_{31},K^0_{32}\}$. Then
$T_{K}\in \mathbb B(L^p)$ for all $1\le p\le \infty$. Moreover, if $1<p<\infty$, $w_p\in A_p$ and $w_1\in A_1$, then \eqref{lemma_2_2_1} and  \eqref{lemma_2_2_2} also hold.
\end{proposition}

\begin{proof}
All the kernels $K^0_{0},K^0_{21},K^0_{22},K^0_{23},K^0_{31}$ and $K^0_{32}$ can be written in the form
\begin{align}
\label{G}
\int_0^\infty \lambda^{7-\alpha-\beta}\chi(\lambda)\Big(R_0^+(\lambda^4)vQ_\alpha BQ_\beta v[R_0^+-R_0^-](\lambda^4)\Big)(x,y)d\lambda
\end{align}
with some $B\in \mathbb B(L^2)$ so that $Q_\alpha BQ_\beta $ is absolutely bounded, $Q_0:=1$ and $(\alpha,\beta)$ is give by
\begin{align*}
(\alpha,\beta)&=\begin{cases}(2,2)&\text{for}\quad K=K^0_{0},\\(1,1)&\text{for}\quad K=K^0_{21},\\(2,0)&\text{for}\quad K=K^0_{22},\end{cases}\
(\alpha,\beta)=\begin{cases}(0,2)&\text{for}\quad K=K^0_{23},\\(1,0)&\text{for}\quad K=K^0_{31},\\(0,1)&\text{for}\quad K=K^0_{32}.\end{cases}
\end{align*}
Let $G^0_{\alpha\beta}(x,y)$ be the function given by \eqref{G}. Then we shall show $T_{G^0_{\alpha\beta}}$ satisfies the desired assertion for any $\alpha,\beta\ge0$. To this end, by Lemmas \ref{lemma_2_1} and \ref{lemma_2_2}, it is enough to show that
\begin{align}
\label{proposition_regular_1_proof_1}
|G^0_{\alpha\beta}(x,y)|\lesssim \<|x|-|y|\>^{-2},\quad x,y\in \R.
\end{align}
We consider three cases (i) $\alpha,\beta\neq0$, (ii) $\beta=0$ and (iii) $\alpha=0$ separately.

{\it Case (i)}. We first suppose $\alpha,\beta\neq0$ and rewrite $G^0_{\alpha\beta}$ as follows. Let
$$
\widetilde f_{\pm,\beta}(\lambda,x)=x^\beta v\int\int_0^1(1-\theta)^{\beta-1}\left(\sgn(y-\theta x)\right)^\beta F_\pm^{(\beta)}(\lambda|y-\theta x|)f(y)d\theta dy.$$
Then Lemma \ref{lemma_projection} and Remark \ref{remark_lemma_projection} (1) imply that
\begin{align}
\nonumber
&\lambda^{6-\alpha-\beta}[R_0^+(\lambda^4)vQ_\alpha BQ_\beta v[R_0^+-R_0^-](\lambda^4) f](x)\\
\nonumber
&=C_\beta \lambda^{3-\alpha}(R_0^+(\lambda^4)vQ_\alpha BQ_\beta [\widetilde f_{+,\beta}-\widetilde f_{-,\beta}])(x)\\
\nonumber
&=C_\alpha C_\beta\int\int_0^1(1-\theta_1)^{\alpha-1}(\sgn (X_1))^\alpha F_+^{(\alpha)}(\lambda|X_1|)u_1^\alpha v(u_1)Q_\alpha BQ_\beta [\widetilde f_+-\widetilde f_-](u_1)d\theta du_1\\
\nonumber
&=\int\left(\int_{\R^2\times[0,1]^2}M_{\alpha\beta}(X_1,Y_2,\Theta)F_+^{(\alpha)}(\lambda|X_1|)[F_+^{(\beta)}-F_-^{(\beta)}](\lambda|Y_2|) d\Theta \right)f(y) dy\\
\label{proposition_regular_1_proof_2}
&=\int\left(\int_{\R^2\times[0,1]^2}M_{\alpha\beta}(X_1,Y_2,\Theta)f_{\alpha \beta}(\lambda,X_1,Y_2) d\Theta \right)f(y) dy,
\end{align}
where we set $C_\alpha=(-1)^\alpha/(4\cdot(\alpha-1)!)$, $f_{\alpha\beta}$ is defined in Lemma \ref{lemma_FF}  and $$\Theta=(u_1,u_2,\theta_1,\theta_2),\quad X_1=x-\theta_1u_1,\quad Y_2=y-\theta_2u_2,$$and $M_{\alpha\beta}(x,y,\Theta)$ is defined by
\begin{align}
\label{M_alphabeta}
M_{\alpha\beta}(x,y,\Theta)=\frac{(-1)^{\alpha+\beta}(1-\theta_1)^{\alpha-1}(1-\theta_2)^{\beta-1}(\sgn x)^\alpha (\sgn y)^\beta u_1^\alpha u_2^\beta (vQ_\alpha  BQ_\beta  v)(u_1,u_2)}{16(\alpha-1)!(\beta-1)!} .
\end{align}
Substituting this formula into \eqref{G}, we obtain
$$
G^0_{\alpha\beta}(x,y)=\int_0^\infty \lambda \chi(\lambda)\left(\int_{\R^2\times[0,1]^2}M_{\alpha\beta}(X_1,Y_2,\Theta)f_{\alpha \beta}(\lambda,X_1,Y_2) d\Theta \right)d\lambda.
$$
It follows from Lemma \ref{lemma_FF} that $f_{\alpha\beta}(\lambda,X_1,Y_2)$ is given by a linear combination of $
e^{i\lambda(|X_1|\pm |Y_2|)}$ and $e^{-\lambda(|X_1|\pm i|Y_2|)}
$ for any $\alpha,\beta$ (not only for the case (i)). Let
\begin{equation}
\begin{aligned}
\label{proposition_regular_1_proof_3}
\Phi^\pm_1(x,y,\Theta)&=i(|X_1|-|x|)\pm i(|Y_2|-|y|),\\
\Phi^\pm_2(x,y,\Theta)&=-|X_1|+|x|\mp i(|Y_2|-|y|).
\end{aligned}
\end{equation}
Then $e^{i\lambda(|X_1|\pm |Y_2|)}=e^{i\lambda(|x|\pm|y|)}e^{\lambda\Phi^\pm_1(x,y,\Theta)}$ and $
e^{-\lambda(|X_1|\pm i|Y_2|)}=e^{-\lambda(|x|\pm i|y|)}e^{\lambda\Phi^\pm_2(x,y,\Theta)}$.
Define
\begin{align}
\label{proposition_regular_1_proof_4}
a_{j}^\pm(\lambda,x,y)&=\int_{\R^2\times[0,1]^2}e^{\lambda\Phi^\pm_j(x,y,\Theta)}M_{\alpha\beta}(X_1,Y_2,\Theta)d\Theta,\\
\nonumber
K^\pm_{a_{1}}(x,y)&=\int_0^\infty  e^{i\lambda(|x|\pm|y|)}\lambda \chi(\lambda)a_{1}^\pm(\lambda,x,y)d\lambda,\\
\nonumber
K^\pm_{a_{2}}(x,y)&=\int_0^\infty  e^{-\lambda(|x|\pm i|y|)}\lambda \chi(\lambda)a_{2}^\pm(\lambda,x,y)d\lambda.
\end{align}
Then $G^0_{\alpha\beta}$ can be written as a linear combination of $K^\pm_{a_1}$ and $K^\pm_{a_2}$.

Here we summarize several properties of $M_{\alpha\beta}$, $\Phi^\pm_j$ and $a^\pm_{j}$ needed in the proof:
\begin{itemize}
\item By \eqref{vBv}, $\<u_1\>^\ell M_{\alpha\beta}(x,y,\Theta)\<u_2\>^\ell \in L^1(\R^2\times[0,1]^2;L^\infty(\R^2_{x,y}))$ for $\ell=0,1,2$ and
\begin{align}
\label{proposition_regular_1_proof_5}
\int_{\R^2\times[0,1]^2}\sup_{x,y\in \R^2}\<u_1\>^\ell |M_{\alpha\beta}(x,y,\Theta)|\<u_2\>^\ell d\Theta\lesssim \norm{\<x\>^{2(\alpha+\beta+\ell)}V}_{L^1}.
\end{align}
\item By the triangle inequality, for all $x,y\in \R$, $\lambda\ge0$ and $\Theta\in \R^2\times[0,1]^2$,
\begin{align}
\label{proposition_regular_1_proof_6}
|e^{\lambda\Phi^\pm_1(x,y,\Theta)}|\le1,\quad |e^{\lambda\Phi^\pm_2(x,y,\Theta)}|\le e^{\lambda |x|},\quad |\Phi^\pm_j(x,y,\Theta)|\le |u_1|+|u_2|.
\end{align}
\item By \eqref{proposition_regular_1_proof_5} and \eqref{proposition_regular_1_proof_6}, $a_{j}^\pm$ are smooth in $\lambda$, satisfying
\begin{align}
\label{proposition_regular_1_proof_7}
|\partial_\lambda^\ell a_{1}^\pm(\lambda,x,y)|+e^{-\lambda|x|}|\partial_\lambda^\ell a_{2}^\pm(\lambda,x,y)|&\lesssim \norm{\<x\>^{4+2\ell}V}_{L^1}
\end{align}
uniformly in $x,y\in \R$ and $\lambda\ge0$, at least for $\ell\le 2$.
\end{itemize}
Since $\chi\in C_0^\infty(\R)$, $K^\pm_{a_1}$ and $K^\pm_{a_2}$ are bounded on $\R^2$. In particular,
$$
|K^\pm_{a_1}(x,y)|+|K^\pm_{a_2}(x,y)|\lesssim1\lesssim \<|x|-|y|\>^{-2}
$$
if $||x|-|y||\le1$. Next, when $||x|-|y||\ge1$, we apply integration by parts twice to compute
\begin{align*}
K^\pm_{a_1}(x,y)
&=-\frac{1}{i(|x|\pm|y|)}\int_0^\infty  e^{i\lambda(|x|\pm|y|)} \left(\chi a_{1}^\pm+\lambda \partial_\lambda(\chi a_{1}^\pm)\right)(\lambda,x,y)d\lambda\\
&=-\frac{a_{1}^\pm(0,x,y)}{(|x|\pm|y|)^2}-\frac{1}{(|x|\pm|y|)^2}\int_0^\infty  e^{i\lambda(|x|\pm|y|)}\left(2\partial_\lambda(\chi a_{1}^\pm)+\lambda\partial_\lambda^2(\chi a_{1}^\pm)\right)d\lambda\\
&=O(\<|x|-|y|\>^{-2}).
\end{align*}
Similarly, it follows from \eqref{proposition_regular_1_proof_7} and the integration by parts that
$$
|K^\pm_{a_2}(x,y)|\lesssim \<|x|-|y|\>^{-2}
$$
Therefore, we have \eqref{proposition_regular_1_proof_1} for the case $\alpha,\beta\neq0$.

{\it Case (ii)}. Let $\beta=0$, $\alpha\neq0$. As in the case (i), it follows from \eqref{free_resolvent} and Lemma \ref{lemma_projection} that
\begin{align*}
G^0_{\alpha0}(x,y)=\int_0^\infty \lambda\chi(\lambda) \left(\int_{\R^2\times[0,1]}M_{\alpha0}(X_1,\Theta_1)f_{\alpha 0}(\lambda,X_1,y-u_2) d\Theta_1 \right)d\lambda
\end{align*}
where $\Theta_1=(u_1,u_2,\theta_1)$, $X_1=x-\theta_1u_1$ and
\begin{align}
\label{M_alpha0}
M_{\alpha0}(x,\Theta_1)=\frac{(-1)^\alpha }{16(\alpha-1)!}(1-\theta_1)^{\alpha-1}(\sgn x)^\alpha u_1^\alpha (vQ_\alpha BQ_{0}v)(u_1,u_2).
\end{align}
Define $\widetilde a_{j}^\pm(\lambda,x)$ by
$$
\widetilde a_{j}^\pm(\lambda,x)=\int_{\R^2\times[0,1]}e^{\lambda\Phi^\pm_j(x,y,\Theta_1)}M_{\alpha0}(X_1,\Theta_1)d\Theta_1.
$$
Then $M_{\alpha0}$ and $\widetilde a_{j}^\pm$ satisfy the same estimates as \eqref{proposition_regular_1_proof_5} and \eqref{proposition_regular_1_proof_7} for $M_{\alpha\beta}$ and $a_{j}^\pm$, respectively. Moreover, $G^0_{\alpha 0}$ is given by a linear combination of the following four functions
$$
\int_0^\infty  e^{i\lambda(|x|\pm|y|)}\lambda \chi(\lambda)\widetilde a_{1}^\pm(\lambda,x)d\lambda,\quad
\int_0^\infty  e^{-\lambda(|x|\pm i|y|)}\lambda \chi(\lambda)\widetilde a_{2}^\pm(\lambda,x)d\lambda.
$$
Hence, it can be shown by the same argument as in the case (i) that $G^0_{\alpha0}$ also satisfies \eqref{proposition_regular_1_proof_1}.

{\it Case (iii)}. Let $\alpha=0$, $\beta\neq0$. Again,  it follows from \eqref{free_resolvent} and Lemma \ref{lemma_projection} that
$$
G^0_{0\beta}(x,y)=\int_0^\infty \lambda\chi(\lambda) \left(\int_{\R^2\times[0,1]}M_{0\beta}(Y_2,\Theta_2)f_{0\beta}(\lambda,x-u_1,Y_2) d\Theta_2 \right)d\lambda,
$$
where $\Theta_2=(u_1,u_2,\theta_2)$, $Y_2=y-\theta_2u_2$ and
\begin{align}
\label{M_0beta}
M_{0\beta}(y,\Theta_2)=\frac{(-1)^\beta}{16(\beta-1)!}(1-\theta_2)^{\beta-1}(\sgn y)^\beta u_2^\beta (vQ_0 BQ_{\beta}v)(u_1,u_2).
\end{align}
Then the same argument as above  implies \eqref{proposition_regular_1_proof_1}. This completes the proof.
\end{proof}

Next we consider the remaining term $T_{K^0_{4}}$ in the class (I).
\begin{proposition}
\label{proposition_regular_2}
$T_{K^0_{4}}$ satisfies the same statement as that in Proposition \ref{proposition_regular_1}.
\end{proposition}

\begin{proof}
We show $K^0_{4}=O(\<|x|-|y|\>^{-3/2})$ which, together with Lemmas \ref{lemma_2_1} and \ref{lemma_2_2}, implies the assertion. The proof is more involved than in the previous case since $\Gamma^0_4$ depends on $\lambda$.

A similar computation as before based on Lemma \ref{lemma_projection} implies
\begin{align*}
K^0_{4}(x,y)
&=\int_0^\infty \lambda^3 \chi(\lambda)\Big(R_0^+(\lambda^4)v\Gamma_4^0(\lambda)v[R_0^+-R_0^-](\lambda^4)\Big)(x,y)d\lambda\\
&=\int_0^\infty \lambda \chi(\lambda)\left(\int_{\R^2}\widetilde \Gamma(\lambda,u_1,u_2)f_{00}(\lambda,x-u_1,y-u_2)du_1du_2\right)d\lambda,
\end{align*}
where we set $\widetilde \Gamma(\lambda,u_1,u_2)=\frac{1}{16\lambda^4}(v\Gamma_4^0(\lambda)v)(u_1,u_2)$ for short and recall (see Lemma \ref{lemma_FF}) that
\begin{align*}
f_{00}(\lambda,x-u_1,y-u_2)
=-\sum_{\pm}\left(e^{i\lambda(|x-u_1|\pm |y-u_2|)}+ie^{-\lambda(|x-u_1|\pm i|y-u_2|)}\right).
\end{align*}
Let $\Phi^\pm_j$ be defined by \eqref{proposition_regular_1_proof_3} and
\begin{align*}
b^\pm_j(\lambda,x,y)&=\int_{\R^2}e^{\lambda \Phi^\pm_j(x,y,u_1,u_2,1,1)}\widetilde \Gamma(\lambda,u_1,u_2)du_1du_2,\\
K^\pm_{b_{1}}(x,y)&=-\int_0^\infty  e^{i\lambda(|x|\pm|y|)}\lambda \chi(\lambda)b_{1}^\pm(\lambda,x,y)d\lambda,\\
\nonumber
K^\pm_{b_{2}}(x,y)&=-i\int_0^\infty  e^{-\lambda(|x|\pm i|y|)}\lambda \chi(\lambda)b_{2}^\pm(\lambda,x,y)d\lambda.
\end{align*}
Then, as before, $K^0_{4}=K^+_{b_{1}}+K^-_{b_{1}}+K^+_{b_{2}}+K^-_{b_{2}}$. By virtue of \eqref{lemma_3_3_4}, the bound $|v(x)|\lesssim \<x\>^{-\mu/2}$ with $\mu>15$ and \eqref{proposition_regular_1_proof_6}, $b^\pm_j(\lambda,x,y)$ satisfy
\begin{align}
\label{proposition_regular_2_proof_1}
|\partial_\lambda^{\ell}b^\pm_1(\lambda,x,y)|+e^{-\lambda|x|}|\partial_\lambda^{\ell}b^\pm_2(\lambda,x,y)|\lesssim \norm{\<x\>^{2\ell}V}_{L^1}\lambda^{-\ell}
\end{align}
for $\lambda>0$, $x,y\in \R$ and $\ell=0,1,2$. To deal with a possible singularity of $\partial_\lambda b_j^\pm$ in $\lambda\ll1$, we decompose $\chi$ by using the dyadic partition of unity $\{\varphi_N\}$ defined in Subsection \ref{subsection_notation}, as
$$
\chi(\lambda)=\sum_{N=-\infty}^{N_0}\widetilde \chi_N(\lambda),\quad \widetilde \chi_N(\lambda):=\chi(\lambda)\varphi_N(\lambda),\quad \lambda>0,
$$
where $N_0\lesssim |\log\lambda_0|\lesssim1$ since $\supp \chi\subset [0,\lambda_0]$. Note that $\supp\widetilde\chi_N\subset [2^{N-2},2^N]$ and
\begin{align}
\label{proposition_regular_2_proof_2}
|\partial_\lambda^\ell \widetilde \chi_N(\lambda)|\le C_\ell 2^{-N\ell}
\end{align}
for all $\ell$. Let $K^{\pm}_{b_j,N}$ is given by $K^{\pm}_{b_j}$ with $\chi$ replaced by $\widetilde \chi_N$ and decompose $K^\pm_{b_j}$ as
$$
K^\pm_{b_j}=\sum_{N\le N_0}K^{\pm}_{b_j,N}.
$$
Since $\lambda\sim 2^N$ on $\supp \widetilde \chi_N$, we know by  \eqref{proposition_regular_2_proof_1} that
$$
|K^{\pm}_{b_j,N}(x,y)|\lesssim 2^N\int_{\supp \widetilde \chi_N}d\lambda\lesssim 2^{2N},\quad x,y\in \R.
$$
In particular, if $||x|-|y||\le1$ then
$$
|K^{\pm}_{b_j,N}(x,y)|\lesssim 2^{2N}\<|x|-|y|\>^{-2}.
$$
On the other hand, when $||x|-|y||>1$,  we obtain by integrating by parts twice that
\begin{align*}
K^{\pm}_{b_1,N}(x,y)=-\frac{1}{(|x|\pm |y|)^2}\int_0^\infty  e^{i\lambda(|x|\pm |y|)}\Big[2\partial_\lambda (\widetilde \chi_Nb^\pm_1)+\lambda\partial_\lambda^2(\widetilde \chi_Nb^\pm_1)\Big]d\lambda
\end{align*}
since $\widetilde\chi_N(0)=0$. Then \eqref{proposition_regular_2_proof_1}, \eqref{proposition_regular_2_proof_2} and the support property of $\widetilde \chi_N$ imply
\begin{align*}
|K^{\pm}_{b_1,N}(x,y)|\lesssim \<|x|-|y|\>^{-2}2^{-N}\int_{2^{N-2}}^{2^N}d\lambda\lesssim \<|x|-|y|\>^{-2}
\end{align*}
if $||x|-|y||>1$. Therefore, $K^{\pm}_{b_1,N}(x,y)$ satisfies
$$
|K^{\pm}_{b_1,N}(x,y)|\lesssim \min\{2^{2N},\<|x|-|y|\>^{-2}\}\lesssim 2^{2N(1-\theta)}\<|x|-|y|\>^{-2\theta},\quad \theta\in [0,1],
$$
uniformly in $N\le N_0$, $x,y\in \R$. In particular, taking $\theta=3/4$ for instance, we obtain
\begin{align*}
|K^{\pm}_{b_1}(x,y)|\lesssim \<|x|-|y|\>^{-3/2}\sum_{N\le N_0}2^{N/2}\lesssim \<|x|-|y|\>^{-3/2}.
\end{align*}
It follows similarly from \eqref{proposition_regular_2_proof_1}, \eqref{proposition_regular_2_proof_2} and the support property of $\widetilde \chi_N$ that
\begin{align*}
|K^{\pm}_{b_2}(x,y)|\lesssim \<|x|+|y|\>^{-3/2}.
\end{align*}
Therefore, $K^0_4(x,y)=O( \<|x|-|y|\>^{-3/2})$ and the result follows by Lemmas \ref{lemma_2_1} and \ref{lemma_2_2}.
\end{proof}

Next we deal with the class (II), namely $T_{K^0_1}$ and $T_{K^0_{33}}$. We begin with $T_{K^0_{33}}$.

\begin{proposition}
\label{proposition_regular_3}
If $1<p<\infty$, $w_p\in A_p$ and $w_1\in A_1$ then $T_{K^0_{33}}$ and $T_{K^0_{33}}^*$ satisfy the same bounds as \eqref{lemma_2_2_1} and  \eqref{lemma_2_2_2}. Moreover, $T_{K^0_{33}},T_{K^0_{33}}^*\in \mathbb B(\H^1(\R),L^1(\R))\cap \mathbb B(L^\infty(\R),\BMO(\R))$.
\end{proposition}

\begin{proof}We shall show that $K^0_{33}$ is written in the form
\begin{align}
\label{proposition_regular_3_proof_1}
K^0_{33}(x,y)=\frac{-1+i}{8} g_{1}^+(x,y)+O(\<|x|-|y|\>^{-2}),
\end{align}
with $g_{1}^+=g_{1,1,1}^+$ defined in Lemma \ref{lemma_2_4} with the choice of $a=b=1$. Then Lemmas \ref{lemma_2_1}--\ref{lemma_2_4} apply to $T_{K^0_{33}}$, yielding the desired assertion. As before, using \eqref{free_resolvent} and Lemma \ref{lemma_FF}, we have
\begin{align*}
K^0_{33}(x,y)
&=\int_0^\infty \lambda^6 \chi(\lambda)\Big(R_0^+(\lambda^4)v\widetilde Pv[R_0^+-R_0^-](\lambda^4)\Big)(x,y)d\lambda\\
&=\frac{1}{16}\int_0^\infty \chi(\lambda) \left(\int_{\R^2} (v\widetilde Pv)(u_1,u_2)f_{00}(\lambda, x-u_1,y-u_2)du_1du_2\right)d\lambda\\
&=K^+_{33,1}(x,y)+K^-_{33,1}(x,y)+K^+_{33,2}(x,y)+K^-_{33,2}(x,y),
\end{align*}
where, using $\Phi^\pm_j$ defined by \eqref{proposition_regular_1_proof_4}, we set
\begin{align*}
c_{j}^\pm(\lambda,x,y)&=\frac{1}{16}\int_{\R^2}e^{\lambda\Phi^\pm_j(x,y,u_1,u_2,1,1)}(v\widetilde Pv)(u_1,u_2)du_1du_2,\\
K^\pm_{33,1}(x,y)&=-\int_0^\infty e^{i\lambda(|x|\pm|y|)}\chi(\lambda)c_{1}^\pm(\lambda,x,y)d\lambda,\\
K^\pm_{33,2}(x,y)&=-i\int_0^\infty e^{-\lambda(|x|\pm i|y|)}\chi(\lambda)c_{2}^\pm(\lambda,x,y)d\lambda.
\end{align*}
By \eqref{vBv} and \eqref{proposition_regular_1_proof_6}, $c_{j}^\pm$ satisfy a similar estimates as that for $a_{j}^\pm$: for $x,y\in \R$, $\lambda\ge0$,
\begin{align*}
|\partial_\lambda^\ell c_{1}^\pm(\lambda,x,y)|+e^{-\lambda|x|}|\partial_\lambda^\ell c_{2}^\pm(\lambda,x,y)|&\lesssim \norm{\<x\>^{2\ell}V}_{L^1}\lesssim 1,\quad \ell=0,1,2.
\end{align*}
Hence, since $\chi\in C_0^\infty$, $K^0_{33}$ is bounded on $\R^2$. Let $\psi(||x|\pm|y||^{2})$ be defined in Lemma \ref{lemma_2_3} such that $\psi(||x|\pm|y||^{2})$ is supported away from the region $\{|x|\pm|y|\le1\}$. We decompose
$$
K^\pm_{33,1}=\psi_\pm K^\pm_{33,1}+\left(1-\psi_\pm\right)K^\pm_{33,1},
$$
where $\psi_\pm:=\psi(||x|\pm|y||^{2})$ for short. The second part of the right hand side satisfies
$$
|\left(1-\psi_\pm\right)K^\pm_{33,1}(x,y)|\lesssim 1\lesssim \<|x|-|y|\>^{-2},\quad x,y\in \R.
$$
To estimate the first term, we recall that $\widetilde P=\frac{-2(1+i)}{\norm{V}_{L^1}^2}\<\cdot,v\>v$ and hence, for all $x,y,j$,
$$
c^\pm_{j}(0,x,y)=\frac{1}{16}\int_{\R^2}(v\widetilde Pv)(u_1,u_2)du_1du_2=-\frac{1+i}{8}.
$$
Then we obtain by integration by parts twice that
\begin{align*}
\psi_\pm K^\pm_{33,1}(x,y)
&=\frac{-1+i}{8}\frac{\psi_\pm}{|x|\pm|y|}+\frac{\psi_\pm}{i(|x|\pm|y|)}\int_0^\infty e^{i\lambda(|x|\pm|y|)}\partial_\lambda(\chi c_{1}^\pm)(\lambda,x,y)d\lambda\\
&=\frac{-1+i}{8}\frac{\psi_\pm}{|x|\pm|y|}
+\frac{\psi_\pm\partial_\lambda(\chi c_{1}^\pm)(0,x,y)}{(|x|\pm|y|)^2}\\
&\quad\quad-\frac{\psi_\pm}{(|x|\pm|y|)^2}\int_0^\infty e^{i\lambda(|x|\pm|y|)}\partial_\lambda^2(\chi c_{1}^\pm)(\lambda,x,y)d\lambda\\
&=\frac{-1+i}{8}\frac{\psi_\pm}{|x|\pm|y|}+O(\<|x|-|y|\>^{-2}).
\end{align*}
Decomposing $K^\pm_{33,2}$ as $
K^\pm_{33,2}=\psi_- K^\pm_{33,2}+(1-\psi_-)K^\pm_{33,2},
$
we similarly have
$$
K^\pm_{33,2}(x,y)=\frac{-1+i}{8}\frac{\psi_-}{|x|\pm i|y|}+O(\<|x|-|y|\>^{-2}).
$$
Therefore, \eqref{proposition_regular_3_proof_1} follows. This completes the proof.
\end{proof}

It remains to deal with the most technical and delicate term $T_{K^0_{1}}$.

\begin{proposition}
\label{proposition_regular_4}
For any $1<p<\infty$ and $w_p\in A_p$,  $T_{K^0_{1}}$ and $T_{K^0_{1}}^*$ satisfy the same bound as  \eqref{lemma_2_2_1}. Moreover, $T_{K^0_{1}}$ satisfies the following statements:
\begin{itemize}
\item[(1)] If $V$ is compactly supported, then $T_{K^0_{1}},T_{K^0_{1}}^*\in \mathbb B(L^1(\R),L^{1,\infty}(\R))$;
\item[(2)] If $Q_1A_1^0Q_1$ is finite rank, then $T_{K^0_{1}}$ and $T_{K^0_{1}}^*$ satisfy the same bound as \eqref{lemma_2_2_2};
\item[(3)] $T_{K^0_{1}},T_{K^0_{1}}^*\in \mathbb B(\H^1(\R),L^1(\R))\cap \mathbb B(L^{\infty}(\R),\BMO(\R))$.
\end{itemize}
\end{proposition}


\begin{proof}[Proof of Proposition \ref{proposition_regular_4}]
The proof is divided into five steps.

\underline{{\it Step 1}}. We first derive a useful asymptotic formula of $K^0_{1}$. Using Lemma \ref{lemma_projection}, we compute
\begin{align}
\label{proposition_regular_4_proof_1}
\nonumber
K^0_{1}(x,y)
&=\int_0^\infty \lambda^4\chi(\lambda)\Big(R_0^+(\lambda^4)vQ_1A_{1}^0Q_1v[R_0^+-R_0^-](\lambda^4)\Big)(x,y)d\lambda\\
&=\int_0^\infty \chi(\lambda)\left(\int_{\R^2\times [0,1]^2} M_{11}(X_1,Y_2,\Theta)f_{11}(\lambda,X_1,Y_2)d\Theta\right)d\lambda,
\end{align}
where $\Theta=(u_1,u_2,\theta_2,\theta_2)\in\R^2\times[0,1]^2$, $
X_1=x-\theta_1u_1$, $Y_2=y-\theta_2u_2,
$
and $M_{11}$ is given by \eqref{M_alphabeta} with $B=A^0_1$. Since $f_{11}$ is bounded on $\R_+\times \R^2$, $\chi\in C_0^\infty$ and $M_{\alpha\beta}$ satisfies \eqref{proposition_regular_1_proof_5},
 $K^0_1$ is absolutely convergent and bounded on $\R^2$. Fubini's theorem then yields
\begin{align}
\nonumber
K^0_{1}(x,y)
&=\int_{\R\times [0,1]}\left(\int_0^\infty \chi(\lambda)\left(\int_{\R\times [0,1]} M_{11}(X_1,Y_2,\Theta)f_{11}(\lambda,X_1,Y_2)du_2d\theta_2\right)d\lambda \right)du_1d\theta_1\\
\label{proposition_regular_4_proof_1_1}
&=\int_{\R\times [0,1]}\widetilde K^0_1(x,y;\theta_1,u_1)du_1d\theta_1,
\end{align}
where
\begin{align}
\label{proposition_regular_4_proof_2}
\widetilde K^0_1(x,y;\theta_1,u_1)=\int_0^\infty \chi(\lambda)\left(\int_{\R\times [0,1]} M_{11}(X_1,Y_2,\Theta)f_{11}(\lambda,X_1,Y_2)du_2d\theta_2\right)d\lambda.
\end{align}
Now we shall show that $\widetilde K^0_1$ is of the form
\begin{align}
\label{proposition_regular_4_proof_3}
\widetilde K^0_1(x,y;\theta_1,u_1)
=\sgn(X_1)g_1^-(x,y) m_{1}(y,u_1,\theta_1)+O\left(\<|x|-|y|\>^{-2}\rho_3 (u_1)\right),
\end{align}
where $g_1^-=g_{1,i,1}^-$ is given in Lemma \ref{lemma_2_4} (with $a=i,b=1$) and $\widetilde{m}_{1},\rho_\ell$ are given by
\begin{align*}
 m_{1}(y,u_1,\theta_1)
&:=\int_{\R\times[0,1]}\frac{M_{11}(X_1,Y_2,\Theta)}{\sgn X_1}du_2d\theta_2\\
&=\frac{1}{16}\int_{\R\times[0,1]}(\sgn Y_2) u_1 u_2 (vQ_1A_1^0Q_1  v)(u_1,u_2)du_2d\theta_2,
\end{align*}
and, for $\ell=0,1,2,...$, $$
\rho_\ell (u_1):=\frac{1}{16}\<u_1\>^\ell \int_{\R}|(vQ_1A_1^0Q_1v)(u_1,u_2)|\<u_2\>^\ell du_2.
$$
Note that $| m_{1}(y,u_1,\theta_1)|\le \rho_1(u_1)$. To prove \eqref{proposition_regular_4_proof_3}, we set
\begin{align*}
d^\pm_j(\lambda,x,y;u_1,\theta_1)&=\int_{\R\times[0,1]}e^{\lambda \Phi^\pm_j(x,y,\Theta)}M_{11}(X_1,Y_2,\Theta)du_2d\theta_2,\\
K^\pm_{1,1}(x,y;u_1,\theta_1)&=\int_0^\infty e^{i\lambda(|x|\pm |y|)}\chi(\lambda)d^\pm_1(\lambda,x,y;u_1,\theta_1)d\lambda,\\
K^\pm_{1,2}(x,y;u_1,\theta_1)&=\int_0^\infty e^{-\lambda(|x|\pm i|y|)}\chi(\lambda)d^\pm_2(\lambda,x,y;u_1,\theta_1)d\lambda,
\end{align*}
where $\Phi^\pm_j$ are defined by \eqref{proposition_regular_1_proof_3}. It follows from \eqref{proposition_regular_4_proof_2} and Lemma \ref{lemma_FF} with $\alpha=\beta=1$ that
\begin{align*}
\widetilde K^0_1(x,y;u_1,\theta_1)=K^+_{1,1}-K^-_{1,1}+K^+_{1,2}-K^-_{1,2}.
\end{align*}
Moreover, since $d^\pm_j(0,x,y;u_1,\theta_1)=\sgn(X_1) m_{1}(y,u_1,\theta_1)$, \eqref{proposition_regular_1_proof_6} and \eqref{vBv} imply
\begin{align}
\label{proposition_regular_4_proof_4}
|\partial_\lambda^\ell d^\pm_1(\lambda,x,y;u_1,\theta_1)|+e^{-\lambda|x|}|\partial_\lambda^\ell d^\pm_2(\lambda,x,y;u_1,\theta_1)|\lesssim \rho_{\ell+1}(u_1)
\end{align}
uniformly in $\lambda\ge0$, $x,y\in \R$, $u_1\in \R$ and $\theta_1\in [0,1]$. Hence $K^\pm_{1,j}$ satisfy
\begin{align}
\label{proposition_regular_4_proof_5}
|K^\pm_{1,j}(x,y;u_1,\theta_1)|\lesssim \rho_{1}(u_1)
\end{align}
uniformly in $x,y,u_1$ and $\theta_1$. We next let $\psi_\pm=\psi(||x|\pm|y||^2)$ be as in Lemma \ref{lemma_2_3} and apply integration by parts twice to $\psi_\pm K^\pm_{1,1}$ as in the previous case, obtaining
\begin{align*}
\psi_\pm K^\pm_{1,1}
&=-\frac{\psi_\pm \sgn(X_1) m_{1}}{i(|x|\pm|y|)}-\frac{\psi_\pm}{i(|x|\pm|y|)}\int_0^\infty e^{i\lambda(|x|\pm |y|)}\partial_\lambda(\chi d^\pm_1)d\lambda\\
&=-\frac{\psi_\pm \sgn(X_1) m_{1}}{i(|x|\pm|y|)}-\frac{\psi_\pm \partial_\lambda(\chi d^\pm_1)|_{\lambda=0}}{(|x|\pm|y|)^2}-\frac{\psi_\pm}{(|x|\pm|y|^2)}\int_0^\infty e^{i\lambda(|x|\pm |y|)}\partial_\lambda^2(\chi d^\pm_1)d\lambda\\
&=-\frac{\psi_\pm \sgn(X_1) m_{1}}{i(|x|\pm|y|)}+O(\<|x|-|y|\>^{-2}\rho_{3}(u_1)).
\end{align*}
The same calculation and \eqref{proposition_regular_4_proof_4} also yield
\begin{align*}
\psi_- K^\pm_{1,2}=\frac{\psi_- \sgn(X_1) m_{1}}{|x|\pm i|y|}+O(\<|x|-|y|\>^{-2}\rho_{3}(u_1)).
\end{align*}
Moreover, since $1-\psi_\pm$ is supported in $\{|x|\pm|y|\le1\}$, we know by \eqref{proposition_regular_4_proof_5} that $(1-\psi_\pm)K^\pm_{1,1}$ and $(1-\psi_-)K^\pm_{1,2}$ are dominated by $\<|x|-|y|\>^{-2}\rho_{3}(u_1)$. Therefore, we have
\begin{align*}
\widetilde K^0_1
&\equiv\sgn(X_1) m_{1}(y,u_1,\theta_1)\left(-i\frac{\psi_+}{(|x|+|y|)}+i\frac{\psi_-}{(|x|-|y|)}-\frac{\psi_-}{(|x|+ i|y|)}+\frac{\psi_-}{(|x|-i|y|)}\right)\\
&\equiv\sgn(X_1) m_{1}(y,u_1,\theta_1)g_1^-(x,y)
\end{align*}
modulo the error term $O(\<|x|-|y|\>^{-2}\rho_{3}(u_1))$ and \eqref{proposition_regular_4_proof_3} thus follows.

\underline{{\it Step 2: Proof of \eqref{lemma_2_2_1}}}. Let $T_{u_1,\theta_1}=T_{\widetilde K^0_1(\cdot,\cdot,u_1,\theta_1)}$ be the integral operator with kernel $\widetilde K^0_1(x,y,u_1,\theta_1)$, where $u_1,\theta_1$ are considered as parameters. We apply Fubini's theorem and Minkowski's integral inequality (which holds for any $\sigma$-finite measures) to \eqref{proposition_regular_4_proof_1_1}, obtaining
\begin{align}
\label{proposition_regular_4_proof_2_1}
\norm{T_{K^0_1}f}_{L^p(w_p)}\lesssim  \int_{\R\times [0,1]}\norm{T_{u_1,\theta_1}f}_{L^p(w_p)}du_1d\theta_1,\quad 1\le p<\infty.
\end{align}
Thanks to \eqref{proposition_regular_4_proof_3}, the main term of $T_{u_1,\theta_1}$ is the composition $\sgn(X_1)T_{g_1^-} m_{1}$. Moreover, since $| m_{1}(y,u_1,\theta_1)|\le \rho_1(u_1)$,  the multiplication by $ m_{1}$ is bounded on $L^p(w_p)$ for any $1\le p<\infty$ with the operator norm at most $\rho_1(u_1)$. It thus  follows from Lemmas  \ref{lemma_2_2} and \ref{lemma_2_3} that
\begin{align*}
\norm{T_{u_1,\theta_1}f}_{L^p(w_p)}
&\lesssim [w_p]_{A_p}^{\max\{1,1/(p-1)\}}\rho_3(u_1)\left(\norm{f}_{L^p(w_p)}+\norm{\tau f}_{L^p(w_p)}\right),\quad 1<p<\infty,
\end{align*}
where $\tau f(x)=f(-x)$. Since $\rho_3(u_1)\in L^1(\R)$ by the assumption on $V$ and \eqref{vBv},  the desired the bound \eqref{lemma_2_2_1} for $T_{K^0_{1}}$ follow by applying this bound to \eqref{proposition_regular_4_proof_2_1}. By the same argument, we also obtain the same the bounds \eqref{lemma_2_2_1} for its adjoint $T_{K^0_1}^*$.

\underline{{\it Step 3: Proof of (1)}}. Suppose $\supp V\subset \{|x|\le r\}$ with some $r>0$. Integrating \eqref{proposition_regular_4_proof_3} over $(u_1,\theta_1)\in \R\times[0,1]$ and using \eqref{proposition_regular_4_proof_1_1}, we have
\begin{align}
\label{remark_regular_4_1}
K^0_1(x,y)=g_1^-(x,y)\widetilde{m}_{1}(x,y)+O(\<|x|-|y|\>^{-2}),
\end{align}
where $\widetilde{m}_{1}$ is given by
\begin{align}
\label{proposition_regular_4_proof_5_1}
\widetilde{m}_{1}(x,y)=\frac{1}{16}\int_{\R^2\times[0,1]^2}(\sgn X_1)(\sgn Y_2) u_1 u_2 (vQ_1A_1^0Q_1  v)(u_1,u_2)du_1du_2d\theta_1d\theta_2.
\end{align}
Hence it is enough to prove $T_{\widetilde{m}_{1}g_1^-}\in \mathbb B(L^1(\R),L^{1,\infty}(\R))$. We decompose it as
$$
T_{\widetilde{m}_{1}g_1^-}=\mathds1_{\{|x|\ge r+1\}}T_{\widetilde{m}_{1}g_1^-}+\mathds1_{\{|x|\le r+1\}}T_{\widetilde{m}_{1}g_1^-}.
$$
For the first part, since $|\theta_1u_1|\le r$ for $u_1\in \supp v=\supp V$ and $0\le \theta_1\le1$, we have
$
\sgn X_1=\sgn(x-\theta_1u_1)=\sgn x
$
if $|x|\ge r+1$,  and hence
$$
\mathds1_{\{|x|\ge r+1\}}T_{\widetilde{m}_{1}g_1^-}=\mathds1_{\{|x|\ge r+1\}}\cdot \sgn x\cdot T_{g_1^-}\cdot \widetilde{m}_{2},
$$
where
$$
\widetilde{m}_{2}(y):=(\sgn y)^{-1}\widetilde m_1=\frac{1}{16}\int_{\R^2\times[0,1]^2}(\sgn Y_2) u_1 u_2 (vQ_1A_1^0Q_1  v)(u_1,u_2)du_1du_2d\theta_1d\theta_2
$$
depends only on $y$ and is bounded on $\R$. Recalling that $g^-_1$ is a linear combination of $k_1^\pm$ and $k_2^\pm$, we thus obtain $\mathds1_{\{|x|\ge r+1\}}T_{\widetilde{m}_{1}g_1^-}\in \mathbb B(L^1(\R),L^{1,\infty}(\R))$ by Lemma \ref{lemma_2_3}.
In fact, the same weighted weak-type bound as \eqref{lemma_2_2_2} holds for $\mathds1_{\{|x|\ge r+1\}}T_{\widetilde{m}_{1}g_1^-}$. For the second term, we set
$$
E_\lambda=\{x\in \R\ |\ |\mathds1_{\{|x|\le r+1\}}T_{\widetilde{m}_{1}g_1^-}f(x)|>\lambda\}
$$
for $f\in L^1(\R)$. Since $\widetilde{m}_{1}g_1^-$ is bounded on $\R$, we obtain
$$
|\mathds1_{\{|x|\le r+1\}}T_{\widetilde{m}_{1}g_1^-}f(x)|\lesssim \|f\|_{L^1(\R)}
$$
We also have $|E_\lambda |\lesssim r$ thanks to the restriction $\mathds1_{\{|x|\le r+1\}}$. Thus,
$$
\|\mathds1_{\{|x|\le r+1\}}T_{\widetilde{m}_{1}g_1^-}f\|_{L^{1,\infty}(\R)}\lesssim \sup_{\lambda>0}\lambda |E_\lambda|\lesssim \|f\|_{L^1(\R)}.
$$
This completes the proof of the item (1).

\underline{{\it Step 4: Proof of (2)}}. Suppose $Q_1A_1^0Q_1$ is finite rank. The proof for this case is almost analogous to that for $\mathds1_{\{|x|\ge r+1\}}T_{\widetilde{m}_{1}g_1^-}$. Indeed, we can write
$$
(Q_1A_1^0Q_1 )(u_1,u_2)=\sum_{i,j=1}^N a_{ij}\varphi_i(u_1)\overline{\varphi_j(u_2)}
$$
with some $\varphi_j\in L^2(\R)$, $a_{ij}\in \C$ and $N<\infty$. With this expression, we can apply Fubini's theorem in \eqref{proposition_regular_4_proof_5_1} to compute the $(u_1,\theta_1)$-integral and $(u_2,\theta_2)$-integral separately, and obtain
$$
\widetilde{m}_{1}(x,y)=\sum_{i,j=1}^N a_{ij}c_i(x)\overline{c_j(y)},\quad c_i(x)=\frac14\int_{\R\times[0,1]}(\sgn X_1)u_1v(u_1)\varphi_i(u_1)du_1d\theta_1\in L^\infty(\R).
$$
Hence, the same argument as above yields the bound \eqref{lemma_2_2_2} for $T_{K_0^1}$ and $T_{K_0^1}^*$.

\underline{{\it Step 5: Proof of (3)}}. In order to prove the item (3), it is enough to show $T_{K^0_{1}},T_{K^0_{1}}^*\in \mathbb B(\H^1,L^1)$ by the duality. For that purpose, \eqref{proposition_regular_4_proof_3} is not useful since the multiplication by $ m_{1}$ does not leave $\H^1$ invariant since $f\in \H^1$ must satisfy $\int_\R f(x)dx=0$. Instead, we use a simple trick based on the translation invariance. Let
$$
k^0_1(x,y)=\sgn(x)\sgn(y) \int_0^\infty \chi(\lambda) f_{11}(\lambda,x,y)d\lambda.
$$
Then, recalling the formula \eqref{proposition_regular_4_proof_1}, we have
\begin{align}
\label{proposition_regular_4_proof_6}
K^0_1(x,y)=\frac{1}{16}\int_{\R^2\times[0,1]^2}u_1u_2(vQ_1A^0_1Q_1v)(u_1,u_2) k^0_1(x-\theta_1u_1,y-\theta_2u_2)d\Theta.
\end{align}
Since the $L^1$-norm and $\H^1$-norm are invariant under the translation $f\mapsto f(\cdot-u)$, assuming
 $T_{k^0_1}\in \mathbb B(\H^1,L^1)$, we obtain by the change of variables $x\mapsto x+\theta_1u_1$ and $y\mapsto y+\theta_2u_2$ that
\begin{align*}
\norm{T_{K^0_1}f}_{L^1}
&\le \int_{\R^2\times[0,1]^2}\<u_1\>\<u_2\>|(vQ_1A^0_1Q_1v)(u_1,u_2)| \norm{T_{k^0_1(\cdot-\theta_1u_1,\cdot-\theta_2u_2)}f}_{L^1}d\Theta\\
&\lesssim \int_{\R^2\times[0,1]^2}\<u_1\>\<u_2\>|(vQ_1A^0_1Q_1v)(u_1,u_2)| \norm{f(\cdot+\theta_2u_2}_{\H^1}d\Theta\\
&\lesssim \norm{f}_{\H^1}.
\end{align*}
The same argument also applies to $T_{K^0_1}^*$. It remains to show $T_{k^0_1},T_{k^0_1}^*\in \mathbb B(\H^1,L^1)$. By a similar argument as in the Step 1 based on Lemma \ref{lemma_FF}, one can obtain
\begin{align}
\label{proposition_regular_4_proof_7}
k^0_1(x,y)=g_4^-(x,y)+O(\<|x|-|y|\>^{-2}),
\end{align}
where  $g_4^-=g_{4,i,1}^-$ are defined in Lemma \ref{lemma_2_4} with the choice of $a=i$ and $b=1$.
Moreover, the kernel of $T_{k^0_1}^*$ is given by
$$
\overline{k^0_1(y,x)}=\overline{g_{4,i,1}^-(y,x)}+O(\<|x|-|y|\>^{-2})=g_{4,-i,i}^+(x,y)+O(\<|x|-|y|\>^{-2}).
$$
Therefore, Lemmas \ref{lemma_2_1} and \ref{lemma_2_4} imply $T_{k^0_1},T_{k^0_1}^*\in \mathbb B(\H^1,L^1)$. This completes the proof.
\end{proof}


By Propositions \ref{proposition_regular_1}--\ref{proposition_regular_4} and \eqref{wave_operator_regular}, we have obtained for the regular case and $1<p<\infty$,
$$
W_-^L, (W_-^L)^*\in \mathbb B(L^p(\R))\cap \mathbb B(\H^1(\R),L^1(\R))\cap \mathbb B(L^\infty (\R),\BMO(\R))
$$
as well as the weighted estimate
$$
\norm{W_-^L f}_{L^p(w_p)}+\norm{(W_-^L)^* f}_{L^p(w_p)}\lesssim [w_p]_{A_p}^{\max\{1,1/(p-1)\}}\left(\norm{f}_{L^p(w_p)}+\norm{\tau f}_{L^p(w_p)}\right)
$$
We have also proved $W_\pm,W_\pm^*\in \mathbb B(L^1(\R),L^{1,\infty}(\R))$ if $V$ is compactly supported,  and
$$
\norm{W_-^L f}_{L^{1,\infty}(w_1)}+\norm{(W_-^L )^*f}_{L^{1,\infty}(w_1)}\lesssim [w_1]_{A_1}(1+\log [w]_{A_1})(\norm{f}_{L^1(w_1)}+\norm{\tau f}_{L^1(w_1)})
$$
if $Q_1A_1^0Q_1$ is finite rank. This completes the proof of Theorem \ref{theorem_low_1} for the regular case.

\subsection{First kind resonant case}
Next we consider the case when zero is a first kind resonance of $H$ and $|V(x)|\lesssim \<x\>^{-\mu}$ with $\mu>21$. 
By \eqref{W^{L}_-_2} and \eqref{lemma_3_3_2}, $W^{L}_-$ is of the form
\begin{align}
\label{wave_operator_first}
W^{L}_-=T_{K^1_{-1}}+\sum_{j=1}^2T_{K^1_{0j}}+\sum_{j=1}^3T_{K^1_{1j}}+\sum_{j=1}^2T_{K^1_{2j}}+\sum_{j=1}^3T_{K^1_{3j}}+T_{K^1_4},
\end{align}
where
\begin{align*}
K^1_{-1}(x,y)&:=\int_0^\infty \lambda^2\chi(\lambda)\Big(R_0^+(\lambda^4)vQ_{2}^{0}A_{-1}^1Q_{2}^{0}v[R_0^+-R_0^-](\lambda^4)\Big)(x,y)d\lambda,\\
K^1_{kj}(x,y)&:=\int_0^\infty \lambda^{3+k}\chi(\lambda)\Big(R_0^+(\lambda^4)vB_{kj}^1v[R_0^+-R_0^-](\lambda^4)\Big)(x,y)d\lambda,\\
K^1_{4}(x,y)&:=\int_0^\infty \lambda^3\chi(\lambda)\Big(R_0^+(\lambda^4)v\Gamma_4^1(\lambda)v[R_0^+-R_0^-](\lambda^4)\Big)(x,y)d\lambda
\end{align*}
with $k=0,1,2$ and
\begin{itemize}
\item $B_{01}^1=Q_{2}A_{01}^1Q_1$ and $B_{02}^1=Q_1A_{02}^1Q_{2}$;
 \vskip0.1cm
\item $B_{1}^1=Q_1A_{11}^1Q_1$, $B_{12}^1=Q_{2}A_{12}^1$ and $B_{13}^1=A_{13}^1Q_{2}$;
 \vskip0.1cm
\item $B_{21}^1=Q_1A_{21}^1$ and $B_{22}^1=A_{22}^1Q_1$;
 \vskip0.1cm
\item $B_{31}^1=Q_1A_{31}^1$, $B_{32}^1=A_{32}^1Q_1$ and $B_{33}^1=\widetilde P$.
\end{itemize}
For any $B\in \{Q_{2}^{0}A_{-1}^1Q_{2}^{0},B_{kj}^1,\Gamma_4^1\}$ and $k\le8$, $vBv$ is an integral operator satisfying the bounds \eqref{vBv}. As in the regular case, Theorem \ref{theorem_low_1} for the second kind resonant case follows from the following proposition.

\begin{proposition}
\label{proposition_first_1}
Let $1<p<\infty$, $w_p\in A_p$ and $w_1\in A_1$. Then all the integral operators in \eqref{wave_operator_first} satisfy the same bound as \eqref{lemma_2_2_1} and belong to $\mathbb B(\H^1(\R),L^1(\R))\cap \mathbb B(L^\infty(\R),\BMO(\R))$. Moreover, we have:
\begin{itemize}
\item these operators also belong to $\mathbb B(L^\infty(\R),L^{1,\infty}(\R))$ if $V$ is compactly supported;
\vskip0.2cm
\item  $T_{K_{31}^1},T_{K_{32}^1}$ and $T_{K_4^1}$  in fact belong to $\mathbb B(L^1(\R))\cap \mathbb B(L^\infty(\R))$.
\end{itemize}
\end{proposition}

\begin{proof}
The proof is essentially same as that of the regular case, so we only give a brief outline. Recall that we do not  distinguish $Q_{2},Q_{2}^{0}$ and use the same notation $Q_{2}$ to denote them.

At first, $T_{K_{33}^1}=T_{K_{33}^0}$. Moreover, $K_{31}^1$ and $K_{32}^1$ are written in the form \eqref{G} with some $B\in \mathbb B(L^2)$ such that $Q_\alpha B Q_\beta$ is absolutely bounded, and $(\alpha,\beta)=(1,0)$ for $K_{31}^1$ and $(\alpha,\beta)=(0,1)$ for $K_{32}^1$. Hence the proofs for $K_{31}^1$ and $K_{32}^1$ are completely same as that of Proposition \ref{proposition_regular_1}. The proof for $T_{K_4^1}$ is also completely same as that for  $T_{K_4^0}$ since $\Gamma_4^1$ satisfies the same estimates as $\Gamma_4^0$ (see \eqref{lemma_3_3_4}).

Next, for the other cases, precisely for the operators $T_{K^1_{-1}},T_{K^1_{0j}},T_{K^1_{1j}},T_{K^1_{2j}}$, the corresponding kernel is written in the following form:
\begin{align}
\label{proposition_first_1_proof_1}
\int_0^\infty \lambda^{6-\alpha-\beta}\chi(\lambda)\Big(R_0^+(\lambda^4)vQ_\alpha BQ_\beta v[R_0^+-R_0^-](\lambda^4)\Big)(x,y)d\lambda,
\end{align}
where $Q_\alpha BQ_\beta$ is absolutely bounded and
\begin{align}
\label{proposition_first_1_proof_2}
(\alpha,\beta)=\begin{cases}
(2,2)&\text{for}\quad K=K^1_{-1},\\
(2,1)&\text{for}\quad K=K^1_{01},\\
(1,2)&\text{for}\quad K=K^1_{02},\\
(1,1)&\text{for}\quad K=K^1_{11},
\end{cases}\quad
(\alpha,\beta)=\begin{cases}
(2,0)&\text{for}\quad K=K^1_{12},\\
(0,2)&\text{for}\quad K=K^1_{13},\\
(1,0)&\text{for}\quad K=K^1_{21},\\
(0,1)&\text{for}\quad K=K^1_{22}.
\end{cases}
\end{align}
Recall that $X_1=x-\theta_1u_1$, $Y_2=y-\theta_2u_2$, $\Theta=(u_1,u_2,\theta_1,\theta_2)$, $\Theta_j=(u_1,u_2,\theta_j)$. Let $M_{\alpha\beta}(X_1,Y_2,\Theta)$, $M_{\alpha0}(X_1,\Theta_1)$ and $M_{0\beta}(Y_2,\Theta_2)$ be as in \eqref{M_alphabeta},  \eqref{M_alpha0} and \eqref{M_0beta}, respectively. For simplicity, without any confusion, we shall use the same notation $M_{\alpha\beta}(X_1,Y_2,\Theta)$ to denote $M_{\alpha0}(X_1,\Theta_1)$ and $M_{0\beta}(Y_2,\Theta_2)$ by regarding $M_{\alpha0}(X_1,\Theta_1)$ (resp. $M_{0\beta}(Y_2,\Theta_2))$ as a constant function of $y,\theta_2$ (resp. $x,\theta_1$).
Let $G^1_{\alpha\beta}(x,y)$ be the function given by \eqref{proposition_first_1_proof_1}. Using $f_{\alpha\beta}$ defined in Lemma \ref{lemma_FF}, we have
\begin{align}
\label{proposition_first_1_proof_3}
G^1_{\alpha\beta}(x,y)=\int_0^\infty \chi(\lambda)\left(\int_{\R^2\times[0,1]^2} M_{\alpha\beta}(X_1,Y_2,\Theta)f_{\alpha\beta}(\lambda,X_1,Y_2)d\Theta\right)d\lambda.
\end{align}
We consider the two cases (i) $(\alpha,\beta)\neq(1,1)$ and (ii) $(\alpha,\beta)=(1,1)$, separately. It will be seen that the proof for the case (i) (resp. (ii)) is similar to that for $T_{K^0_{33}}$ (resp. $T_{K^0_1}$).
\vskip0.2cm
\underline{{\it Case (i)}}. Let $(\alpha,\beta)\neq(1,1)$. Then the same argument as in the proof of  Proposition \ref{proposition_regular_3} based on Lemma \ref{lemma_FF} yields that $G^1_{\alpha\beta}(x,y)$ is of the form
$$
C_{\alpha\beta}(x,y)\Big(a_{\alpha\beta}\left(k_1^++(-1)^\beta k_1^-\right)+b_{\alpha\beta}\left(k_2^++(-1)^\beta k_2^-\right)\Big)(x,y)+O(\<|x|-|y|\>^{-2}),
$$
where $k^\pm_j$ are defined in Lemma \ref{lemma_2_3}, $a_{\alpha\beta}=i^{\alpha+\beta+1}$, $b_{\alpha\beta}=(-1)^{\alpha+\beta}i^{\beta+1}$ and
$$
C_{\alpha\beta}(x,y)=\int_{\R^2\times[0,1]^2}M_{\alpha\beta}(X_1,Y_2,\Theta)d\Theta.
$$
An important feature is that only one of $\sgn X_1$ or $\sgn Y_2$ appears in the integrand of $C_{\alpha\beta}$ since one of $\alpha,\beta$ is even in \eqref{proposition_first_1_proof_2}, except for $(\alpha,\beta)=(1,1)$. In particular, $C_{\alpha\beta}(x,y)$ is of the form $C^1_\alpha(x)C^2_\beta(y)$ with some bounded functions $C^1_\alpha$ and $C^2_\beta$ (see \eqref{M_alphabeta},  \eqref{M_alpha0} and \eqref{M_0beta} and recall the convention $(\sgn x)^2=1$). Hence, $T_{G^1_{\alpha\beta}}$ is a sum of the composition $C^1_\alpha T_{g_{\alpha\beta}}C^2_\beta$ and an error term satisfying the condition of Lemma \ref{lemma_2_1}, where
\begin{align}
\label{proposition_first_1_proof_4}
g_{\alpha\beta}=a_{\alpha\beta}\left(k_1^++(-1)^\beta k_1^-\right)+b_{\alpha\beta}\left(k_2^++(-1)^\beta k_2^-\right).
\end{align}
Since the multiplication operator by bounded function is bounded on $L^p(w_p)$ for any $1\le p<\infty$ and on $L^{1,\infty}(w_1)$, we can apply Lemma \ref{lemma_2_3} to obtain that the same bounds as \eqref{lemma_2_2_1} holds for $T_{G^1_{\alpha\beta}},T_{G^1_{\alpha\beta}}^*$ and hence for all $T_{K^1_{-1}},T_{K^1_{0j}},T_{K^1_{1j}},T_{K^1_{2j}}$ and their adjoint operators.

To obtain $T_{G^1_{\alpha\beta}}\in \mathbb B(\H^1(\R),L^1(\R))\cap \mathbb B(L^\infty(\R),\BMO(\R))$, we use the same trick as in Proposition \ref{proposition_regular_4} based on the translation invariance of the $L^1$, $\H^1$  and $\BMO$-norms to reduce the proof to that of $T_{g_{\alpha\beta}}\in \mathbb B(\H^1(\R),L^1(\R))\cap \mathbb B(L^\infty(\R),\BMO(\R))$, where
\begin{align}
\label{proposition_first_1_proof_5}
\widetilde g_{\alpha\beta}(x,y)=(\sgn x)^\alpha (\sgn y)^\beta\int_0^\infty \chi(\lambda) f_{\alpha\beta}(\lambda,x,y)d\lambda.
\end{align}
Namely, we have for $\mathcal Y=L^1(\R)$ and $\BMO(\R)$,
$$
\norm{T_{G^1_{\alpha\beta}}f}_{\mathcal Y}\lesssim \norm{T_{g_{\alpha\beta}}f}_{\mathcal Y}.
$$
By Lemma \ref{lemma_3_4} and integration by parts, we find that
$$
\widetilde g_{\alpha\beta}=(\sgn x)^\alpha g_{\alpha\beta}(x,y)(\sgn y)^\beta+O(\<|x|-|y|\>^{-2}).
$$
For $(\alpha,\beta)$ in \eqref{proposition_first_1_proof_2} and $(\alpha,\beta)\neq(1,1)$, $\widetilde g_{\alpha\beta}$ coincides with one of $g^+_1,g_2^-$ and $g_3^+$ with some $a,b\in \C$ given in Lemma \ref{lemma_2_4} (recall that the convention $(\sgn x)^2=1$). Hence Lemma \ref{lemma_2_4} applies to $T_{\widetilde g_{\alpha\beta}}$, obtaining $T_{\widetilde g_{\alpha\beta}}\in \mathbb B(\H^1(\R),L^1(\R))\cap \mathbb B(L^\infty(\R),\BMO(\R))$.
\vskip0.2cm
\underline{{\it Case (ii)}}. The function $K^1_{11}=G^1_{11}$, which is given by \eqref{proposition_first_1_proof_1} with $B=A^1_{11}$ and $\alpha=\beta=1$, essentially coincides with the function $K^0_{1}$ which is given by \eqref{proposition_first_1_proof_1} with $B=A^0_1$ and $\alpha=\beta=1$ (see \eqref{proposition_regular_4_proof_1}). Hence the same proof as that of Proposition \ref{proposition_regular_4} yields that $T_{G^1_{11}}$ satisfies the statement of Proposition \ref{proposition_first_1}.

Summarizing the above two cases (i) and (ii), we conclude that,  for all $(\alpha,\beta)$ in \eqref{proposition_first_1_proof_2}, $T_{G^1_{\alpha\beta}},T_{G^1_{\alpha\beta}}^*$ satisfy the same bound as \eqref{lemma_2_2_1}, as well as \eqref{lemma_2_2_2} if $V$ is compactly supported, and belong to $\mathbb B(\H^1(\R),L^1(\R))\cap \mathbb B(L^\infty(\R),\BMO(\R))$. This completes the proof of the proposition and hence of Theorem \ref{theorem_1} for the first kind resonance case.
\end{proof}

\begin{remark}
\label{remark_first_1}
Note that for any odd integer $\alpha,\beta\ge1$, we also have $T_{\widetilde g_{\alpha\beta}}\in \mathbb B(\H^1(\R),L^1(\R))$ by the same argument as in the case $\alpha=\beta=1$ since, in such a case, we can choose $a,b\in \C$ appropriately so that $\widetilde g_{\alpha\beta}=g_{4,a,b}^-$, where $g_{4,a,b}^-$ is defined in Lemma \ref{lemma_2_4}.
\end{remark}


\subsection{Second kind resonant case}
Finally we consider the case when zero is a second kind resonance of $H$ and $|V(x)|\lesssim \<x\>^{-\mu}$ with some $\mu>29$. In such a case, according to the expansion \eqref{lemma_3_3_3}, $W^{L}_-$ consists of 19 integral operators as
\begin{align}
\label{wave_operator_second}
\nonumber
W^{L}_-&=T_{K^2_{-3}}+\sum_{j=1}^2T_{K^2_{-2j}}+\sum_{j=1}^3T_{K^2_{-1j}}+\sum_{j=1}^4T_{K^2_{0j}}\\
&\quad\quad+\sum_{j=1}^3T_{K^2_{1j}}+\sum_{j=1}^2T_{K^2_{2j}}+\sum_{j=1}^3T_{K^1_{3j}}+T_{K^2_4},
\end{align}
where
\begin{align*}
K^2_{-3}(x,y)&:=\int_0^\infty \chi(\lambda)\Big(R_0^+(\lambda^4)vQ_{3}A_{-3}^2Q_{3}v[R_0^+-R_0^-](\lambda^4)\Big)(x,y)d\lambda,\\
K^2_{kj}(x,y)&:=\int_0^\infty \lambda^{3+k}\chi(\lambda)\Big(R_0^+(\lambda^4)vB_{kj}^2v[R_0^+-R_0^-](\lambda^4)\Big)(x,y)d\lambda,\\
K^2_{4}(x,y)&:=\int_0^\infty \lambda^3\chi(\lambda)\Big(R_0^+(\lambda^4)v\Gamma_4^2(\lambda)v[R_0^+-R_0^-](\lambda^4)\Big)(x,y)d\lambda
\end{align*}
with $k=-2,-1,0,1,2$ and
\begin{itemize}
\item $B_{-21}^2=Q_{3}A_{-21}^2Q_{2}$ and $B_{-22}^2=Q_{2}A_{-22}^2Q_{3}$; \vskip0.1cm

\item $B_{-11}^2=Q_{2}A_{-11}^2Q_{2}$, $B_{-12}^2=Q_{3}A_{-12}^2Q_1$ and $B_{-13}^2=Q_1A_{-13}^2Q_{3}$;\vskip0.1cm
\item $B_{01}^2=Q_{2}A^2_{01}Q_1$, $B_{02}^2=Q_1A^2_{02}Q_{2}$, $B_{03}^2=Q_{3}A^2_{03}$ and $B_{04}^2=A^2_{04}Q_{3}$;\vskip0.1cm
\item $B_{11}^2=Q_1A_{11}^2Q_1$, $B_{12}^2=Q_{2}A^2_{12}$ and $B_{13}^2=A^2_{13}Q_{2}$;\vskip0.1cm
\item $B_{21}^2=Q_1A^2_{21}$ and $B_{22}^2=A^2_{22}Q_1$;\vskip0.1cm
\item $B_{31}^2=Q_1A_{31}^2$, $B_{32}^2=A_{32}^2Q_1$ and $B_{33}^2=\widetilde P$.
\end{itemize}
As in the previous two cases, Theorem \ref{theorem_low_1} for the second kind resonant case then follows from the following proposition:

\begin{proposition}
\label{proposition_second_1}
Let $1<p<\infty$, $w_p\in A_p$ and $w_1\in A_1$. Then all the integral operators in \eqref{wave_operator_first} satisfy the same bound as \eqref{lemma_2_2_1} and belong to $\mathbb B(\H^1(\R),L^1(\R))$. If in addition $V$ is compactly supported, then they also satisfy the same bound as \eqref{lemma_2_2_2}. Moreover, we have:
\begin{itemize}
\item except for $T_{K^2_{-3}}$ and $T_{K^2_{-12}}$, these operators belong to $\mathbb B(L^\infty(\R),\BMO(\R))$;\vskip0.2cm
\item $T_{K^2_{31}},T_{K^2_{32}}$ and $T_{K^2_4}$ in fact belong to $\mathbb B(L^1(\R))\cap \mathbb B(L^\infty(\R))$. \end{itemize}
\end{proposition}

\begin{proof}
The proof is similar to that of the previous two cases. Indeed, $T_{K^2_{33}}$ is equal to $T_{K^0_{33}}$. The proof for $T_{K^2_{31}},T_{K^2_{32}},T_{K^2_4}$ is same as that for $T_{K^0_{31}},T_{K^0_{32}},T_{K^0_4}$, respectively.

All the other operators in \eqref{wave_operator_second} can be written in the form \eqref{proposition_first_1_proof_3} with $(\alpha,\beta)$ given by
\begin{align*}
(\alpha,\beta)=\begin{cases}
(3,3)&\text{for}\quad K=K^2_{-3},\\
(3,2)&\text{for}\quad K=K^2_{-21},\\
(2,3)&\text{for}\quad K=K^2_{-22},\\
(2,2)&\text{for}\quad K=K^3_{-11},\\
(3,1)&\text{for}\quad K=K^2_{-12},\\
(1,3)&\text{for}\quad K=K^2_{-13},\\
(2,1)&\text{for}\quad K=K^3_{01},\\
(1,2)&\text{for}\quad K=K^2_{02},
\end{cases}\quad\quad
(\alpha,\beta)=\begin{cases}
(3,0)&\text{for}\quad K=K^2_{03}\\
(0,3)&\text{for}\quad K=K^2_{04},\\
(1,1)&\text{for}\quad K=K^2_{11},\\
(2,0)&\text{for}\quad K=K^2_{12},\\
(0,2)&\text{for}\quad K=K^2_{13},\\
(1,0)&\text{for}\quad K=K^3_{21},\\
(0,1)&\text{for}\quad K=K^2_{22}.\end{cases}
\end{align*}

We consider the following three cases separately: (i) one of $\alpha,\beta$ is even, (ii) $(\alpha,\beta)=(1,1)$, $(1,3)$, and (iii) $(\alpha,\beta)=(3,1)$, $(3,3)$.

\underline{{\it Case (i)}}. If in addition that one of $\alpha,\beta$ is even, then the same argument as that in the case (i) of the proof for the first kind resonant case yields that these operators satisfy the same bounds as \eqref{lemma_2_2_1}, as well as the $\H^1$-$L^1$ and $L^\infty$-$\BMO$ boundedness.

\underline{{\it Case (ii)}}. If  $(\alpha,\beta)=(1,1)$, $(1,3)$, the completely same argument as that in the proof for $T_{K^0_{1}}$ works. Indeed, for $(\alpha,\beta)=(1,1)$, $K_{11}^2$ can be obtained by replacing $A_1^0$ in the formula of $K^0_{1}$ (see \eqref{proposition_regular_4_proof_1}) by $A_{11}^2$. Moreover, for $(\alpha,\beta)=(1,3)$, $K_{-13}^2$ is given by
\begin{align*}
K_{-13}^2(x,y)
&=\int_0^\infty \chi(\lambda)\left(\int_{\R\times[0,1]} M_{13}(X_1,Y_2,\Theta)f_{13}(\lambda,X_1,Y_2)d\Theta\right)d\lambda\\
&=-\int_0^\infty \chi(\lambda)\left(\int_{\R\times[0,1]} M_{13}(X_1,Y_2,\Theta)f_{11}(\lambda,X_1,Y_2)d\Theta\right)d\lambda,
\end{align*}
where, with some constant $c_{13}>0$,
$$
M_{13}(X_1,Y_2,\Theta)=c_{13}(\sgn X_1)(\sgn Y_2)(1-\theta_2)^2u_1u_2^3(vQ_1A_{-13}^2Q_3v)(u_1,u_2).
$$
Applying the same argument as in the Step 1 of  Proposition \ref{proposition_regular_4}, we can write
$$
K_{-13}^2(x,y)=\int_{\R\times [0,1]}\Big(\sgn(X_1)g_1^-(x,y) m_{13}(y,u_1,\theta_1)+O\left(\<|x|-|y|\>^{-2}\rho_8 (u_1)\right)\Big)du_1d\theta_1
$$
with $m_{13}(y,u_1,\theta_1)=\int_{\R\times[0,1]}\frac{M_{13}(X_1,Y_2,\Theta)}{\sgn X_1}du_2d\theta_2$. Hence, the same argument as that in Proposition \ref{proposition_regular_4} also applies to $T_{K_{-13}}^2$.

\underline{{\it Case (iii)}}. Let $(\alpha,\beta)=(3,1)$ or $(3,3)$, namely $K=K_{-12}^2$  or $K_{-3}^2$.  In this case,  an almost same argument as for $T_{K^0_{1}}$ still works, except for the part of the boundedness from $L^\infty$ to $\BMO$. If we rewrite \eqref{proposition_first_1_proof_3} as
$$
\int_{\R\times[0,1]}G ^1_{\alpha\beta}(x,y;u_1,\theta_1)du_1d\theta_1
$$
with
$$
G ^1_{\alpha\beta}(x,y;u_1,\theta_1)=\int_0^\infty \chi(\lambda)\left(\int_{\R\times[0,1]} M_{\alpha\beta}(X_1,Y_2,\Theta)f_{\alpha\beta}(\lambda,X_1,Y_2)d\Theta\right)d\lambda,
$$
then we find by the same argument as in Proposition \ref{proposition_regular_4} that
$$
G ^1_{\alpha\beta}(x,y;u_1,\theta_1)=\sgn(X_1)g_{\alpha\beta}(x,y)m_{\alpha\beta}(y,\theta_1,u_1)+O(\<|x|-|y|\>^{-2}\rho_8(u_1)),
$$
where $g_{\alpha\beta}$ is given by \eqref{proposition_first_1_proof_4} and
$$
m_{\alpha\beta}(y,\theta_1,u_1)=\int_{\R\times[0,1]}\frac{M_{\alpha \beta}(X_1,Y_2,\Theta)}{\sgn X_1}du_2d\theta_2=O(\rho_6(u_1)).
$$
Since $T_{g_{\alpha\beta}}\in \mathbb B(L^p(w_p))\cap \mathbb B(L^1(\R),L^{1,\infty}(\R))$ for any $\alpha,\beta$, as in the Steps 2 and 3 in the proof of Proposition \ref{proposition_regular_4}, we obtain $T_{K_{-12}^2},T_{K_{-3}^2}\in \mathbb B(L^p(w_p))$ for $1<p<\infty$, as well as $T_{K_{-12}^2},T_{K_{-3}^2}\in \mathbb B(L^1(\R),L^{1,\infty}(\R))$ if $V$ is compactly supported.

As in Proposition \ref{proposition_regular_4}, the $\H^1$-$L^1$ boundedness is deduced from the bound
$$
\norm{T_{k_{\alpha\beta}}f}_{L^1}\lesssim \norm{f}_{\H^1},
$$
with $k_{31}=g_{31}\sgn x\sgn y=g_{4,i,-1}^-$ and $k_{33}=g_{33}\sgn x\sgn y=g_{4,-i,1}^-$. Hence, Applying Lemma \ref{lemma_2_4}, we obtain $T_{K_{-12}^2},T_{K_{-3}^2}\in \mathbb B(\H^1,L^1)$.
\end{proof}


Putting Propositions \ref{proposition_regular_1}--\ref{proposition_regular_4}, \ref{proposition_first_1} and \ref{proposition_second_1} all together, we have finished the proof of Theorem \ref{theorem_low_1}.

\section{High energy estimate}
\label{section_high}
Here we give the proof of the high energy part of Theorem \ref{theorem_1}, that is, the following theorem. Recall that the high energy part $W^H_-$ of the wave operator was given by \eqref{W^H_-}.

\begin{theorem}
\label{theorem_high_1}
Suppose that $|V(x)|\lesssim \<x\>^{-\mu}$ with $\mu>3$ and  $H$ has no embedded eigenvalues. Then $W^H_-$ is bounded on $L^p$ for any $1\le p\le \infty$. Moreover, for any $1<p<\infty$ and $w_p\in A_p$ and $w_1\in A_1$, $W^H_-$ and $(W^H_-)^*$ satisfy the same bounds as \eqref{theorem_1_1} and \eqref{theorem_1_2}
 \end{theorem}

The proof of this theorem consists of two parts. Using the resolvent equation
$$
R^+_V(\lambda^4)=R_0^+(\lambda^4)-R^+_0(\lambda^4)VR_V^+(\lambda^4),
$$
we write $W^H_-=W^H_1-W^H_2$, where $\widetilde \chi=1-\chi$ and
\begin{align*}
W^H_1&=\int_0^\infty \lambda^3\widetilde \chi(\lambda)R_0^+(\lambda^4)V[R_0^+-R_0^-](\lambda^4)d\lambda,\\
W^H_2&=\int_0^\infty \lambda^3\widetilde \chi(\lambda)R^+_0(\lambda^4)VR_V^+(\lambda^4)V[R_0^+-R_0^-](\lambda^4)d\lambda.
\end{align*}

By virtue of Lemmas \ref{lemma_2_1} and \ref{lemma_2_2}, Theorem \ref{theorem_high_1} follows from the following Propositions \ref{proposition_high_2} and \ref{proposition_high_3}.

\begin{proposition}
\label{proposition_high_2}
Suppose $|V(x)|\lesssim \<x\>^{-\mu}$ with $\mu>3$. Then the integral kernel $K^H_1(x,y)$ of $W^H_1$ satisfies $|K^H_1(x,y)|\lesssim \<|x|-|y|\>^{-2}$ on $\R^2$.
\end{proposition}

\begin{proof}
By the formula \eqref{free_resolvent} and the same argument as in the proof of Proposition \ref{proposition_regular_1}, $K^H_1(x,y)$ can be written the form
\begin{align*}
&K_1^H(x,y)\\
&=\frac{1}{16}\int_0^\infty \int\lambda^{-3}\widetilde \chi(\lambda) F_+(\lambda|x-u|) V(u) [F_+-F_-](\lambda|y-u|)dud\lambda\\
&=\frac{1}{16}\int_0^\infty\lambda^{-3}\widetilde \chi(\lambda) \left(\int_{\R}V(u)f_{00}(\lambda,x-u,y-u)du\right)d\lambda\\
&=\sum_\pm \left(\int_0^\infty e^{i\lambda(|x|\pm|y|)} \lambda^{-3}\widetilde \chi(\lambda)A^\pm_1(\lambda,x,y)d\lambda+\int_0^\infty e^{-\lambda(|x|\pm i|y|)} \lambda^{-3}\widetilde \chi(\lambda)A^\pm_2(\lambda,x,y)d\lambda\right),
\end{align*}
where $f_{00}$ is defined in Lemma \ref{lemma_FF} and $A^\pm_j$ satisfy, for all $x,y\in \R$, $\lambda\ge1$ and $\ell=0,1,2$,
\begin{align*}
|\partial_\lambda^\ell A^\pm_1(\lambda,x,y)|+e^{-\lambda|x|}|\partial_\lambda^\ell A^\pm_2(\lambda,x,y)|\lesssim \norm{\<x\>^\ell V}_{L^1}<\infty.
\end{align*}
Therefore, the same argument as in the low energy case based on integration by parts implies $|K^H_1(x,y)|\lesssim \<|x|-|y|\>^{-2}$. This completes the proof.
\end{proof}

\begin{proposition}
\label{proposition_high_3}
Under the assumption in Theorem \ref{theorem_high_1}, the integral kernel $K^H_2(x,y)$ of $W^H_2$ satisfies $|K^H_2(x,y)|\lesssim \<|x|-|y|\>^{-2}$ on $\R^2$.
\end{proposition}

In  the proof of this proposition, we need the following high energy resolvent estimate:

\begin{lemma}[{\cite[Theorem 2.23]{FSY18}}]
\label{lemma_high_4}
Suppose that  $|V(x)|\lesssim \<x\>^{-\mu}$ with $\mu>1$ and $H$ has no embedded eigenvalues. Then, for any integer $0\le \ell<\mu$ and $\ep>0$, the map $(0,\infty)\ni\lambda \mapsto \<x\>^{-\sigma}R_V^\pm(\lambda^4)\<x\>^{-\sigma}$ is of $C^\ell$-class in the norm topology on $L^2$ and satisfies
$$
\norm{\<x\>^{-1/2-\ell-\ep}\partial_\lambda^\ell \left\{R_V^\pm(\lambda^4)\right\}\<x\>^{-1/2-\ell-\ep}}_{L^2\to L^2}\le C_\ell\<\lambda\>^{-3},\quad \lambda\ge \lambda_0.
$$
\end{lemma}

\begin{proof}[Proof of Proposition \ref{proposition_high_3}]
As before, $W^H_2$ is given by an integral operator with the kernel
\begin{align*}
K_2^H(x,y)
&=\frac{1}{16}\int_0^\infty \int_{\R^2}\frac{\widetilde \chi(\lambda)}{\lambda^3} \Gamma^H(\lambda,u_1,u_2) F_+(\lambda|x-u_1|) [F_+-F_-](\lambda|y-u_2|)du_1du_2d\lambda\\
&=\frac{1}{16}\int_0^\infty \int_{\R^2}\frac{\widetilde \chi(\lambda)}{\lambda^3} \Gamma^H(\lambda,u_1,u_2) f_{00}(\lambda,x-u_1,y-u_2)du_1du_2d\lambda
\end{align*}
where $\Gamma^H(\lambda,u_1,u_2)=(VR_V^+(\lambda^4)V)(u_1,u_2)$. Note that Lemma \ref{lemma_high_4} and H\"older's inequality imply that, for any $\ell=0,1,2$, any $k$ satisfying $\ell+k\le2$ and small $\ep>0$ with $3+\ep<\mu$,
\begin{align*}
&\norm{\<x\>^kV\partial_\lambda^\ell R_V^\pm(\lambda^4)V\<x\>^kf}_{L^1}\\
&\lesssim \norm{\<x\>^{1/2+\ell+k+\ep}V}_{L^2}^2\norm{\<x\>^{-1/2-\ell-\ep}\partial_\lambda^\ell R_V^\pm(\lambda^4)\<x\>^{-1/2-\ell-\ep}}_{L^2\to L^2}\norm{f}_{L^\infty}\\
&\lesssim \<\lambda\>^{-3}\norm{\<x\>^{1/2+\ell+k+\ep}V}_{L^2}^2\norm{f}_{L^\infty}
\end{align*}
uniformly in $\lambda\ge \lambda_0$. Hence $\Gamma^H(\lambda, u_1,u_2)$ satisfies
\begin{align}
\label{proposition_high_4_proof_1}
\int_{\R^2}\<u_1\>^k|\partial_\lambda^\ell\Gamma^H(\lambda,u_1,u_2)|\<u_2\>^kdu_1du_2\lesssim \<\lambda\>^{-3}\norm{\<x\>^{1/2+\ell+k+\ep}V}_{L^2}^2
\end{align}
for $\ell=0,1,2$ and $\lambda\ge\lambda_0$. With this estimate at hand, we can see that the rest of the proof is essentially same as that of Proposition \ref{proposition_regular_2}. Indeed, setting
\begin{align*}
B^\pm_j(\lambda,x,y)&=\int_{\R^2}e^{\lambda \Phi^\pm_1(x,y,u_1,u_2,0,0)}\Gamma^H(\lambda, u_1,u_2)du_1du_2,
\end{align*}
where $\Phi^\pm_j$ are defined by \ref{proposition_regular_1_proof_3}, we have that $K_2^H$ is a linear combination of
\begin{align}
\label{proposition_high_4_proof_2}
\int_0^\infty e^{i\lambda(|x|\pm|y|)}\lambda^{-3}\widetilde \chi(\lambda)B^\pm_1(\lambda,x,y)d\lambda,\quad
\int_0^\infty e^{-\lambda(|x|\pm i|y|)}\lambda^{-3}\widetilde \chi(\lambda)B^\pm_2(\lambda,x,y)d\lambda.
\end{align}
Moreover, \eqref{proposition_regular_1_proof_6} and \eqref{proposition_high_4_proof_1} imply that for $\ell=0,1,2$,
\begin{align*}
&|\partial_\lambda^\ell B^\pm_1(\lambda,x,y)|+e^{-\lambda|x|}|\partial_\lambda^\ell B^\pm_2(\lambda,x,y)|\\
&\lesssim\sum_{k+\ell'=\ell}\int_{\R^2}\<u_1\>^{k}\<u_2\>^{k}|\partial_\lambda^{\ell'}\Gamma^H(\lambda, u_1,u_2)|du_1du_2\\
&\lesssim \lambda^{-3}\norm{\<x\>^{5/2+\ep}V}_{L^2}^2.
\end{align*}
Hence, since $\widetilde \chi(0)=0$, we obtain by integrating by parts twice that all the 4 integrals in \eqref{proposition_high_4_proof_2} are $O(\<|x|-|y|\>^{-2})$. This proves the desired assertion.
\end{proof}

This completes the proof of Theorem \ref{theorem_high_1}. By virtue of \eqref{wave_operator} and
Theorem \ref{theorem_low_1}, this also completes the proof of Theorem \ref{theorem_1} for $W_-$. As mentioned in Section \ref{subsection_stationary}, this also gives the desired results for $W_+$ since $W_+f=\overline{W_-\overline  f}$. We thus have finished the proof of Theorem \ref{theorem_1}.

\section{Counterexample for endpoint estimates}
\label{section_counterexample}

Here we prove Theorem \ref{theorem_2}. Throughout the paper, we assume that $H$ has no embedded eigenvalue in $(0,\infty)$.

\subsection{Counterexample for the $L^1$ and $L^\infty$ boundedness} In this subsection, we suppose that zero is a regular point of $H$ and prove Theorem \ref{theorem_2} (1).

Before staring the proof, we explain briefly its strategy. To disprove the $L^1$- and $L^\infty$-boundedness, we first observe by Propositions \ref{proposition_regular_1} and \ref{proposition_regular_2} that all the terms appeared in the right hand side of \eqref{wave_operator_regular}, except for the two terms $T_{K^0_1}$ and $T_{K^0_{33}}$, are bounded on $L^1$ and $L^\infty$. Hence, we need to deal with $T_{K^0_1}$ and $T_{K^0_{33}}$. We then shall show that, for a test function $$f_R=\chi_{[-R,R]},$$
\begin{itemize}
\item[(a)] $\norm{T_{K^0_{33}}f_R}_{L^\infty(\R)}$ is not bounded in $R\gg1$ and $T_{K^0_{33}}f_1\notin L^1(\R)$, but

\vskip0.2cm
\item[(b)] $\norm{T_{K^0_{1}}f_R}_{L^\infty(\R)}$ is bounded in $R>0$ and $T_{K^0_{1}}f_1\in L^1(\R)$.
\end{itemize}
These properties (a) and (b) will be shown in Propositions \ref{proposition_example_1} and \ref{proposition_example_2}, respectively.

\vskip0.2cm
We begin with the statement (a):

\begin{proposition}
\label{proposition_example_1}
Let $f_R=\chi_{[-R,R]}$. Then $|(T_{K^0_{33}}f_R)(R+2)|\to \infty$ as $R\to \infty$. Moreover, $T_{K^0_{33}}f_1\notin L^1(\R)$. In particular, $T_{K^0_{33}}$ is neither bounded on $L^\infty(\R)$ nor on $L^1(\R)$.
\end{proposition}

\begin{proof}
Recall that $K^0_{33}=\frac{-1+i}{8}g_1^++O(\<|x|-|y|\>^{-2})$ (see \eqref{proposition_regular_3_proof_1} and Lemma \ref{lemma_2_4}). We compute
\begin{align*}
&g_1^+(x,y)
=\chi_{\{||x|-|y||\ge2\}}g_1^+(x,y)+\chi_{\{||x|-|y||\le2\}}g_1^+(x,y)\\
&=\chi_{\{||x|-|y||\ge2\}}\left(\frac{1}{|x|+|y|}+\frac{1}{|x|-|y|}+\frac{1}{|x|+i|y|}+\frac{1}{|x|-i|y|}\right)+O(\<|x|-|y|\>^{-2})\\
&=\chi_{\{||x|-|y||\ge2\}}\left(\frac{1}{|x|+|y|}+\frac{1}{|x|-|y|}+\frac{2|x|}{x^2+y^2}\right)+O(\<|x|-|y|\>^{-2}),
\end{align*}
where we have used the property $\psi(||x|\pm|y||^2)=1$ for $||x|-|y||\ge2$. Note that
$$
\sup_x\int \frac{|x|}{x^2+y^2}|f_R(y)|dy\le \pi\norm{f_R}_{L^\infty(\R)}\le \pi.
$$
Hence, by Lemma \ref{lemma_2_1}, there exists constants $c_0,c_1>0$ such that
$$
|(T_{K^0_{33}}f_R)(x)|\ge c_0\left|\int_{-R}^R\left(\frac{\chi_{\{||x|-|y||\ge2\}}}{|x|+|y|}+\frac{\chi_{\{||x|-|y||\ge2\}}}{|x|-|y|}\right)dy\right|-c_1.
$$
We thus have $|(T_{K^0_{33}}f_R)(R+2)|\to \infty$ as $R\to \infty$ since
\begin{align*}
&\int_{-R}^R\left(\frac{\chi_{\{|R+2-|y||\ge2\}}}{R+2+|y|}+\frac{\chi_{\{|R+2-|y||\ge2\}}}{R+2-|y|}\right)dy
=2\int_0^R\left(\frac{1}{R+2+y}+\frac{1}{R+2-y}\right)dy\\
&=\left.2\log\frac{R+2+y}{R+2-y}\right|_{0}^R=2\log(R+1).
\end{align*}

We next prove $T_{K^0_{33}}f_R\notin L^1(\R)$. Since $g_1^+(x,y)$ is continuous on $\R^2$,
$$
\int_{-R-2}^{R+2}\int_{-R}^R|g_1^+(x,y)|dxdy<\infty.
$$
On the other hand, by the same computation as above, we have
\begin{align*}
&\int_{R+2\le |x|\le R'}\left|\int_{-R}^R\left(\frac{\chi_{\{||x|-|y||\ge2\}}}{|x|+|y|}+\frac{\chi_{\{|x|-|y||\ge2\}}}{|x|-|y|}+\frac{2\chi_{\{||x|-|y||\ge2\}}|x|}{x^2+y^2}\right)dy\right|dx\\
&=4\int_{R+2}^{R'}\int_0^R\left(\frac{1}{x+y}+\frac{1}{x-y}+\frac{2x}{x^2+y^2}\right)dydx\\
&\ge 4\int_{R+2}^{R'}\log\frac{x+R}{x-R}dx\gtrsim \log R'\to \infty
\end{align*}
as $R'\to \infty$. Hence $T_{K^0_{33}}f_R\notin L^1(\R)$.
\end{proof}

We next prove the item (b) for the operator $T_{K^0_{1}}$:
\begin{proposition}
\label{proposition_example_2}
Let $f_R=\chi_{[-R,R]}$. Then
$
\sup_{R>0}\norm{T_{K^0_1}f_R}_{L^\infty(\R)}<\infty
$
 and $T_{K^0_1}f_1\in L^1(\R)$.
\end{proposition}

\begin{proof}
It follows from \eqref{proposition_regular_4_proof_6} and \eqref{proposition_regular_4_proof_7} that
$$
K^0_1(x,y)=\int_{\R^2\times[0,1]^2}M_{11}(u_1,u_2)g_4^-(X_1,Y_2)d\Theta+e(x,y)
$$
where $X_1=x-\theta_1u_1$, $Y_2=y-\theta_2u_2$, $\Theta=(u_1,u_2,\theta_1,\theta_2)$, $g_4^-$ is given by Lemma \ref{lemma_2_4} (with the choice of $a=i$, $b=1$) and
$
M_{11}(u_1,u_2)=\frac{1}{16}u_1u_2(vQ_1A^0_1Q_1v)(u_1,u_2)
$
satisfies $$\norm{\<u_1\>^kM_{11}(u_1,u_2)\<u_2\>^{k}}_{L^1(\R^2)}\lesssim \norm{\<x\>^{2+2k}V}_{L^1},\quad k\le 6.$$
Moreover, $e(x,y)$ is the error term satisfying
$$
|e(x,y)|\lesssim \int_{\R^2\times[0,1]^2}M_{11}(u_1,u_2)\<|X_1|-|Y_2|\>^{-2}d\Theta
$$
It is easy to see that $T_e\in \mathbb B(L^1)\cap \mathbb B(L^\infty)$ by Lemma \ref{lemma_2_1}.
As above, we can write
\begin{align*}
g_4^-(x,y)
&=i\sgn x\left(\frac{\chi_{\{||x|-|y||\ge2\}}}{|x|+|y|}-\frac{\chi_{\{|x|-|y||\ge2\}}}{|x|-|y|}-\frac{2\chi_{\{||x|-|y||\ge2\}}|y|}{x^2+y^2}\right)\sgn y+O(\<|x|-|y|\>^{-2})\\
&=:\widetilde g_4^-(x,y)+O(\<|x|-|y|\>^{-2}).
\end{align*}
Note that $\widetilde g_4$ is bounded on $\R^2$ by the support property of $\chi_{\{||x|-|y||\ge2\}}$.
Define
$$
G(x,y)=\int_{\R^2\times[0,1]^2}M_{11}(u_1,u_2)\widetilde g_4(X_1,Y_2)d\Theta.
$$

Now we shall prove $\norm{T_{K^0_1}f_R}_{L^\infty(\R)}\lesssim1
$ uniformly in $R>0$.
Lemma \ref{lemma_2_1} implies there exists $C>0$ independent of $R$ such that
$$\norm{T_{K^0_1}f_R}_{L^\infty(\R)}\le \norm{T_{G}f_R}_{L^\infty(\R)}+C.
$$
Next, we set $U_R=\{(u_1,u_2)\ |\ |u_1|\ge R/2\ \text{or}\ |u_2|\ge R/2\}$ and decompose
\begin{align*}
G(x,y)
&=\left(\int_{U_R\times[0,1]^2}+\int_{U_R^c\times[0,1]^2}\right)M_{11}(u_1,u_2)\widetilde g_4^-(X_1,Y_2)d\Theta\\
&=:G_1(x,y)+G_2(x,y)
\end{align*}
For the former term $G _1$, since $\widetilde g_4$ is bounded on $\R^2$ and
$$
\norm{M_{11}}_{L^1(\R^2)}\lesssim \norm{\<u_1\>\<u_2\>M_{11}}_{L^1(\R^2)}R^{-1}\lesssim R^{-1},\quad (u_1,u_2)\in U_R,
$$
 we have
$
\norm{T_{G_1}f_R}_{L^\infty(\R)}\lesssim 1
$ uniformly in $R>0$.
To deal with the latter term $G _2$, we observe that the interval $(R-\theta_2u_2,-R-\theta_2u_2)$ contains the origin since $|u_2|\le R/2$ and $\theta_2\in [0,1]$. Hence, since $\widetilde g_4^-$ is an odd function in the $y$-variable (thanks to the term $\sgn y$), we have
$$
\int_{-R}^R\widetilde g_4^-(X_1,Y_2)dy=\int_{-R-\theta_2u_2}^{R-\theta_2u_2}\widetilde g_4^-(X_1,y)dy=\int_{-R-\theta_2u_2}^{-R+\theta_2u_2}\widetilde g_4^-(X_1,y)dy=O(\<u_2\>)
$$
for the case $\theta_2u_2\ge0$,  and
$$
\int_{-R}^R \widetilde g_4^-(X_1,Y_2)dy
=\int_{R-\theta_2u_2}^{R+\theta_2u_2}\widetilde g_4^-(X_1,y)dy=O(\<u_2\>)
$$
for the case $\theta_2u_2\le0$.
Therefore, we obtain uniformly in $R>0$ that
$$
\norm{T_{G _2}f_R}_{L^\infty(\R)}\lesssim \|\<u_2\>M_11\|_{L^1}\lesssim1.
$$

We next prove $T_{K^0_1}f_1\in L^1(\R)$. As above, we have
$$
\norm{T_{K^0_1}f_1}_{L^1(\R)}\le \norm{T_{G}f_1}_{L^1(\R)}+C.
$$
with some $C>0$ by Lemma \ref{lemma_2_1}. Using Fubini's theorem, Minkowski's inequality and the translation invariance of the $L^1$-norm, we compute
\begin{align*}
\norm{T_{G}f_1}_{L^1(\R)}
\le \int_{\R^2\times[0,1]^2}|M_{11}(u_1,u_2)|\left(\int_\R\int_{-1}^1 |\widetilde g_4^-(X_1,Y_2)|dydx\right)d\Theta
\end{align*}
Since $|x^2-(y-\theta_2u_2)^2|\gtrsim (|x|-|y-\theta_2u_2|)^2\gtrsim \<|x|-|y-\theta_2u_2|\>^2$ on $\supp \widetilde g_4^-$, we have
$$
|\widetilde g_4^-(x,Y_2)|\le \frac{4|y-\theta_2u_2|\chi_{\{||x|-|y-\theta_2u_2||\ge2\}}}{|x^2-(y-\theta_2u_2)^2|}\lesssim \<u_2\>\<|x|-|y-\theta_2u_2|\>^{-2}
$$
for $x\in \R,y\in [-1,1]$ and hence, again by the translation invariance of the $L^1$-norm,
$$
\int_{\R^2\times[0,1]^2}|M_{11}(u_1,u_2)|\int_\R\int_{-1}^1 |\widetilde g_4^-(x,Y_2)|dydxd\Theta \lesssim\int_{\R^2\times[0,1]^2}|M_{11}(u_1,u_2)|\<u_2\>d\Theta<\infty.
$$
This shows $T_{K^0_1}f_1\in L^1(\R)$ and completes the proof.
\end{proof}

\begin{proof}[Proof of Theorem \ref{theorem_2}(1)]
We know by Proposition \ref{proposition_regular_1} that all the operators, except for $T_{K^0_1}$ and $T_{K^0_{33}}$, appeared in the right hand side of \eqref{wave_operator_regular} are bounded on $L^1(\R)$ and on $L^\infty(\R)$. By Propositions \ref{proposition_example_1} and \ref{proposition_example_2}, $W_- f_1\notin L^1(\R)$ and there exists $C>0$, independent of $R$, such that
\begin{align*}
\norm{W_- f_R}_{L^\infty(\R)}&\ge |(T_{K^0_{33}}f_R)(R+2)|-C\to \infty,\quad R\to \infty.
\end{align*}
Hence $W_-$ is neither bounded on $L^1(\R)$ nor on $L^\infty(\R)$. 
\end{proof}

\begin{remark}
\label{remark_counterexample_1}
Combining with the idea in Subsection \ref{subsection_counterexample_2} below and the above constructions, one can also obtain some results on the unboundedness of $W_\pm$ in $L^1$ and $L^\infty$ for the resonant cases. Suppose zero is a first kind resonance of $H$ and $V$ is compactly supported. Set
$$
C_{\alpha\beta}^*=\int_{\R^2\times[0,1]^2} M_{\alpha\beta}(\Theta)d\Theta,\quad M_{\alpha\beta}(\Theta)=\frac{M_{\alpha\beta}(x,y,\Theta)}{\sgn x\sgn y},
$$
where $M_{\alpha\beta}(x,y,\Theta)$ is given by \eqref{M_alphabeta}. Then one can show $W_\pm \notin \mathbb B(L^1(\R))$ if
$$
2C_{02}^*+(1+i)(C_{10}^*-C_{12}^*)\neq\frac{1-i}{8}.
$$
Moreover, $W_\pm \notin \mathbb B(L^\infty(\R))$ if
$$
iC_{02}^*+C_{10}^*-C_{12}^*+iC_{20}^*-iC_{22}^*\neq\frac{1-i}{8}.
$$
Similar type counterexamples can be also obtained for the second resonant case. We however do not pursue this issue for simplicity.
\end{remark}

\subsection{Counterexample for the $L^\infty$-$\BMO$ boundedness}
\label{subsection_counterexample_2}
We next prove Theorem \ref{theorem_2} (2), precisely the following Proposition.
\begin{proposition}
\label{proposition_example_3}
Suppose that zero is a second kind resonance of $H$ and $V$ is compactly supported. If $D_*\neq0$, then  $W_\pm \notin \mathbb B(L^\infty(\R),\BMO(\R))$ and $W_\pm^*\notin \mathbb B(\H^1(\R),L^1(\R))$, where
\begin{align}
D_*=\int_{\R^2\times [0,1]^2}\Big(6u_1^3u_2(v Q_3A_{-12}^2Q_1v)(u_1,u_2)-u_1^3u_2^3(v Q_3A_{-3}^2Q_3v)(u_1,u_2)\Big)du_1du_2.
\end{align}
\end{proposition}

\begin{proof}
Let $K=K_{-12}^2+K_{-3}^2$. By virtue of Proposition \ref{proposition_second_1} and the duality , it is enough to show
$
T_K^*\notin \mathbb B(\H^1(\R),L^1(\R))
$. By Lemma \ref{lemma_FF}, we have $f_{33}=-f_{31}$. Hence
\begin{align*}
K(x,y)=\int_0^\infty \chi(\lambda)\left(\int_{\R^2\times[0,1]^2}M(X_1,Y_2,\Theta)f_{31}(\lambda,X_1,Y_2)d\Theta\right)d\lambda,
\end{align*}
where $\varphi_1(u_1,u_2)=(v Q_3A_{-12}^2Q_1v)(u_1,u_2)$ and $\varphi_2(u_1,u_2)=(v Q_3A_{-3}^2Q_3v)(u_1,u_2)$ and
\begin{align*}M(x,y,\Theta)
&=(M_{31}-M_{33})(x,y,\Theta)\\
&=\frac{1}{64}(\sgn x)(\sgn y)(1-\theta_1^2)\Big(2u_1^3u_2\varphi_1(u_1,u_2)-(1-\theta_2^2)\varphi_2(u_1,u_2)\Big)
\end{align*}
The same argument as in the proof of Proposition \ref{proposition_regular_4} then yields that, modulo an error term whose associated integral operator belongs to $\mathbb B(L^\infty(\R))$,
\begin{align*}
K(x,y)&\equiv m(x,y)g_1^-(x,y)\\
&\equiv m(x,y)\chi_{\{||x|-|y||\ge2\}}\sum_\pm\left(\frac{\mp i}{|x|\pm |y|}\pm \frac{1}{|x|\pm i|y|}\right)\\
&=m(x,y)\chi_{\{||x|-|y||\ge2\}}\left(-\frac{i}{|x|+|y|}+\frac{i}{|x|-|y|}-\frac{2i|y|}{x^2+y^2}\right),
\end{align*}
where
$$
m(x,y)=\int_{\R^2\times[0,1]^2}M(X_1,Y_2,\Theta)d\Theta.
$$
The kernel of $T_{K}^*$, denoted by $K^*$, thus is given by
$$
K^*(x,y)\equiv \overline{m(y,x)}\chi_{\{||x|-|y||\ge2\}}\left(\frac{i}{|x|+|y|}+\frac{i}{|x|-|y|}+\frac{2i|x|}{x^2+y^2}\right)
$$
again modulo a harmless term. Now we suppose  $\supp V\subset\{|x|\le R-1\}$ with $R\ge2$ and let $$g_R(x)=\sgn(x)\chi_{\{R\le |x|\le 2R\}}(x)\in \H^1(\R).$$
Here we observe that since $\supp v\subset[-R+1,R-1]$ and $\theta_1,\theta_2\in [0,1]$,
$$
\sgn (X_1)\sgn (Y_2)=\sgn(x-\theta_1u_1)\sgn(y-\theta_2u_2)=\sgn x\sgn y
$$
if $|x|\ge 2R+2$, $|y|\ge R$ and $u_1,u_2\in \supp v$. Hence, if $|x|\ge 2R+2$ and $|y|\ge R$, then
\begin{align*}
m(x,y)
=\sgn x\sgn y\int_{\R^2\times[0,1]^2}\frac{M(x,y,\Theta)}{\sgn x\sgn y}d\Theta=\frac{D_*}{576} \sgn x\sgn y.
\end{align*}
Modulo an integral term, we then have for sufficiently large $|x|\ge 2R+2$
\begin{align*}
(T_{K}^*g_R)(x)
&\equiv \frac{\overline{D_*}}{576} \sgn x\int_{R\le |y|\le 2R}\left(\frac{i}{|x|+|y|}+\frac{i}{|x|-|y|}+\frac{2i|x|}{x^2+y^2}\right)dy\\
&=\frac{i\overline{D_*}}{288}\sgn x\int_{R}^{2R}\left(\frac{1}{|x|+y}+\frac{1}{|x|-y}+\frac{2|x|}{x^2+y^2}\right)dy\\
&=\frac{i\overline{D_*}}{288}\left(\log\frac{1+R/x-2R^2/x^2}{1-R/x-2R^2/x^2}+2\arctan\frac{2R}{x}-2\arctan\frac{R}{x}\right)\\
&=\frac{i\overline{D_*}}{288}\left(Rx^{-1}+Rx^{-1}+4Rx^{-1}-2Rx^{-1}\right)+O(|x|^{-2})\\
&=\frac{i\overline{D_*}}{72}Rx^{-1}+O(|x|^{-2})
\end{align*}
by Taylor's expansion near $x=\infty$. Hence, modulo an integral term,
$$
|(T_{K}^*g_R)(x)|\gtrsim |D_*|R|x|^{-1}.
$$
This shows $T_{K}^*g_R\notin L^1(\R)$ and hence $T_{K}^*\notin \mathbb B(\H^1(\R),L^1(\R))$ as long as $D_*\neq0$.
\end{proof}

\section{Boundedness on Sobolev spaces}
\label{section_Sobolev}
Here we prove Theorem \ref{theorem_3}. We follow the same argument as in Finco--Yajima \cite[Section 7]{Finco_Yajima_II}. Recall that $B^N$ for $N\ge1$ is defined in \eqref{B^N}. For short, we set $B^0=L^\infty$.

\begin{lemma}
\label{lemma_6_2}
Let $1<p<\infty$, $N\in \N\cup\{0\}$, $V\in B^{4N}(\R)$ and $E>0$ be large enough. Then $(\Delta^2+E)^{s/4}(H+E)^{-s/4},(H+E)^{s/4}(\Delta^2+E)^{-s/4}\in \mathbb B(L^p(\R))$ for all $0<s\le 4(N+1)$.
\end{lemma}

\begin{proof}
The proof is decomposed into several steps.

{\it Step 1}. We first prove $(\Delta^2+E)(H+E)^{-1}\in\mathbb B(L^p)$.
Since $H$ is bounded below, there exists $E_0>0$ such that if $E\ge E_0$ then $H+E$ is a positive self-adjoint operator and
$$
(H+E)^{-1}f=\int_0^\infty e^{-tH}e^{-Et}fdt,\quad f\in L^2.
$$
It was proved by Deng--Ding--Yao \cite[Theorem 1.1]{DDY_JFA} that $e^{-tH}$ (initially defined on $L^2$) extends to an analytic semi-group $e^{-zH}$ on $L^1$ with angle $\pi/2$ and its kernel satisfies:
\begin{align}
\label{lemma_6_2_proof_1}
|e^{-tH}(x,y)|\lesssim t^{-1/4}\exp\left(-\frac{c|x-y|^{4/3}}{t^{1/3}}+\omega t\right),\quad t>0
\end{align}
with some constant $c,\omega>0$. In particular, $e^{-tH}\in \mathbb B(L^p(\R))$ for all $1\le p\le \infty,t\ge0$ and
$$
\norm{e^{-tH}}_{L^p\to L^p}\lesssim e^{\omega t}.
$$
In what follows, we always assume $E>\max(E_0,\omega)$. Then, for $f\in L^2\cap L^p$,
$$
\norm{(H+E)^{-1}f}_{L^p}\le \int_0^\infty e^{-Et}\norm{e^{-tH}f}_{L^p}dt\lesssim \int_0^\infty e^{-(E-\omega)t}dt\norm{f}_{L^p}\lesssim |E-\omega|^{-1}\norm{f}_{L^p}.
$$
Hence $(H+E)^{-1}$ extends to a bounded operator on $L^p$. Moreover, we have
$$
\Delta^2(H+E)^{-1}f=(H+E-V-E)(H+E)^{-1}f=1-(V+E)(H+E)^{-1}f
$$
and hence $
\norm{\Delta^2(H+E)^{-1}f}_{L^p}\lesssim(1+\norm{V}_{L^\infty})\norm{f}_{L^p}
$ for all $f\in L^2\cap L^p$.
By the density argument, we thus obtain $(\Delta^2+E)(H+E)^{-1}\in \mathbb B(L^p)$.
\vskip0.2cm
{\it Step 2}. Next we prove $(\Delta^2+E)^{s/4}(H+E)^{-s/4}\in\mathbb B(L^p)$ for $0<s<4$. It follows from \eqref{lemma_6_2_proof_1} that $H+E$ satisfies the generalized gaussian bound:
\begin{align}
\label{lemma_6_2_proof_2}
|e^{-t(H+E)}(x,y)|\lesssim t^{-1/4}\exp\left(-\frac{c|x-y|^{4/3}}{t^{1/3}}\right),\quad t>0,\quad E>\max(E_0,\omega).
\end{align}
With this bound at hand, we can apply the abstract spectral multiplier theorem by Blunck \cite[Theorem1.1 and Remark (b) after Theorem1.1]{Blunck} to $H+E$ obtaining $$\norm{(H+E)^{i\beta}}_{L^p\to L^p}\le C_p\<\beta\>^2,\quad 1<p<\infty,\ \beta\in \R.$$ This $L^p$-bounds allow us to interpolate between the trivial case $s=0$ and the case $s=4$ proved in the above Step 1 by applying Stein's analytic interpolation theorem \cite{Stein_56}, yielding $(\Delta^2+E)^{s/4}(H+E)^{-s/4}\in \mathbb B(L^p)$ for $0<s<4$.
\vskip0.2cm
{\it Step 3}. Next, we prove by induction that $(\Delta^2+E)^{N+1}(H+E)^{-N-1}\in\mathbb B(L^p)$ if $V\in B^{4N}(\R)$. The case $N=0$ holds by Step 1. If $N\ge1$, we find by the resolvent equation that
$$
(H+E)^{-N-1}f=(\Delta^2+E)^{-1}(H+E)^{-N}f-(\Delta^2+E)^{-1}V(H+E)^{-N-1}f,\quad f\in L^2.
$$
We also know that $(H+E)^{-N},V(H+E)^{-N-1}\in \mathbb B(L^p,W^{4N,p})$ by the assumption on $V$, the fact $(H+E)^{-1}\in  \mathbb B(L^p)$ and the induction hypothesis. Moreover, it is well known that $(\Delta^2+E)^{-1}\in \mathbb B(W^{4(N-1),p},W^{4N,p})$. Therefore, it follows for $f\in L^p\cap L^2$ that
\begin{align*}
\norm{(H+E)^{-N-1}f}_{W^{4N,p}}
\lesssim \norm{(H+E)^{-N}f}_{W^{4N,p}}+\norm{V(H+E)^{-N-1}f}_{W^{4N,p}}
\lesssim \norm{f}_{L^p}
\end{align*}
Hence $(\Delta^2+E)^{N+1}(H+E)^{-N-1}\in\mathbb B(L^p,W^{4(N+1),p})$ by the density argument.
\vskip0.2cm
{\it Step 4}. The same interpolation argument as above with $(H+E)^{-1}$ replaced by $(H+E)^{-N-1}$, together with the above Step 3, implies $(\Delta^2+E)^{s/4}(H+E)^{-s/4}\in\mathbb B(L^p)$ for all $0<s<4(N+1)$ if $V\in B^{4N}(\R)$. This completes the proof of $(\Delta^2+E)^{s/4}(H+E)^{-s/4}\in \mathbb B(L^p)$. The proof of $(H+E)^{s/4}(\Delta^2+E)^{-s/4}\in \mathbb B(L^p)$ is analogous, so we omit it.
\end{proof}

\begin{proof}[Proof of Theorem \ref{theorem_3}]
Let $E$ be as in Lemma \ref{lemma_6_2} and $f\in C_0^\infty(\R)$. It follows from Theorem \ref{theorem_1}, Lemma \ref{lemma_6_2}  and the intertwining property $(H+E)^s W_\pm =W_\pm (\Delta^2+E)^s$ that
\begin{align*}
\norm{W_\pm f}_{W^{s,p}}\lesssim \norm{(H+E)^{-s}}_{L^p\to W^{s,p}} \norm{W_\pm (\Delta^2+E)^s f}_{L^p}\lesssim \norm{(\Delta^2+E)^s f}_{L^p}\lesssim \norm{f}_{W^{s,p}}
\end{align*}
Since $(\Delta^2+E)^s W_\pm^* =W_\pm^* (H+E)^s$, it also follows from Theorem \ref{theorem_1} and Lemma \ref{lemma_6_2} that
\begin{align*}
\norm{W_\pm^* f}_{W^{s,p}}\lesssim \norm{W_\pm^* (H+E)^s f}_{L^p}\lesssim \norm{ (H+E)^s f}_{L^p}\lesssim \norm{f}_{W^{s,p}}.
\end{align*}
Then the result follows by the density argument.
\end{proof}

\section{Applications}
\label{section_application}
In this section we consider two types of applications of Theorem \ref{theorem_1}: the $L^p$-$L^q$ decay estimates for the propagator $e^{-itH}P_{\ac}(H)$ and the H\"ormander-type $L^p$-boundedness theorem for the spectral multiplier $f(H)$.

\subsection{$L^p$-$L^q$ decay estimates for the propagator $e^{-itH}$}

\begin{theorem}\label{Lp-Lq estimate}
Let $H=\Delta^2+V$ satisfy the same conditions of Theorem \ref{theorem_1}. Then
\begin{equation}\label{decay estimate}
\norm{e^{-itH}P_{\ac}(H)f}_{L^q(\R)}\lesssim |t|^{-\frac{1}{4}(\frac{1}{p}-\frac{1}{q})}\norm{f}_{L^p(\R)}, \ \ t\neq 0,
\end{equation}
for all $(\frac{1}{p},\frac{1}{q})\in \Box_{\mathrm{ABCD}}\setminus\{\overline{\mathrm{BC}},\overline{\mathrm{DC}}\}$, where $\Box_{\mathrm{ABCD}}$ is the closed quadrangle by the four vertex points (see Figure 1):
$\mathrm{A}=(\frac{1}{2},\frac{1}{2})$, $\mathrm{B}=(1,\frac{1}{3})$, $\mathrm{C}=(1,0)$, $\mathrm{D}=(\frac{2}{3}, 0)$, and $\overline{\mathrm{BC}}$ (resp. $\overline{\mathrm{DC}}$ ) is the closed line segment linked by two points $\mathrm{B}, \mathrm{C}$ (resp. $\mathrm{D}, \mathrm{C}$).
\end{theorem}

\begin{figure}[htbp]
\begin{center}
\setlength{\unitlength}{4cm}
\begin{picture}(3.0, 1.7)
\put(0.7,0.2){\vector(0,1){1.5}}
\put(0.7,0.2){\vector(1,0){1.9}}
\put(2.6,0.1){${1\over p}$}
\put(0.6,1.7){${1\over q}$}
\put(1.2,0.2){\line(0,1){1}}
\put(1.7,0.2){\line(0,1){1}}
\put(0.7,1.2){\line(1,0){1}}
 \put(0.7,0.7){\line(1,0){1}}
 \put(1.72,1.22){F(1,1)}
\put(1.2,0.75){A(${1\over2}$,${1\over2}$)}
\put(1.72,0.5){B$(1,{1\over 3})$}  \put(1.7,0.1){C(1,0)}
\put(1.35,0.1){D$({2\over 3}, 0)$}


\put(1.385,0.173){$\bullet$}
\put(1.68,0.47){$\bullet$}
\put(1.18,0.68){$\bullet$}
\put(1.68,0.173){$\bullet$}
\put(0.6,0.1){O}
\put(1.15,0.1){$1\over
2$}
 \put(0.6,0.65){$1\over 2$}
 \put(0.6,1.2){1}
 \put(1.2,0.7){\line(5,-2){0.50}}
 \put(1.2,0.7){\line(2,-5){0.195}}
{\color{red}\put(1.2,0.70){\line(1,-1){0.50}}}

\end{picture}\\
\end{center}
\caption{The closed quadrangle $\Box_{\mathrm{ABCD}}$}
\end{figure}

\begin{remark}
The vertex point $\mathrm{C}=(1,0)$ is not covered by Theorem \ref{Lp-Lq estimate} above. This actually corresponds to the following endpoint decay estimate:
\begin{equation}\label{endpoint decay estimate}\|e^{-itH}P_{\ac}(H)\|_{L^1-L^\infty}\lesssim |t|^{-\frac{1}{4}}, \ \ t\neq 0,\end{equation}
which was directly proved in Soffer--Wu--Yao \cite{SWY21} by the oscillatory integrals method. Furthermore, by \eqref{endpoint decay estimate} and the $L^2$-$L^2$ estimate of $e^{-itH}$, the interpolation can give
\begin{equation}\label{ partial decay estimate}
\|e^{-itH}P_{\ac}(H)\|_{L^p-L^{p'}}\lesssim |t|^{-\frac{1}{4}(\frac{1}{p}-\frac{1}{p'})}, \ \ t\neq 0,
\end{equation}
for all $1\le p\le 2$, which correspond to the line segment $\overline{\mathrm{AC}}$. Hence except for the endpoint $\mathrm{C}=(1,0)$, it is obvious that Theorem \ref{Lp-Lq estimate} extends the admissible line segment $\overline{\mathrm{AC}}$ (i.e. \eqref{ partial decay estimate}) obtained by Soffer--Wu--Yao \cite{SWY21} to the region $\Box_{\mathrm{ABCD}}\setminus\{\overline{\mathrm{BC}},\overline{\mathrm{DC}}\}$.
\end{remark}
\begin{proof}[Proof of Theorem \ref{Lp-Lq estimate}]
Recall that the $L^p$-$L^q$ estimates for $e^{-it\Delta^2}$ were proved  as a special case by Ding--Yao in \cite[Theorem 2.3]{DY} (also see \cite{BKS}). In particular, for any $(\frac{1}{p},\frac{1}{q})\in \Box_{\mathrm{ABCD}
}\setminus\{\mathrm{B},\mathrm{C}\}$ (see the definition of $\Box_{\mathrm{ABCD}}$ in Theorem \ref{Lp-Lq estimate} above), we have
\begin{equation}\label{free decay estimate}
\|e^{-it\Delta^2}\|_{L^p\to L^q}\lesssim |t|^{-\frac{1}{4}(\frac{1}{p}-\frac{1}{q})}, \ \ t\neq 0.
\end{equation}
Since we have the $L^p$-boundedness of $W_\pm$ and $W_\pm^*$ for all $1<p<\infty$ by Theorem \ref{theorem_1}, the intertwining property $\eqref{intertwining_1}$ and \eqref{free decay estimate} yield
\begin{align}
\|e^{-itH}P_{\ac}(H)\|_{L^p\to L^q}\le \|W_\pm\|_{L^q \to L^q}\ \|e^{-it\Delta^2}\|_{L^p\to L^q}\ \|W_\pm^*\|_{L^p \to L^p}\lesssim|t|^{-\frac{1}{4}(\frac{1}{p}-\frac{1}{q})},
\end{align}
for any $(\frac{1}{p},\frac{1}{q})\in \Box_{\mathrm{ABCD}}\setminus\{\overline{\mathrm{BC}},\overline{ \mathrm{DC}}\}$. Thus the proof is concluded.
\end{proof}

\subsection{H\"ormander-type spectral multiplier $f(H)$}

\begin{theorem}\label{Lp-Lp estimate}
Let $H=\Delta^2+V$ satisfy the same conditions of Theorem \ref{theorem_1}. If a bounded Borel function $f:\mathbb{R}\mapsto \mathbb{C}$ satisfies the so-called H\"ormander condition:
\begin{equation}\label{Hormander_conditions}
\sup_{\delta>0}\|\eta(\cdot)f(\delta\cdot)\|_{H^{s}(\mathbb{R})}\le M<\infty,
\end{equation}
with some $s>1/2$ and $\eta\in  C_0^\infty(\mathbb{R}\setminus0)$.
Then for all $1<p<\infty$ we have
\begin{equation}\label{Hormander multiplier}
\|f(H)\phi\|_{L^p}\lesssim (\norm{f}_{L^\infty}+M)\|\phi\|_{L^p},\ \ \phi\in L^p(\mathbb{R}).
\end{equation}
\end{theorem}
\begin{remark}
It is well known that the following Mikhlin's condition
\begin{equation}\label{Mikhlin conditions}
|f^{(j)}(\lambda)|\le C_j |\lambda|^{-j},\quad j=0,1,\quad \lambda>0,
\end{equation}
implies \eqref{Hormander_conditions} (see e.g. Stein \cite[P. 263]{Stein}).
\end{remark}
\begin{remark}Under the assumptions of  Theorem \ref{theorem_1}, by the scattering theory (see {\it e.g.} H\"ormander \cite[Chap.14]{H2}), the spectrum $\sigma(H)$ consists of finitely many negative eigenvalues $\{\lambda_k\}_{k=1}^{N}$ with finite multiplicity and the absolutely continuous spectrum $\sigma_{ac}(H)=[0, \infty)$. In particular, $H$ does not have neither embedded positive eigenvalues nor singular spectrum.  Hence by the spectral theorem and the intertwining property $\eqref{intertwining_1}$, we can write down
\begin{equation}\label{Formula of multiplier}
f(H)=\sum_{j=1}^Nf(\lambda_j)P_{\lambda_j}+W_\pm f(\Delta^2)W_\pm^*,
\end{equation}
where $P_{\lambda_j}$ is the projection onto the eigenspace $\mathcal{H}_j$ corresponding to the eigenvalue $\lambda_j<0$ and $\dim\mathcal{H}_j<\infty$. By counting the finite multiplicity, without loss of generality, we may assume that $\lambda_1\le \lambda_2\le \cdots\le\lambda_N<0$, $H e_j=\lambda_j e_j$ and $P_{\lambda_j}\phi=\langle\phi,e_j\rangle e_j$ for $j=1,...,N$.
\end{remark}

\begin{proof}[Proof of Theorem \ref{Lp-Lp estimate}] Recall that $W_+,W_+^*\in \mathbb B(L^p(\R))$  for all $1<p<\infty$ by Theorem \ref{theorem_1} (1). Since  $f\in H^s(\R)\subset L^\infty(\R)$ for $s>1/2$, we thus obtain by \eqref{Formula of multiplier} that
\begin{align*}
\norm{f(H)}_{L^p\to L^p}
\lesssim \norm{f}_{L^\infty(\R)}\sum_{j=1}^N\norm{P_{\lambda_j}}_{L^p\to L^p}+\norm{f(\Delta^2)}_{L^p\to L^p}.
\end{align*}

In order to deal with the term $f(\Delta^2)$, we let $\widetilde \eta(\xi)=\eta(\xi^4)$ and $m(\xi)=f(\xi^4)$ so that $\eta(\xi^4)f(\xi^4)=\widetilde\eta(\xi)m(\xi)$ and thus $f(\Delta^2)=m(D)$. By H\"{o}rmander's condition \eqref{Hormander_conditions}, we have
$$
\sup_{\delta>0}\|\widetilde \eta(\cdot)m(\delta\cdot)\|_{H^{s}(\mathbb{R})}\le C_s M<\infty
$$
with some $C_s>0$ independent of $m,M$, which implies by the classical H\"{o}rmander Fourier multiplier theorem (see \cite[P. 263]{Stein} or Grafakos \cite[Theorem 6.2.7]{Grafakos_Classical_III}) that
$$
\|f(\Delta^2)\|_{L^p\to L^p}=\|m(D)\|_{L^p\to L^p}\lesssim M+\norm{f}_{L^\infty},\quad 1<p<\infty.
$$

It remains to show $P_{\lambda_j}\in \mathbb B(L^p(\R))$ for each $1\le j\le N$.  In fact, we just need to show the eigenfunction $e_j(x)$ belongs to $L^p$ for all $1\le p\le\infty$ since
\begin{equation}\label{first term}
\|P_{\lambda_j}\phi\|_{L^p}\le |\langle\phi,e_j\rangle|\|e_j\|_{L^p}\le \|e_j\|_{L^p}\|e_j\|_{L^{p'}}\|\phi\|_{L^p}, \ \ 1\le p\le \infty.
\end{equation}
by H\"older's inequality. Note that $P_{\lambda_j} e_j=\lambda_j e_j$, hence by scattering theory (see e.g. H\"ormander \cite[Theorem 14.5.2]{H2}), we can obtain that $e_j$ is a rapidly decreasing eigenfunction, i.e. \begin{align}
\label{eigen}
\<x\>^\ell \partial_x^k e_j\in L^2(\R)\ \text{for all $\ell\in \N$ and  $0\le k\le 2$}.
\end{align}
In particular, $e_j\in L^\infty(\mathbb{R})$ by Sobolev's embedding. Moreover,
H\"older's inequality implies
$\|e_j\|_{L^1}\lesssim \|\<x\>e_j\|_{L^2}<\infty$. Hence $e_j\in L^p(\R)$ for all $1\le p\le\infty$ by interpolation.  
\end{proof}

\begin{remark}
In fact, $P_j\in \mathbb B(L^p(w))$ for any $w\in A_p$ and $1<p<\infty$. Indeed, since $\<x\>^2 e_j\in L^\infty(\R)$ by \eqref{eigen} and the embedding $H^1(\R)\subset L^\infty(\R)$, the kernel $e_j(x)\overline{e_j(y)}$ of $P_j$ satisfies
$
|e_j(x)\overline{e_j(y)}|\lesssim \<x\>^{-2}\<y\>^{-2}\lesssim \<x-y\>^{-2}
$. Hence $P_j\in \mathbb B(L^p(w))$ by Lemma \ref{lemma_2_2}. Therefore, one can also obtain the $L^p(w_p)$-boundedness of $f(H)$ by the same argument as above and Theorem \ref{theorem_1} (2). Namely, if $1<p<\infty$ and $w\in A_p$ is even then
$$
\norm{f(H)}_{L^p(w)\to L^p(w)}\lesssim \norm{f(\Delta^2)}_{L^p(w)\to L^p(w)}+\norm{f}_{L^\infty(\R)}
$$
as long as $f(\Delta^2)\in \mathbb B(L^p(w))$. For instance, if $f$ satisfies \eqref{Hormander_conditions} with $s=1$, then we have $f(\Delta^2)\in \mathbb B(L^p(w))$ for any $w\in A_p$ and $1<p<\infty$ (see Kurtz \cite{Kurtz}). For further results on the weighted boundedness of the Fourier multiplier, we refer to \cite{FHL} and references therein.
 \end{remark}

\appendix

\section{A quick review of Calder\'on--Zygmund operators}
\label{appendix_CZ}
We here give a brief short review of several mapping properties of Calder\'on--Zygmund operators. We refer to textbooks of Grafakos \cite{Grafakos_Classical_III,Grafakos_Modern_III} for general theory.

\subsection{$A_p$-weight}
Let $w\in L^1_{\loc}(\R^n)$ be positive almost everywhere such that $w^{-1}\in L^1_{\loc}(\R^n)$. Then $w$ is said to be of the {\it Muckenhoupt class $A_p$} if
\begin{align*}
[w]_{A_p}&=\sup_{Q}\left[\left(\frac{1}{|Q|}\int_Qw(x)dx\right)\left(\frac{1}{|Q|}\int_Q w(x)^{-\frac{1}{p-1}}dx\right)^{p-1}\right]<\infty,\quad 1<p<\infty,\\
[w]_{A_1}&=\sup_Q\left[\norm{w^{-1}}_{L^\infty(Q)}\left(\frac{1}{|Q|}\int_Q w(x)dx\right)\right]<\infty,\quad p=1,
\end{align*}
where the supremum is taken over all cubes $Q\subset \R^n$. 

Typical examples of $A_p$-weights on $\R^n$ we have in mind are $|x|^a$ and $\<x\>^a$, which belong to $A_p$ if $-n<a<n(p-1)$ for $1<p<\infty$ and if $-n<a\le 0$ for $p=1$.

\subsection{Calder\'on--Zygmund operator} We say that $K$ is a {\it standard kernel} if $K$ satisfies:

\begin{itemize}
\item $|K(x,y)|\lesssim |x-y|^{-n}$ for $x\neq y$, and
\item there exists $\delta>0$ such that, for $x,y,h\in \R^n$ satisfying $|x-y|\ge 2|h|>0$,
$$
|K(x,y)-K(x+h,y)|+|K(x,y)-K(x,y+h)|\lesssim |h|^\delta |x-y|^{-n-\delta}.
$$
\end{itemize}
It is easy to see that $K$ is a standard kernel if $K\in C^1(\R^{2n}\setminus\{(x,y)\ |\ x=y\})$ and
$$
\partial_x^\alpha\partial_y^\beta K(x,y)=O(|x-y|^{-n-|\alpha|-|\beta|}),\quad |\alpha|+|\beta|\le1.
$$
In particular, $\<x-y\>^{-\rho}$ with $\rho>n$ is a standard kernel.

An $L^2$-bounded integral operator $T_K\in \mathbb B(L^2(\R^n))$ with a standard kernel $K$ is called a {\it Calder\'on--Zygmund operator}. Then we have the following theorem (see \cite[Theorems 4.2.2, 4.2.6 and 4.2.7]{Grafakos_Modern_III} for the item (1) and \cite{Hytonen,LOP} for the item (2), respectively):

\begin{theorem}
\label{theorem_CZ_1}
Let $T_K$ be a Calder\'on--Zygmund operator and $1<p<\infty$. Then:
\begin{itemize}
\item[(1)] $T_K\in\mathbb B(L^p(\R))\cap \mathbb B(L^1(\R),L^{1,\infty}(\R))\cap\mathbb B(\H^1(\R),L^1(\R))\cap \mathbb B(L^\infty(\R),\BMO(\R))$.
\item[(2)] $T_K\in \mathbb B(L^p(w_p))\cap \mathbb B(L^1(w_1),L^{1,\infty}(w_1))$ for all $w\in A_p,w_1\in A_1$. Moreover, one has
\begin{align*}
\norm{T_Kf}_{L^p(w_p)}&\lesssim [w_p]_{A_p}^{\max\{1,1/(p-1)\}}\norm{f}_{L^p(w_p)},\\
\norm{T_Kf}_{L^{1,\infty}(w_1)}&\lesssim [w_1]_{A_1}(1+\log[w_1]_{A_1})\norm{f}_{L^1(w_1)},
\end{align*}
with implicit constants being independent of $w_p,w_1$.
\end{itemize}
\end{theorem}

\section*{Acknowledgments}
H. Mizutani is partially supported by JSPS KAKENHI Grant-in-Aid for Scientific Research  (B) \#JP17H02854 and (C) \#JP21K03325. Z. Wan and X. Yao are partially supported by NSFC grants No.11771165 and 12171182. The authors would like to express their thanks to Professor Avy Soffer for his interests and insightful discussions about the topics of higher-order operators. We also thank Dr. SiSi Huang for very careful reading and helpful comments on the first version of the paper. We finally also thank the anonymous referee for insightful comments.


\end{document}